\documentclass[a4paper,reqno,10pt]{amsart}

\usepackage{preamble}

\renewcommand{\tilde}{\widetilde}
\renewcommand{\hat}{\widehat}

\title{A generic classification of locally free representations of affine GLS algebras}

\author{Calvin Pfeifer}

\address{ 
    Calvin Pfeifer: 
    Center for Quantum Mathematics, 
    Department of Mathematics and Computer Science, 
    University of Southern Denmark, 
    Campusvej 55, DK-5230 Odense M, Denmark
} 

\email{capf@sdu.dk}

\date{\today}

\begin{document}

\begin{abstract}
    Throughout, let $K$ be an algebraically closed field of characteristic 0.
    We provide a generic classification of locally free representations 
    of Geiß-Leclerc-Schröer's algebras $H_K(C,D,\Omega)$
    associated to affine Cartan matrices $C$ 
    with minimal symmetrizer $D$ 
    and acyclic orientation $\Omega$.
    Affine GLS algebras are ``smooth'' degenerations of tame hereditary algebras
    and as such their representation theory is presumably still tractable.
    Indeed, we observe several ``tame'' phenomena of affine GLS algebras even though they are in general representation wild.
    For the GLS algebras of type $\widetilde{\BC}_1$ we achieve a classification of all stable representations.
    For general GLS algebras of affine type, 
    we construct a $1$-parameter family of representations
    stable with respect to the defect.
    Our construction is based on a generalized one-point extension technique.
    This confirms in particular $\tau$-tilted versions of the second Brauer-Thrall Conjecture 
    recently raised by Mousavand and Schroll-Treffinger-Valdivieso 
    for the class of GLS algebras. 
    Finally, we show that generically every locally free $H$-module 
    is isomorphic to a direct sum of $\tau$-rigid modules 
    and modules from our $1$-parameter family. 
    This generalizes Kac's canonical decomposition from the symmetric 
    to the symmetrizable case in affine types and we obtain such a decomposition by ``folding" 
    the canonical decomposition of dimension vectors over path algebras. 
    As a corollary we obtain that affine GLS algebras are $E$-tame in the sense of Derksen-Fei and Asai-Iyama.
\end{abstract}

\maketitle

\tableofcontents

\section{Introduction} \label{sec:introduction}

In a series of articles 
\cite{GLS17i}, \cite{GLS18ii}, \cite{GLS16iii}, \cite{GLS18iv}, \cite{GLS18v} and \cite{GLS20}
Geiß-Leclerc-Schröer (abbr. GLS) develop a theory of finite-dimensional algebras $H = H(C,D,\Omega)$ 
defined by quivers $Q = Q(C,\Omega)$ with relations $I = I(C,D,\Omega)$
associated to generalized Cartan matrices $C$
with symmetrizer $D$ and orientation $\Omega$.
Their algebras are $1$-Iwanaga-Gorenstein 
and arise as ``smooth'' degenerations of hereditary algebras $\tilde{H}$.
Therefore, 
they may be seen in non-commutative analogy with singular projective curves
like Kodaira fibres of elliptic surfaces.
The algebras $H$ are often representation wild even if $\tilde{H}$ is representation finite or tame.
Similarly, 
Kodaira fibres may be vector bundle wild \cite{DG01}
while vector bundles over smooth elliptic curves are very well understood by \cite{Ati57}.
Remarkably,
Bodnarchuk-Drozd-Greuel use in \cite{BDG12} matrix problems to classify stable bundles 
on in general vector bundle wild plane degenerations of elliptic curves 
and obtain that these are brick tame in their sense \cite{BD10}.
On the non-commutative side, the structure of rigid modules over $\tilde{H}$
transfers to $\tau$-rigid modules over $H$ \cite{GLS20}.
In particular,
a GLS algebra $H = H(C,D,\Omega)$ is $\tau$-tilting finite in the sense of \cite{DIJ19}
if and only if $C$ is of finite type.
By the work of Demonet-Iyama-Jasso \cite{DIJ19}
and Brüstle-Smith-Treffinger \cite{BST19}
there are then only finitely many stable $H$-modules
and they can be classified explicitly.

\bigskip

The original aim of Geiß-Leclerc-Schröer was a vast generalization 
of many of the connections between path algebras, 
preprojective algebras, Lie algebras and cluster algebras 
from the symmetric to the symmetrizable case.
Classically, non-simply laced types are modeled by species
which generalize path algebras, are still hereditary 
but require the existence of certain finite field extensions of the ground field $K$.
In contrast, GLS algebras $H = KQ/I$ are defined over any ground field $K$
in particular allow to take the complex numbers $K = \bC$.
This has the advantage that varieties of representations
are complex varieties which is in the favour of Geiß-Leclerc-Schröer's effort 
to construct geometrically defined dual semicanonical bases for coordinate algebras of unipotent cells
generalizing their prestigious work in the simply laced case culminating in \cite{GLS11} and \cite{GLS12}.
Inter alia, Geiß-Leclerc-Schröer succeeded to realize 
the universal enveloping algebras of the positive parts of a semisimple finite-dimensional complex Lie algebras (i.e. in finite type)
as convolution algebras of constructible functions on varieties of representations of their algebras $H$.

\bigskip

This fairly complete picture in finite types 
motivates us to pass on to affine types 
where GLS algebras $H$ are often representation wild 
but their hereditary deformations $\tilde{H}$ are still representation tame 
with very well understood module categories by 
\cite{DR76} and \cite{Rin76}. 
The connected valued graphs representing affine Cartan matrices 
are displayed in Table \ref{tab:affine}. 
Assume now that $K$ is algebraically closed of characteristic $0$,
that $C$ is a connected affine Cartan matrix 
of size $n\times n$
and that the symmetrizer $D$ is minimal. 
Let $\bmeta\in\bN^n$ be the primitive null root of $C$. 
The Main Theorem \ref{thm:main} of the present work 
states that every rank vector $\bv\in\bN^n$
can be uniquely written as a sum 
$\bv = \bw + m \bmeta$ 
for some $\bw\in\bN^n$ and $m\geq 0$
such that 
the generic locally free representation of $H$  with rank vector $\bv$ 
is isomorphic to a direct sum of a rigid locally free
representation with rank vector $\bw$ 
and $m$ representations with rank vector $\bmeta$ 
from an explicit $1$-parameter family.
This is a generalized canonical decomposition 
and we obtain such a decomposition by ``folding" 
Kac's canonical decomposition of dimension vectors over path algebras. 
For general symmetrizers, 
say $k D$ with $k\geq 1$,
we expect a similar canonical decomposition of rank vectors to be possible
but involving a $k$-parameter family of locally free representations with rank vector $\bmeta$; see Remark \ref{rem:general_symmetrizer}.
Our proof strategy relies on methods from $\tau$-tilting theory. 
Conversely, 
our $1$-parameter family confirms recent $\tau$-tilted versions of the second Brauer-Thrall Conjecture 
raised in \cite{M22} and \cite{STV21}
for the class of GLS algebras
and shows that affine GLS algebras are $E$-tame in the sense of \cite{DF15} and \cite{AI21}.

\bigskip

After fixing preliminary conventions of general nature,
we proceed with a more detailed overview of the content.
For the definition of GLS algebras we refer to \cite{GLS17i}, \cite{GLS20}
(and to Section \ref{sec:definition_gls} formulated in the language of valued quivers).

\subsection{Conventions} \label{sec:intro_conventions}

Throughout, we fix an algebraically closed ground field $K$ of characteristic $0$.
By an \emph{algebra} $A$ we mean an associative unital $K$-algebra
and by an \emph{$A$-module} we mean a finitely-generated left $A$-module
unless statet otherwise.
We write $\mod(A)$ for the Abelian category of $A$-modules
and $\proj(A)$ for the full subcategory of projective $A$-modules.
For a commutative ring $R$, 
let $\K_0(A)_R := \K_0(A) \otimes_{\bZ} R$ denote the $R$-linear \emph{Grothendieck group} of $\mod(A)$.
Further, $\K_0(A)^+$ denotes the submoniod of classes of $A$-modules in $\K_0(A)$
and $\K_0(A)_R^* := \Hom_{R}(\K_0(A)_R,R)$ the $R$-linear \emph{weight space} of $A$.
We write $\K_0^{\fin}(A)_R$ for the $R$-linear Grothendieck group of the full subcategory of $\mod(A)$
consisting of $A$-modules of finite projective dimension. 
If $A$ is finite-dimensional,
we have the \emph{Euler pairing}
\begin{align} \label{euler_pairing}
    \euler{-,?}_A\colon \K_0^\fin(A)\times \K_0(A) \to \bZ,
    ~~~~
    \euler{V,W}_A := \sum_{i\geq 0} (-1)^i \dimext_A^i(V,W)
\end{align}
where $\dimext^i_A(V,W) :=\dim_K \Ext^i_A(V,W)$ and $\dimhom_A(V,W) := \dimext^0_A(V,W)$ for $V,W\in\mod(A)$.
The Euler pairing induces an embedding
\begin{align} \label{euler_embedding}
    (-)^\vee\colon \K_0^{\fin}(A) \hookrightarrow \K_0(A)^*
\end{align}
which is even an isomorphism because $A$ is a split $K$-algebra.
Further, 
we let $\D:=\Hom_K(-,K)$ be the \emph{standard duality}
and denote by $\tau_A$ the \emph{Auslander-Reiten translation} 
for a finite-dimensional algebra $A$.
We refer to the book \cite{ARS95} for background on the representation theory of finite-dimensional algebras and Auslander-Reiten theory.

\bigskip

A \emph{quiver} $Q = (Q_0,Q_1,s,t)$ consists of a finite set of \emph{vertices} $Q_0$,
a finite set of \emph{arrows} $Q_1$ and
$s,t\colon Q_1 \to Q_0$ associate to an arrow $a$
its \emph{source} $s(a)$ and \emph{target} $t(a)$.
We write $KQ$ for the \emph{path algebra} of $Q$ with coefficients in a field $K$
where we concatenate arrows and paths like our functions from right to left.
We freely identify $KQ$-modules and $K$-linear representations of $Q$.
In particular, if $A \cong KQ/I$ for an admissible ideal $I$,
we canonically identify $\K_0(A) = \bZ Q_0$ with basis elements $\be_i = [S_i]$ 
and $S_i$ the simple $KQ$-module at $i\in Q_0$.
In general, we write $RX$ for the free $R$-module with basis given by the elements of a set $X$.
Given $\bd\in\bN Q_0$ written as $\bd = (d_i)_{i\in Q_0}$, 
let $\bRep(A,\bd)$ be the affine \emph{scheme of representations} of $A$ with dimension vector $\bd$.
We write $\Rep(A,\bd)$ for the affine \emph{variety of representations} of $A$ with dimension vector $\bd$,
that is the reduced variety of $\bRep(A,\bd)$.
We let the product of general linear groups
\begin{align*}
    \GL(K,\bd) := \prod_{i\in Q_0} \GL(K,d_i) 
\end{align*}
act on $\Rep(A,\bd)$ via conjugation
and denote by $\cO(V)$ the $\GL(K,\bd)$-orbit of $V\in\Rep(A,\bd)$. 
Write $\Irr(A,\bd)$ for the set of irreducible components of $\Rep(A,\bd)$ and set
\begin{align*}
    \Irr(A) := \bigsqcup_{\bd\in\bN Q_0} \Irr(A,\bd).
\end{align*}
Given an irreducible variety $\cZ$
and a property $\sfP$ of points of $\cZ$,
we say ``$\cZ$ satisfies $\sfP$ \emph{generically}''
provided there exists a non-empty open hence dense subset $\cU\subseteq \cZ$
such that every $z\in\cU$ satisfies $\sfP$.
In particular,
for a constructible function $f\colon \cZ \to \bN$ 
we write $f(\cZ)\in\bN$ for the generic value of $f$ on $\cZ$.
Finally,
$\dim\cZ$ denotes the \emph{Krull dimension} of $\cZ$.
For further background on varieties and schemes of representations we refer to 
\cite{CB93} and \cite{Gei96}.

\subsection{Locally free modules and $\tau$-tilting theory} \label{sec:intro_lf_tau_tilting}

With Adachi-Iyama-Reiten's $\tau$-tilting theory \cite{AIR14} 
an important new branch of representation theory emerged in 2014.
It generalizes the mutation theory of quivers with potentials \cite{DWZ08}
to arbitrary finite-dimensional algebras $A$
and may be seen from two ``Koszul dual'' perspectives:

On the one hand, 
central to $\tau$-tilting theory \cite{AIR14}
are \emph{$\tau$-rigid modules}, 
those $V\in\mod(A)$ with $\Hom_A(V,\tau_A(V)) = 0$,
a special class of rigid modules.
Any $\tau$-rigid module $V$ is determined by its $\g$-vector
\begin{align} \label{g_vector}
    \g(V) := [P_0] - [P_1] \in \K_0^\fin(A)
\end{align}
where $P_1\to P_0 \to V \to 0$ is a minimal projective presentation.
The $\g$-vectors of $\tau$-rigid modules
together with $\g$-vectors of shifted projective $A$-modules $\g(\Sigma P) := -\g(P)$ for $P\in\proj(A)$
span a fan $\fan(A)$ in $\K_0^\fin(A)_\bR$, 
the \emph{$\g$-vector fan} of $A$ introduced and studied in \cite[Section 6]{DIJ19}.

On the other hand, 
there are \emph{bricks} \cite[Section 4]{DIJ19}, 
that is modules $V\in\mod(A)$ with $\End_A(V)\cong K$, 
and more specifically \emph{stable modules} \cite{BST19}. 
A module $V\in\mod(A)$ is \emph{stable} in the sense of King \cite[Definition 1.1]{K94}
if there exists a weight $\theta\in\K_0(A)^*$
such that
\begin{align} \label{stable}
    \text{$\theta(V) = 0$ and $\theta(U) < 0$ for all proper non-zero submodules $U\subset V$}
\end{align}
This yields Bridgeland's \emph{wall and chamber structure} of the weight space $\K_0(A)_\bR^*$ \cite[Section 6]{B17}.
Brüstle-Smith-Treffinger \cite[Proposition 3.15]{BST19} show that 
the $\g$-vector fan embeds via the Euler pairing (\ref{euler_embedding}) into the wall and chamber structure of $\K_0(A)_\bR^*$.
Their key tool is \emph{Auslander-Reiten's $\g$-vector formula} \cite[Theorem 1.4]{AR85}
\begin{align} \label{ar_g_vector_formula}
    \euler{\g(V),\dimv(U)}_A = \dimhom_A(V,U) - \dimhom_A(U,\tau_A(V))
\end{align}
valid for all $V,U\in\mod(A)$.
If $A$ is $\tau$-tilting finite,
then $\fan(A)$ is \emph{complete} \cite[Theorem 5.4, Corollary 6.7]{DIJ19} 
i.e. covers all of $\K_0(A)_\bR^*$.
This allows to classify all stable modules via \emph{mutation} as developed in \cite{AIR14}.
For $\tau$-tilting infinite $A$,
only weights outside the $\g$-vector fan are left to consider.
To understand those, we invoke Geiß-Leclerc-Schröer's \emph{generically $\tau$-reduced components}
from \cite[Section 1.5]{GLS12} 
(where they are called strongly reduced and are defined for certain Jacobi algebras)
as natural generalizations of $\tau$-rigid modules.
These are irreducible components $\cZ\in\Irr(A)$
with
\begin{align} \label{tau_reduced}
    c_A(\cZ) = \dimhom_A^\tau(\cZ) := \min \{\dimhom_A(V,\tau_A(V)) \mid V \in \cZ\}
\end{align}
where $c_A(\cZ)$ denote the \emph{generic number of parameters} of $\cZ$
\begin{align} \label{generic_number_parameters}
    c_A(\cZ) := \min \{\dim \cZ - \dim \cO(V) \mid V \in \cZ\}.
\end{align}

We apply this general theory to GLS algebras.
It was already observed in \emph{Demonet's Lemma} \cite[Lemma 6.2]{GLS20} 
that $\tau$-tilting theory 
occurs naturally in the representation theory of GLS algebras;
namely an $H$-module $V$ is $\tau$-rigid if and only if $V$ is rigid and \emph{locally free}.
An $H$-module $V$ is said to be \emph{locally free} if $V(i) := e_i V$ is a free module over $H(i) := e_i H e_i$ 
for every $i\in Q_0$ and corresponding idempotent $e_i\in H$.
With the help of Plamondon's classification of generically $\tau$-reduced components \cite[Theorem 1.2]{Pla13}, 
we show that Demonet's Lemma generalizes to the geometric set up:

\begin{theorem}[Theorem \ref{thm:tau_reduced_locally_free}]
    Let $H$ be a GLS algebra. 
    The following are equivalent for $\cZ\in \Irr(H)$:
    \begin{enumerate}[label = (\roman*)]
        \item The component $\cZ$ is generically $\tau$-reduced.
        \item The component $\cZ$ is generically locally free.
    \end{enumerate}
\end{theorem}

The generically locally free components are described in \cite[Section 3.1]{GLS18ii}:
Every locally free $H$-module $V$ has an associated
\emph{rank vector} $\rkv(V) = (\rk_i(V))_{i\in Q_0}$ 
with entries $\rk_i(V)$ the rank of the free $H(i)$-module $V(i)$ for $i\in Q_0$.
To distinguish rank vectors from dimension vectors
our rank vectors live in $\bZ\Gamma_0$ for a set $\Gamma_0 = Q_0$
and dimension vectors live as usual in $\bZ Q_0$.
The symmetrizer $D$ then defines a linear embedding 
$D\colon \bZ\Gamma_0 \hookrightarrow \bZ Q_0$ 
such that $\dimv(V) = D(\rkv(V))$ for all locally free $V\in\mod(H)$.
Later, we define GLS algebras in terms of valued quivers $\Gamma$
and $\Gamma_0$ will be the set of vertices
justifying our ad hoc convention here.
For every rank vector $\br\in \bN \Gamma_0$
the subset
\begin{align} \label{lf_rep_variety}
    \Rep_\lf(H,\br) := \{V\in\Rep(H,D\br) \mid \text{$V$ is locally free}\} \subseteq \Rep(H,D\br)
\end{align}
is open, irreducible and smooth of dimension
\begin{align} \label{lf_rep_varitey_dim}
    \dim \Rep_{\lf}(H,\br) = \dim\GL(K,D\br) - q_{DC}(\br)
\end{align}
where $q_{DC}$ denotes the quadratic \emph{Tits form} associated to $DC$. 
Therefore, any generically locally free component is of the form
\begin{align}
    \cZ(\br) := \overline{\Rep_\lf(H,\br)}
\end{align}
for some rank vector $\br\in\bZ \Gamma_0$.

\subsection{Generic classification of locally free representations}

We call a GLS algebra $H = H(C,D,\Omega)$ \emph{affine} 
if $C$ is a generalized Cartan matrix of affine type.
They are characterized by the existence of a \emph{primitive null root} $\bmeta\in\bN \Gamma_0 \setminus\{0\}$, 
minimal with $q_{DC}(\bmeta) = 0$.
Via the identification $\rkv\colon \K_0^\fin(H)\xrightarrow{\sim} \bZ\Gamma_0$ 
and the Euler pairing (\ref{euler_embedding}), 
one obtains the \emph{defect} $\partial := \euler{\bmeta, - }_H \in \K_0(H)^*$
which already plays a crucial role in the representation theory of affine quivers.
We can now state our Main Theorem:

\begin{theorem}[Theorem \ref{thm:generic_classification}] \label{thm:main}
    Let $H = H(C,D,\Omega)$ with $C$ a connected affine Cartan matrix 
    and $D$ its minimal symmetrizer.
    For any rank vector $\bv\in \bN \Gamma_0$ exist unique $m\geq 0$ and $\bw\in\bN \Gamma_0$
    such that
    \begin{align} \label{folded_kac}
        \cZ(\bv) = \overline{\cZ(\bmeta)^m \oplus \cZ(\bw)}
    \end{align}
    and
    \begin{align} \label{generic_classification}
        \cZ(\bmeta) = \overline{\bigcup_{\lambda\in\bP^1} \cO(V_\lambda)} &&
        \cZ(\bw) = \overline{\cO(W)}
    \end{align}
    for some $\tau$-rigid $W\in\mod(H)$
    and a $\bP^1$-family of $\partial$-semistable $V_\lambda\in\mod(H)$.
\end{theorem}

The expression in (\ref{folded_kac})
is the direct sum of irreducible components as defined in e.g. \cite{CBS02},
and means that the generic element of $\cZ(\br)$
is isomorphic to a direct sum of $m$ elements of $\cZ(\bmeta)$ with an element of $\cZ(\bw)$.
This is a generalized Kac decomposition
and we do obtain the sum decomposition $\br = m\bmeta + \bw$ 
by ``folding'' the corresponding Kac decomposition over an unfolded path algebra
introduced in Section \ref{sec:unfolded}.

\begin{remark}\label{rem:general_symmetrizer}
    Let $C$ be a connected affine Cartan matrix
    with minimal symmetrizer $D$.
    A general symmetrizer for $C$
    is of the form $kD$ for some $k\geq 1$.
    The same canonical decomposition as in (\ref{folded_kac})
    should be possible over $H_k := H(C,kD,\Omega)$.
    However, we expect in this case
    \begin{align} \label{higher_symmetrizer}
        \cZ(\bmeta) = \overline{\bigcup_{\underline{\mu}\in \bA^k} \cO(V'_{\underline{\mu}})}
    \end{align}
    for a $k$-parameter family of representations 
    $V'_{\underline{\mu}} \in \Rep_{\lf}(H_k,\bmeta)$ for $\underline{\mu}\in\bA^k$ with the following properties:
    \begin{enumerate}[label = (\roman*)]
        \item $\End_{H_k}(V'_{\underline{\mu}}) \cong K[X]/\ideal{X^k}$ where $X$ is variable;
        \item $V'_{\underline{\mu}}$ is free as an $\End_{H_k}(V'_{\underline{\mu}})$-module;
        \item the top of $V'_{\underline{\mu}}$ as an $\End_{H_k}(V'_{\underline{\mu}})$-module
        is isomorphic to $V_\lambda$ for some $\lambda\in\bA^1$ as in (\ref{generic_classification}).
    \end{enumerate}
\end{remark}

\subsection{Affine type $\tilde{\BC}_1$} \label{sec:intro_bc1}

There are two affine Cartan matrices of rank $1$, namely
\begin{align*}
    C = \rsm{2 & -2 \\ -2 & 2} &&
    \text{and} &&
    C = \rsm{2 & -1 \\ -4 & 2}.
\end{align*}
The former is symmetric and its GLS algebra (for the minimal symmetrizer) is the Kronecker algebra
which plays a core role for all simply laced types.
The other is symmetrizable of type $\tilde{\A}_{1,2}$ in Moody's notation \cite[Table of Euclidean Matrices]{Moo68},
$\tilde{\A}_{1,1}$ in Dlab-Ringel's notation \cite[p. 3]{DR76}
and $A^{(2)}_2$ in Kac's notation \cite[Table Aff 2]{K90}.
We choose Macdonald's name $\tilde{\BC}_1$ \cite[p. 103]{Mac71},
because the corresponding GLS algebra $H$ fits naturally in the $\tilde{\BC}_{n\geq 2}$ family.
Indeed we will see that the algebra $H$ is of similar importance to the GLS algebras of type $\tilde{\BC}_n$
as the Kronecker algebra is to the path algebras of simply laced types.
In particular, $H$ is not only an important example of an affine GLS algebra
but does play a key role in our proof of the Main Theorem \ref{thm:main}.
For the minimal symmetrizer $D = \rsm{4 & 0 \\ 0 & 1}$ and up to duality $H = KQ/I$ is given by the quiver with relations
\[
    \begin{tikzcd}
        Q\colon ~
        2 \ar[r, "\alpha"] &
        1 \ar[loop, out=30, in=-30, distance=5ex, "\varepsilon"] &&&&
        I := \ideal{\varepsilon^4}.
    \end{tikzcd}
\]
We achieve a classification of all stable modules over the GLS algebras $H$ of affine type 
$\tilde{\BC}_{1}$.
The indecomposable $\tau$-rigid $H$-modules are well known to be precisely the preprojective and preinjective ones
i.e. Auslander-Reiten translations of projectives and injectives.
The support $\tau$-tilting exchange quiver and the $\g$-vector fan of $H$ is sketched in Figure \ref{fig:exchange_quiver_and_fan}.
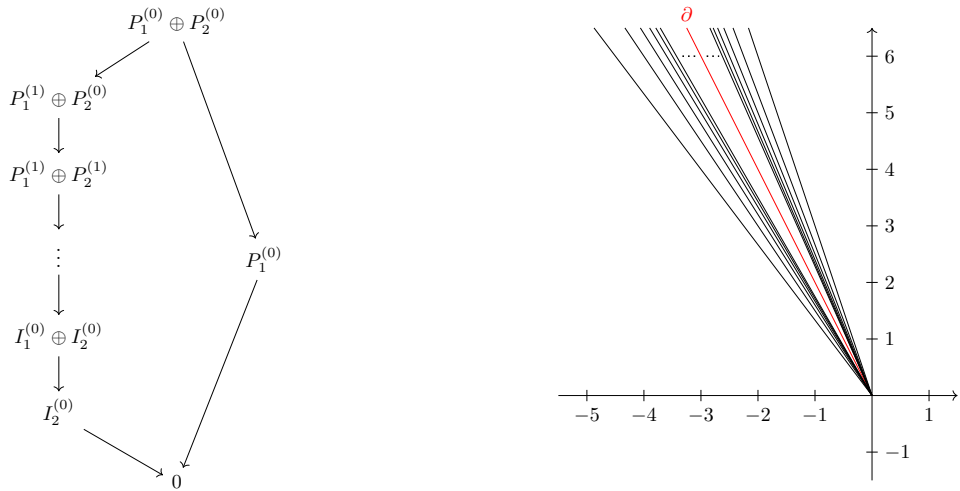
\begin{figure}[h]
    \centering
    \begin{multicols}{2}
        \scalebox{.75}{
            \begin{tikzcd}[column sep = .1em, ampersand replacement = \&]
                \& 
                P_1^{(0)} \oplus P_2^{(0)}
                \ar[dl] \ar[dddr] 
                \& 
                \\
                P_1^{(1)} \oplus P_2^{(0)} 
                \ar[d] 
                \&\& 
                \\
                P_1^{(1)} \oplus P_2^{(1)} 
                \ar[d] 
                \&\& 
                \\
                \vdots 
                \ar[d] 
                \&\& 
                P_1^{(0)}
                \ar[dddl] 
                \\
                I_1^{(0)} \oplus I_2^{(0)} \ar[d] 
                \&\& 
                \\
                I_2^{(0)} \ar[dr] 
                \&\& 
                \\
                \& 
                0
                \&
            \end{tikzcd}
        }
        
        \scalebox{.75}{
            \begin{tikzpicture}[domain=0:6.5]
                \draw[->] (-5.5,0) -- (1.5,0) node[right] {};
                \draw[->] (0,-1.5) -- (0,6.5) node[above] {};
                \foreach \x in {1,-1,-2,-3,-4,-5}
                    \draw[-] (\x,.1) -- (\x,-.1) node[below] {$\x$};
                \foreach \y in {-1,1,2,3,4,5,6}
                    \draw[-] (-.1,\y) -- (.1,\y) node[right] {$\y$};
                \draw[color=red]    plot (-1/2*\x,\x)       node[above] {$\partial$};
                \draw[] plot (-1/3*\x,\x)       node[right] {};
                \draw[] plot (-2/5*\x,\x)       node[right] {};
                \draw[] plot (-3/7*\x,\x)       node[right] {};
                \draw[] plot (-3/8*\x,\x)       node[right] {};
                \draw[] plot (-5/12*\x,\x)      node[right] {};
                \draw[] plot (-7/16*\x,\x)      node[right] {};
                \draw[] plot (-2/3*\x,\x)       node[right] {};
                \draw[] plot (-3/5*\x,\x)      node[right] {};
                \draw[] plot (-4/7*\x,\x)      node[right] {};
                \draw[] plot (-3/4*\x,\x)       node[right] {};
                \draw[] plot (-5/8*\x,\x)      node[right] {};
                \draw[] plot (-7/12*\x,\x)      node[right] {};
                \draw[dotted,thick] (-2.9,6) -- (-2.68,6);
                \draw[dotted,thick] (-3.1,6) -- (-3.39,6);
            \end{tikzpicture}
        }
    \end{multicols}
    \caption{
        Part of the support $\tau$-tilting exchange quiver (left)
        and the $\g$-vector fan (right) 
        for $H$ of affine type $\tilde{\BC}_{1}$.
        }
    \label{fig:exchange_quiver_and_fan}
\end{figure}
The red ray is spanned by the defect $\partial = (-1,2) \in \bR^2$
and does not belong to the $\g$-vector fan
but is the limit of walls in there.
By the work of \cite[Proposition 3.13]{BST19} and \cite[Theorem 3.8]{J15}, 
it remains to classify $\partial$-stable $H$-modules and we have

\begin{theorem}[Corollary \ref{cor:defect_stable}]
    Let $V\in\mod(H)$ be $\partial$-stable.
    Then $V \cong \bar{V}_\infty$ or $V\cong V_\lambda$ for some $\lambda\in K$ where
    \[
        \begin{tikzcd}[ampersand replacement=\&]
            \bar{V}_{\infty}\colon 
            ~
            K \ar[r,"\rsm{1\\0}"] 
            \&
            K^2 
            \ar[loop, out=30, in=-30, distance=5ex, "\rsm{0 & 0 \\ 1 & 0}"]
        \end{tikzcd},
        ~~~~ ~~~~ ~~~~ ~~~~
        \begin{tikzcd}[ampersand replacement=\&]
            V_\lambda\colon 
            ~ 
            K^2 
            \ar[rr,"\rsm{1 & 0 \\ 0 & 1 \\ 0 & \lambda \\ 0 & 0}"] 
            \&\&
            K^4 \ar[loop, out=30, in=-30, distance=5ex, "\rsm{0 & 0 & 0 & 0 \\
            1 & 0 & 0 & 0 \\ 0 & 1 & 0 & 0 \\ 0 & 0 & 1 & 0}"] 
        \end{tikzcd}.
    \]
\end{theorem}

Further, this gives rise to a generic classification of representations with vanishing defect 
(see Proposition \ref{prop:generic_defect_semistable})
and settles our Main Theorem \ref{thm:main} in this case (see Corollary \ref{cor:generic_locally_free}).
As an application, 
the generic classification allows for an explicit computation of locally free Caldero-Chapoton functions.
Together with a recent formula of Mou \cite[Theorem 1.1]{LMou22} 
and comparing with Sherman-Zelevinsky's canonical basis \cite[Theorem 2.8]{SZ04},
we obtain a ``generic basis'' in the spirit of \cite[Theorem 5]{GLS12}
for the (coefficient-free) cluster algebra $\cA(1,4)$ 
(see Remark \ref{rem:generic_basis}).

\subsection{A 1-parameter family of stable modules}

For any affine GLS algebra $H$ with minimal symmetrizer,
we construct an interesting $1$-parameter family of locally free $H$-modules
$(V_\lambda)_{\lambda\in\bP^1}$ of rank $\bmeta$ 
which are in general $\partial$-semistable and generically $\partial$-stable.
These are the modules appearing in the Main Theorem \ref{thm:main}.
They should be seen as degenerations of the homogeneous quasi-simple $\tilde{H}$-modules.

\begin{theorem}[Theorem \ref{thm:null_family}]
    Let $H = H(C,D,\Omega)$
    for an affine Cartan matrix $C$
    with minimal symmetrizer $D$.
    There are $V_\lambda\in \mod(H)$ for $\lambda\in\bP^1$ 
    with the following properties:
    \begin{enumerate}[label = (\roman*)]
        \item 
        Each $V_\lambda$ for $\lambda\in \bP^1$ is locally free with $\rkv(V_\lambda) = \bmeta$,
        \item
        for all $\lambda\neq \mu \in \bP^1$ is $\Hom_H(V_\lambda,V_\mu) = 0 = \Ext^1_H(V_\lambda,V_\mu)$,
        \item 
        for all $\lambda\in\bA^1$ is $\End_H(V_\lambda) \cong K$,
        \item 
        for all $\lambda\in\bP^1$ is $\Hom_H(V_\lambda,H) = 0$,
        \item 
        each $V_\lambda$ for $\lambda\in\bP^1$ is $\partial$-semistable,
        \item 
        for almost all $\lambda\in\bP^1$ is $\tau_H(V_\lambda) \cong V_\lambda$,
        \item 
        almost all $V_\lambda$ for $\lambda\in\bP^1$ are $\partial$-stable.
    \end{enumerate}
\end{theorem}

In particular, for the class of GLS algebras, 
this confirms recent $\tau$-tilted versions of the second Brauer-Thrall conjecture 
raised in \cite[Conjecture 1.3.(2)]{M22} and \cite[Conjecture 2]{STV21} (see also \cite{Pfe23bt}).
Our construction uses a generalized one-point extension technique
inspired by Ringel's approach to representations of affine path algebras \cite{Rin84}.
The modules $V_\lambda$ live in a wide subcategory $\cB\subseteq \mod(H)$
which is equivalent to a module category $\cB \simeq\mod(B)$ 
for a finite-dimensional algebra $B$.
The arising algebras $B$ may be of independent interest.
Moody \cite{M69} associates a \emph{tier number} $t(C)$
to every affine Cartan matrix $C$
(this is Kac's \emph{twisting number} from \cite{K90}).
Let $C'$ be the transposed matrix of $C$.
If $t({C'}) = 1$, 
in particular if $C$ is simply-laced, 
then $B$ is the Kronecker algebra.
If $C$ is of type $\tilde{\BC}_{n\geq 1}$, 
then $B$ is isomorphic to a GLS algebra of type $\tilde{\BC}_1$.
If $t(C') = 2$ and $C$ is not of type $\tilde{\BC}_{n\geq 1}$, 
then $B \cong KQ/I$ is a well known gentle algebra with
\[
    \begin{tikzcd}
        Q \colon &
        0 \ar[loop, in = 150, out = -150, distance = 5ex, "\delta_0"] \ar[r,"\beta"] &
        1 \ar[loop, in = 30, out = -30, distance = 5ex, "\delta_1"']
        &&&&
        I = \ideal{\delta_0^2,\delta_1^2}.
    \end{tikzcd}
\]
If $t(C') = 3$, i.e. $C$ is of type $\tilde{\G}_{2,1}$, 
then $B\cong KQ/I$ is given by
\[
    \begin{tikzcd}
        Q \colon &
        0 \ar[loop, in = 150, out = -150, distance = 5ex, "\delta_0"] \ar[r,"\beta"] &
        1 \ar[loop, in = 30, out = -30, distance = 5ex, "\delta_1"']
        &&&&
        I = \ideal{\delta_0^3,\delta_1^3, \delta_1^2\beta + \delta_1\beta\delta_0 + \beta\delta_0^2}.
    \end{tikzcd}
\]
In all cases $B$ deforms to a tame hereditary algebra, 
but the latter two are not GLS algebras.

\subsection{$\tau$-tilting tameness}

Recently, some notions of ``tameness'' in $\tau$-tilting theory arose e.g.
\cite{BST19}, \cite{AY23}, \cite{AI21}.
As a first hint that affine GLS algebras $H$ have a ``tame'' $\tau$-tilting theory,
it follows from the work of \cite{GLS20} and \cite{DR76} that $H$ is $\g$-tame in the sense of 
Aoki-Yurikusa \cite[Definition 1.2]{AY23}
that means the $\g$-vector fan $\fan(H)$ is dense in $\K_0^\fin(H)_\bR$.
This was our motivation to initiate the present work,
because that leaves only few weights to consider,
most prominently the defect.
Plamondon-Yurikusa prove in \cite{PY23} that representation tame algebras are $\g$-tame.
An important ingredient in their proof is a result of 
Geiß-Labardini-Fragoso-Schröer \cite[Theorem 3.2]{GLFS22} and \cite[Corollary 1.7]{GLFS23}
which states that representation tame algebras are $E$-tame 
in the sense of Derksen-Fei \cite[Definition 4.6]{DF15} and Asai-Iyama \cite[Definition 6.3]{AI21}.
More precisely,
it is easy to see that representation tame algebras
are what we call \emph{generically $\tau$-reduced tame} (see e.g. \cite[Lemma 3]{CC15inv}):

\begin{definition}
    Let $A$ be a finite-dimensional (basic) algebra.
    Then $A$ is said to be \emph{generically $\tau$-reduced tame}
    if $c_A(\cZ)\leq 1$
    for all generically indecomposable and $\tau$-reduced components $\cZ \in \Irr(A)$.
\end{definition}

It then follows from \cite[Theorem 1.5]{GLFS23} that any generically $\tau$-reduced tame algebra is $E$-tame.
As a consequence of our Main Theorem we obtain:

\begin{corollary}[Corollary \ref{cor:tau_reduced_tame}]
    Let $H = H(C,D,\Omega)$ with $C$ affine.
    If $D$ is minimal, then $H$ is generically $\tau$-reduced tame.
    For general $D$ is $H$ $E$-tame.
\end{corollary}

\begin{table}
    \centering
        \scalebox{.8}{
            \begin{tabular}{ l r r >{\centering}m{5.5cm} c c c c c }

    Type 
    &
    Rank 
    &
    Tier
    &
    Valued Graph
    &
    Null Root
    \\ 
    \hline
    \\

    $\tilde{\A}_1$
    &
    1
    &
    1
    & 
    \begin{tikzpicture}
        \node (a1) at (0,0) {$\bullet$};
        \node (a2) at (1,0) {$\circ$} edge [-] node[above] {$2\mid 2$} (a1);
    \end{tikzpicture}
    &
    $\sm{1 & 1}$
    \\[7mm]

    $\tilde{\A}_{n}$ 
    &
    $n\geq 2$
    &
    1 
    &
    \begin{tikzpicture}
        \node (a1) at (0,0) {$\bullet$};
        \node (b1) at (1,.5) {$\bullet$};
        \node (b2) at (2,.5) {};
        \node (b3) at (3,.5) {};
        \node (b4) at (4,.5) {$\bullet$};
        \node (c1) at (1,-.5) {$\bullet$};
        \node (c2) at (2,-.5) {};
        \node (c3) at (3,-.5) {};
        \node (c4) at (4,-.5) {$\bullet$};
        \node (a2) at (5,0) {$\circ$};
        \draw[-] (a1) to (b1);
        \draw[-] (a1) to (c1);
        \draw[-] (b1) to (b2);
        \draw[-] (c1) to (c2);
        \draw[dotted] (b2) to (b3);
        \draw[dotted] (c2) to (c3);
        \draw[-] (b3) to (b4);
        \draw[-] (c3) to (c4);
        \draw[-] (b4) to (a2);
        \draw[-] (c4) to (a2);
    \end{tikzpicture}
    &
    $\sm{ & 1 & \cdots & 1 \\ 1 & & &  & 1 \\ & 1 & \cdots & 1}$
    \\[7mm]

    $\tilde{\B}_{n}$ 
    &
    $n\geq 2$
    &
    2
    &
    \begin{tikzpicture}
        \node (a1) at (0,0) {$\bullet$};
        \node (a2) at (1,0) {$\bullet$} edge [-] node[above] {$1\mid 2$} (a1);
        \node (a3) at (2,0) {} edge [-] (a2);
        \node (a4) at (3,0) {} edge [dotted] (a3);
        \node (a5) at (4,0) {$\bullet$} edge [-] (a4);
        \node (a6) at (5,0) {$\circ$} edge [-] node[above] {$2\mid 1$} (a5);
    \end{tikzpicture}
    &
    $\sm{1 & 1 & \cdots & 1 & 1}$
    \\ [7mm]

    $\tilde{\C}_{n}$ 
    &
    $n\geq 2$
    & 
    1
    &
    \begin{tikzpicture}
        \node (a1) at (0,0) {$\bullet$};
        \node (a2) at (1,0) {$\bullet$} edge [-] node[above] {$2\mid 1$} (a1);
        \node (a3) at (2,0) {} edge [-] (a2);
        \node (a4) at (3,0) {} edge [dotted] (a3);
        \node (a5) at (4,0) {$\bullet$} edge [-] (a4);
        \node (a6) at (5,0) {$\circ$} edge [-] node[above] {$1\mid 2$} (a5);
    \end{tikzpicture}
    &
    $\sm{1 & 2 & \cdots & 2 & 1}$
    \\ [7mm]

    $\tilde{\D}_{n}$ 
    &
    $n\geq 4$
    & 
    1
    &
    \begin{tikzpicture}
        \node (a1) at (0,-.5) {$\bullet$};
        \node (a2) at (1,0) {$\bullet$} edge [-] (a1);
        \node (a3) at (2,0) {} edge [-] (a2);
        \node (a4) at (3,0) {} edge [dotted] (a3);
        \node (a5) at (4,0) {$\bullet$} edge [-] (a4);
        \node (a6) at (5,-.5) {$\bullet$} edge [-] (a5);
        \node (b1) at (0,.5) {$\bullet$} edge [-] (a2);
        \node (b2) at (5,.5) {$\circ$} edge [-] (a5);
    \end{tikzpicture}
    &
    $\sm{1 & & & & 1\\ & 2 & \cdots & 2 \\1 & & & & 1}$
    \\ [7mm]

    $\tilde{\BC}_{1}$ 
    &
    1
    & 
    2
    &
    \begin{tikzpicture}
        \node (a1) at (0,0) {$\bullet$};
        \node (a2) at (1,0) {$\circ$} edge [-] node[above] {$1\mid 4$} (a1);
    \end{tikzpicture}
    &
    $\sm{2 & 1}$
    \\ [7mm]

    $\tilde{\BC}_{n}$
    & 
    $n\geq 2$
    &
    2
    &
    \begin{tikzpicture}
        \node (a1) at (0,0) {$\bullet$};
        \node (a2) at (1,0) {$\bullet$} edge [-] node[above] {$1\mid 2$} (a1);
        \node (a3) at (2,0) {} edge [-] (a2);
        \node (a4) at (3,0) {} edge [dotted] (a3);
        \node (a5) at (4,0) {$\bullet$} edge [-] (a4);
        \node (a6) at (5,0) {$\circ$} edge [-] node[above] {$1\mid 2$} (a5);
    \end{tikzpicture}
    &
    $\sm{2 & 2 & \cdots & 2 & 1}$
    \\ [7mm]

    $\tilde{\BD}_{n}$ 
    &
    $n\geq 3$
    &
    1 
    &
    \begin{tikzpicture}
        \node (a1) at (0,0) {$\bullet$};
        \node (a2) at (1,0) {$\bullet$} edge [-] node[above] {$1\mid 2$} (a1);
        \node (a3) at (2,0) {} edge [-] (a2);
        \node (a4) at (3,0) {} edge [dotted] (a3);
        \node (a5) at (4,0) {$\bullet$} edge [-] (a4);
        \node (a6) at (5,-.5) {$\bullet$} edge [-] (a5);
        \node (b1) at (5,.5) {$\circ$} edge [-] (a5);
    \end{tikzpicture}
    &
    $\sm{ & & & & 1 \\ 2 & 2 & \cdots & 2 \\ & & & & 1}$
    \\ [7mm]

    $\tilde{\CD}_{n}$ 
    &
    $n\geq 3$
    & 
    2
    &
    \begin{tikzpicture}
        \node (a1) at (0,0) {$\bullet$};
        \node (a2) at (1,0) {$\bullet$} edge [-] node[above] {$2\mid 1$} (a1);
        \node (a3) at (2,0) {} edge [-] (a2);
        \node (a4) at (3,0) {} edge [dotted] (a3);
        \node (a5) at (4,0) {$\bullet$} edge [-] (a4);
        \node (a6) at (5,-.5) {$\bullet$} edge [-] (a5);
        \node (b1) at (5,.5) {$\circ$} edge [-] (a5);
    \end{tikzpicture}
    &
    $\sm{ & & & & 1 \\ 1 & 2 & \cdots & 2 \\ & & & & 1}$
    \\ [7mm]

    $\tilde{\E}_{6}$ 
    &
    6
    & 
    1
    &
    \begin{tikzpicture}
        \node (a1) at (3,0) {$\bullet$};
        \node (a2) at (2,0) {$\bullet$} edge [-] (a1);
        \node (a3) at (1,0) {$\bullet$} edge [-] (a2);
        \node (a4) at (2,-.5) {$\bullet$} edge [-] (a3);
        \node (a5) at (3,-.5) {$\bullet$} edge [-] (a4);
        \node (b1) at (2,.5) {$\bullet$} edge [-] (a3);
        \node (c1) at (3,.5) {$\circ$} edge [-] (b1);
    \end{tikzpicture}
    &
    $\sm{ & 2 & 1 \\ 3 & 2 & 1 \\ & 2 & 1}$
    \\ [7mm]

    $\tilde{\E}_{7}$  
    & 
    7
    &
    1
    &
    \begin{tikzpicture}
        \node (a1) at (3,.5) {$\circ$};
        \node (a2) at (2,.5) {$\bullet$} edge [-] (a1);
        \node (a3) at (1,.5) {$\bullet$} edge [-] (a2);
        \node (a4) at (0,0) {$\bullet$} edge [-] (a3);
        \node (a5) at (1,0) {$\bullet$} edge [-] (a4);
        \node (a6) at (2,0) {$\bullet$} edge [-] (a5);
        \node (a7) at (3,0) {$\bullet$} edge [-] (a6);
        \node (b1) at (1,-.5) {$\bullet$} edge [-] (a4);
    \end{tikzpicture}
    & 
    $\sm{ & 3 & 2 & 1 \\ 4 & 3 & 2 & 1 \\ & 2}$
    \\ [7mm]

    $\tilde{\E}_{8}$  
    & 
    8
    &
    1
    &
    \begin{tikzpicture}
        \node (a1) at (2,0) {$\bullet$};
        \node (a2) at (1,0) {$\bullet$} edge [-] (a1);
        \node (a3) at (0,0) {$\bullet$} edge [-] (a2);
        \node (a4) at (1,.5) {$\bullet$} edge [-] (a3);
        \node (a5) at (2,.5) {$\bullet$} edge [-] (a4);
        \node (a6) at (3,.5) {$\bullet$} edge [-] (a5);
        \node (a7) at (4,.5) {$\bullet$} edge [-] (a6);
        \node (a8) at (5,.5) {$\circ$} edge [-] (a7);
        \node (b1) at (1,-.5) {$\bullet$} edge [-] (a3);
    \end{tikzpicture}
    &
    $\sm{ & 5 & 4 & 3 & 2 & 1\\ 6 & 4 & 2 \\ & 3}$
    \\ [7mm]

    $\tilde{\F}_{4,1}$  
    & 
    4
    &
    1
    &
    \begin{tikzpicture}
        \node (a1) at (4,0) {$\circ$};
        \node (a2) at (3,0) {$\bullet$} edge [-] (a1);
        \node (a3) at (2,0) {$\bullet$} edge [-] (a2);
        \node (a4) at (1,0) {$\bullet$} edge [-] node[above] {$1\mid 2$} (a3);
        \node (a5) at (0,0) {$\bullet$} edge [-] (a4);
    \end{tikzpicture}
    &
    $\sm{2 & 4 & 3 & 2 & 1}$
    \\ [7mm]

    $\tilde{\F}_{4,2}$  
    & 
    4
    &
    2
    &
    \begin{tikzpicture}
        \node (a1) at (0,0) {$\bullet$};
        \node (a2) at (1,0) {$\bullet$} edge [-] (a1);
        \node (a3) at (2,0) {$\bullet$} edge [-] node[above] {$2\mid 1$} (a2);
        \node (a4) at (3,0) {$\bullet$} edge [-] (a3);
        \node (a5) at (4,0) {$\circ$} edge [-] (a4);
    \end{tikzpicture}
    &
    $\sm{1 & 2 & 3 & 2 & 1}$
    \\ [7mm]

    $\tilde{\G}_{2,1}$  
    & 
    2
    &
    1
    &
    \begin{tikzpicture}
        \node (a1) at (2,0) {$\circ$};
        \node (a2) at (1,0) {$\bullet$} edge [-] (a1);
        \node (a3) at (0,0) {$\bullet$} edge [-] node[above] {$1\mid 3$} (a2);
    \end{tikzpicture}
    &
    $\sm{3 & 2 & 1}$
    \\ [7mm]

    $\tilde{\G}_{2,3}$  
    & 
    2
    &
    3
    &
    \begin{tikzpicture}
        \node (a1) at (0,0) {$\bullet$};
        \node (a2) at (1,0) {$\bullet$} edge [-] node[above] {$3\mid 1$} (a1);
        \node (a3) at (2,0) {$\circ$} edge [-] (a2);
    \end{tikzpicture}
    &
    $\sm{1 & 2 & 1}$
    \\ [7mm]
\end{tabular}
        }
    \caption{
        The affine valued graphs.
        The name $\tilde{\BC}_1$ is from \cite[p. 103]{Mac71} 
        and called 
        $\tilde{\A}_{1,2}$ in \cite[Table of Euclidean Matrices]{Moo68}, 
        $\tilde{\A}_{1,1}$ in \cite[p. 3]{DR76}
        and $A_2^{(2)}$ in \cite[Table Aff 2]{K90}.
        The remaining symbols, but not their indices, are those from \cite[p. 3]{DR76}.
        The first index equals the rank i.e. the number of vertices minus $1$.
        The second index, if needed, equals the tier number defined in \cite{M69}
        and coincides with the twisting number in \cite{K90}
        appearing there as superscripts.
        Our choice of extending vertex is highlighted in white and coincides with those in \cite[Table 2]{M69}.
        We also list their primitive null roots and arrange their coefficients in the shape of the valued graph.
    }
    \label{tab:affine}
\end{table}

\newpage

\section{Preliminaries} \label{sec:vq_combinatorics}

In this section we collect some necessary preliminaries for the theory of GLS algebras.
Originally, GLS algebras $H = H(C,D,\Omega)$ are associated to 
generalized Cartan matrices $C$ with symmetrizer $D$ and orientation $\Omega$ (see \cite{GLS17i}).
Our construction of a family of stable modules over affine GLS algebras in Section \ref{sec:stable_family}
is based on a generalized one-point extension technique.
Therefore, it will be more convenient for us to work with valued quivers $\Gamma$ as in \cite{DR76}
and we recall their combinatorics in Section \ref{sec:valued_quivers}.
For us valued quivers always come with a symmetrizer and, as the name suggests, an orientation.
Thus, the data of a valued quiver $\Gamma$ will be equivalent to a triple $(C,D,\Omega)$.
Further, almost all algebras studied here are tensor algebras of modulated simple quivers
and we fix our notation in Section \ref{sec:modulation}.
Modulations and their representations generalize the theory of species
and are systematically developed in \cite{Li12}.

\subsection{Local algebra} \label{sec:local}

Let $K$ be an algebraically closed field of characteristic $0$
and fix a formal variable $\epsilon$. 
For each $k\geq 1$ consider the \emph{ring of formal power series} and
the \emph{field of Laurent series}
\begin{align*}
    R_k := K[[\epsilon^{1/k}]] &&
    L_k := K((\epsilon^{1/k}))
\end{align*} 
Set $R := R_1$ and $L := L_1$.
The field of Laurent series is a \emph{quasi-finite field} 
in the sense of \cite{Ser79}
thus enjoys particularly nice properties:
The field $L_k$ is up to isomorphism the unique field extension of $L$ 
of degree $[L_k:L] = k$
and its Galois group is cyclic $G_k:=\Gal(L_k/L)\cong \bZ/k$ 
with generator $\sigma_k \colon \varepsilon^{1/k} \mapsto \zeta_k\varepsilon^{1/k}$
for a fixed primitive $k^{th}$ root of unity $\zeta_k\in K$.
In particular, 
the algebraic closure of $L$ is the \emph{field of Puiseux series}
\begin{align*}
    \overline{L} = L_\infty := \varprojlim_{k\geq 0} L_k
\end{align*}
and its Galois group is isomorphic to the group of profinite integers $\bar{G} = G_\infty:=\Gal(\bar{L}/L)\cong\hat{\bZ}$.
Recall the canonical short exact sequence of abelian groups
\begin{align} \label{chinese_remainder_sequence}
    0 \to G_m \xrightarrow{\iota} G_k \times G_l \xrightarrow{\pi} G_n \to 0
\end{align}
for all $1\leq k,l < \infty$, $m = \lcm(k,l)$ and $n = \gcd(k,l)$.
This gives rise to a well defined isomorphism of rings
\begin{align} \label{tensor_product_of_fields}
    L_k \otimes_{L} L_l \xrightarrow{\sim} \prod_{g\in G_n} L_m, ~ 
    x \otimes y \mapsto (g_1(x) \cdot g_2(y))_{g\in G_n}
\end{align}
where we canonically identify $G_n \cong (G_k\times G_l)/G_m$ 
by means of the short exact sequence (\ref{chinese_remainder_sequence}).

\subsection{Valued quivers} \label{sec:valued_quivers}

A \emph{simple graph} is a tuple $\Gamma = (\Gamma_0,\Gamma_1)$ consisting of
a set $\Gamma_0$ and a subset $\Gamma_1\subseteq \Gamma_0 \times \Gamma_0$
such that the following holds:
\begin{enumerate}[label = (\roman*)]
    \item If $(j,i)\in \Gamma_1$, then $(i,j)\in \Gamma_1$;
    \item for all $i\in\Gamma_0$ is $(i,i)\not \in \Gamma_1$.
\end{enumerate}
Elements of $\Gamma_0$ are called \emph{vertices},
and elements of $\Gamma_1$ are called \emph{halfedges}.
A \emph{valuation} $\nu$ on a simple graph $\Gamma$ 
assigns to every halfedge a positive integer:
\begin{align*}
    \nu:\Gamma_1\to\bZ_{>0},~ (i,j) \mapsto \nu_{ij}.
\end{align*}
A \emph{symmetrizer} $\bc:\Gamma_0\to\mathbb{Z}_{>0}, i \mapsto c_i$
for a valuation $\nu$ assigns to every vertex a positive integer  
such that for all $(j,i)\in\Gamma_1$ holds
\begin{align*}
    c_i\nu_{ij} = c_j\nu_{ji}.
\end{align*}
A \emph{valued graph} is a tuple $\Gamma = (\Gamma_0,\Gamma_1, \nu, \bc)$
consisting of a simple graph $(\Gamma_0,\Gamma_1)$ 
and a valuation $\nu$ 
with symmetrizer $\bc$.
A valued graph is \emph{simply laced} provided $c_i = 1$ for all $i\in\Gamma_0$.
A \emph{simple quiver} $\Gamma = (\Gamma_0, \Gamma_1, \Omega)$ 
is a simple graph $(\Gamma_0,\Gamma_1)$ 
together with an \emph{orientation} $\Omega$,
that is a subset $\Omega\subseteq \Gamma_1$ 
such that either $(i,j)\in\Omega$ or $(j,i)\in\Omega$ for each $(i,j)\in\Gamma_1$.
Write $i\to j$ provided $(j,i)\in \Omega$ 
and $\rightsquigarrow$ for the transitive closure of the relation $\to$.
Finally, a \emph{valued quiver} $\Gamma = (\Gamma_0,\Gamma_1,\Omega,\nu,\bc)$ 
is a simple quiver $(\Gamma_0,\Gamma_1,\Omega)$
together with a valuation $\nu$ 
and a symmetrizer $\bc$.
Throughout, we will assume that all our valued quivers satisfy
\begin{enumerate}[label = (\roman*)]
    \item the set of vertices is finite 
    i.e. $\norm{\Gamma_0} < \infty$,
    \item there are no oriented cycles 
    i.e. there is no $i\in\Gamma_0$ with $i \rightsquigarrow i$.
\end{enumerate}

Given a valued quiver $\Gamma$,
consider the lattice $\bZ\Gamma_0$ with standard basis $(\bmalpha_i \mid i\in\Gamma_0)$ 
and equipped with the bilinear \emph{Ringel form} $\euler{-,-}_\Gamma\colon \bZ \Gamma_0\times \bZ \Gamma_0 \to \bZ$,
its \emph{symmetrization} $\cartan{-,-}_\Gamma\colon \bZ \Gamma_0\times \bZ \Gamma_0 \to \bZ$ 
and associated quadratic \emph{Tits form} $q_\Gamma\colon \bZ \Gamma_0 \to \bZ$ defined by
\begin{align} \label{ringel_form}
    \euler{\bv,\bw}_\Gamma := \sum_{i \in \Gamma_0} c_iv_iw_i - \sum_{i\to j \in \Omega} c_i\nu_{ij}v_iw_j,
    &&
    \cartan{\bv,\bw}_\Gamma := \euler{\bv,\bw}_\Gamma + \euler{\bw,\bv}_\Gamma,
    &&
    q_\Gamma(\bv) := \euler{\bv,\bv}_\Gamma
\end{align}
for $\bv,\bw\in\bZ\Gamma_0$ written as $\bv = (v_i)_{i\in Q_0}$ and $\bw = (w_i)_{i\in Q_0}$.
The set of non-negative vectors is 
\begin{align*}
    \bN\Gamma_0 := \{\bv \in \bZ\Gamma_0 \mid \text{For all $i\in\Gamma_0$ is $v_i \geq 0$} \}.
\end{align*}
Further, write $\bR\Gamma_0 := \bZ\Gamma_0 \otimes_{\bZ} \bR$ 
for the $\bR$-vector space with basis $(\bmalpha_i \mid i\in\Gamma_0)$
and all forms in \ref{ringel_form} naturally extend from $\bZ\Gamma_0$ to $\bR\Gamma_0$.
For an anisotropic $\bv\in\bZ\Gamma_0$ define the reflection
\begin{align*}
    s_\bv\colon \bR\Gamma_0 \to \bR\Gamma_0 && s_\bv(\bw) := \bw - \frac{\cartan{\bw,\bv}_\Gamma}{q_\Gamma(\bv)} \cdot \bv
\end{align*}
The \emph{Weyl group} of $\Gamma$ is the subgroup $\W(\Gamma)\subseteq \Aut_\bZ(\bZ\Gamma_0)$ 
generated by all \emph{simple reflections} $s_i:=s_{\bmalpha_i}$ for $i\in\Gamma_0$.
The sets of \emph{real} and \emph{positive real roots} of $\Gamma$ are
\begin{align*}
    \Delta_\re(\Gamma) := \bigcup_{i\in\Gamma_0} \W(\Gamma)\bmalpha_i &&
    \Delta_\re^+(\Gamma) := \Delta_\re(\Gamma) \cap \bN\Gamma_0
\end{align*}
An ordering $\Gamma_0 = \{i_1,\dots,i_n\}$ is \emph{admissible}
if there exists a path $i_s \rightsquigarrow i_t$ in $\Gamma$
only if $s\geq t$.
The \emph{Coxeter transformation} is then defined as
\begin{align*}
    \Phi_\Gamma := s_{i_n} \cdots s_{i_1} \in \W(\Gamma).
\end{align*}
Note that $\Phi_\Gamma$ only depends on $\Omega$ 
and not on the choice of the admissible ordering.

\subsection{Modulations of simple quivers} \label{sec:modulation}

A \emph{modulation} $\cM$ of a simple quiver $\Gamma = (\Gamma_0,\Gamma_1,\Omega)$ is a tuple
\begin{align*}
    \left( \cM(i), \cM(j,i) \mid i\in\Gamma_0, (j,i)\in\Omega \right)
\end{align*}
consisting of a $K$-algebra $\cM(i)$ for every $i\in\Gamma_0$ 
and a $K$-vector space $\cM(j,i)$ for every $(j,i)\in\Omega$ 
with the structure of a right $\cM(i)$-module and a left $\cM(j)$-module
such that $K$ acts centrally.
For any $l\geq 0$ define
\begin{align*}
    \cM(\Gamma_0) := \prod_{i\in \Gamma_0} \cM(i)
    &&
    \cM(\Gamma_1) := \bigoplus_{(j,i)\in\Omega} \cM(j,i)
    &&
    \cM(\Gamma_l) := \underbrace{\cM(\Gamma_1) \otimes_{\cM(\Gamma_0)} \cdots \otimes_{\cM(\Gamma_0)} \cM(\Gamma_1)}_\text{l times}
\end{align*}
then $\cM(\Gamma_0)$ is naturally a $K$-algebra and 
$\cM(\Gamma_l)$ for $l\geq 1$ admits the structure of a $\cM(\Gamma_0)$-bimodule 
by considering $\cM(j,i)$ as 
a trivial right $\cM(k)$-module for any $k\in\Gamma_0\setminus\{i\}$
and a trivial left $\cM(k)$-module for any $k\in\Gamma_0\setminus\{j\}$.
The \emph{tensor algebra} of $\cM$ is then defined as
\begin{align*}
    \cM(\Gamma) := \bigoplus_{l\geq 0} \cM(\Gamma_l)
\end{align*}
with multiplication given by the natural maps 
$\cM(\Gamma_l)\otimes_{\cM(\Gamma_0)}\cM(\Gamma_{l'}) \rightarrow \cM(\Gamma_{l+l'})$ 
for $l,l'\geq 0$. 
A \emph{representation} of $\cM(\Gamma)$ is a tuple
\begin{align*}
    V = (V(i),V(j,i) \mid i\in\Gamma_0, (j,i)\in \Omega)
\end{align*}
consisting of $V(i) \in\mod(\cM(i))$ for every $i\in \Gamma_0$
and $V(j,i)\in\Hom_{\cM(j)}(\cM(j,i)\otimes_{\cM(i)}V(i), V(j))$ for every $(j,i)\in\Omega$.
A homomorphism $f\colon V \to W$ between representations of $\cM(\Gamma)$
is a tuple $f=(f_i \mid i\in\Gamma_0)$ with $f(i)\in\Hom_{\cM(i)}(V(i),W(i))$ for $i\in\Gamma_0$
such that for every $(j,i)\in\Omega$ the following diagram commutes
\[
    \begin{tikzcd}
        \cM(j,i) \otimes_{\cM(i)} V(i) 
        \ar[rr,"\text{$V(j,i)$}"] 
        \ar[d,"\id\otimes f(i)"'] 
        &&
        V(j) \ar[d,"f(j)"]
        \\
        \cM(j,i) \otimes_{\cM(i)} W(i) \ar[rr,"\text{$V(j,i)$}"] &&
        W(j)
    \end{tikzcd}
\]
This defines the category of representation $\rep(\cM(\Gamma))$.
There is a well known and natural equivalence of categories 
(see e.g. \cite[Theorem 3.2]{Li12})
\begin{align*}
    \mod(\cM(\Gamma)) \xrightarrow{\sim} \rep(\cM(\Gamma)).
\end{align*}
We will freely identify $\cM(\Gamma)$-modules and representations of $\cM(\Gamma)$.

\section{Representations of valued quivers} \label{sec:definition_gls}

Beside their algebras $H$,
Geiß-Leclerc-Schröer define in \cite{GLS20}
an order $\hat{H}$ over the formal power series ring $R = K[[\varepsilon]]$
which may be seen as a formal family of algebras
whose special fibre recovers the GLS algebra $H \cong \hat{H} \otimes_R K$
and whose generic fibre $\tilde{H} := \hat{H} \otimes_R L$ 
is a species over the field of Laurent series $L := K((\varepsilon))$
in the sense of \cite[Section 7]{Gab73}.
Geiß-Leclerc-Schröer explore in \cite{GLS17i} that representations of $H$ 
behave in many aspects like representations of the species $\tilde{H}$
when restricting attention to the class of \emph{locally free} $H$-modules 
whose basic properties are summarized Section \ref{sec:lf}.
The definition of GLS algebras, orders and species from \cite{GLS20}
is recalled in the language of valued quivers in Section \ref{sec:gls_modulation}.
In Section \ref{sec:unfolded},
we add to this triple a fourth algebra $\bar{H} := \tilde{H} \otimes_L \bar{L}$
which is isomorphic to an ordinary path algebra $\bar{H} \cong \bar{L} \bar{Q}$
over the field of Puiseux series $\bar{L}$ .
We refer to $\bar{H}$ as the \emph{unfolded path algebra} of $H$.
Intuitively, $\bar{H}$ is an honest deformation of $H$ 
and helps to understand the species $\tilde{H}$
in terms of a more elementary path algebra.
In particular, we describe rigid $\tilde{H}$-modules
in terms of rigid $\bar{H}$-modules via Galois descent in Section \ref{sec:tau_rigid}. 
We were inspired by \cite{GRM22}
where Galois descent was used 
to describe simple regular modules over the GLS species of type $\tilde{\BC}_1$.
The final Section \ref{sec:tau_reduced} is dedicated to 
generically $\tau$-reduced components for GLS algebras. 
There we show that generically $\tau$-reduced components 
are precisely the generically locally free ones
and obtain some stronger characterizations
for strongly primitive $\Gamma$.

\subsection{GLS modulations} \label{sec:gls_modulation}

Throughout, we fix a valued quiver $\Gamma = (\Gamma_0,\Gamma_1,\Omega,\nu,\bc)$.
Associate to $\Gamma$ the following integers 
defined for every $(j,i)\in\Gamma_1$
\begin{align*}
    g_{ji} := \gcd(\nu_{ji},\nu_{ij})
    &&
    f_{ji} := \frac{\nu_{ji}}{g_{ji}}.
\end{align*}
In \cite{GLS17i} and \cite{GLS20} three algebras associated to $\Gamma$ are introduced.
They are defined in terms of modulations $\hat{H}, \tilde{H}$ and $H$ of the underlying simple quiver $\Gamma$
given by
\begin{align*}
    \hat{H}(i) &:= K[[\varepsilon_i]] &&& 
    \hat{H}(j,i) &:= K[[\varepsilon_i,\varepsilon_j]]^{g_{ji}} \\
    \tilde{H}(i) &:= K((\varepsilon_i)) &&&
    \tilde{H}(j,i) &:= K((\varepsilon_i,\varepsilon_j))^{g_{ji}} \\
    H(i) &:= K[\varepsilon_i]/(\varepsilon_i^{c_i}) &&&
    H(j,i) &:= \left(K[\varepsilon_i,\varepsilon_j]/(\varepsilon_i^{c_i},\varepsilon_j^{c_j},\varepsilon_i^{f_{ji}}-\varepsilon_j^{f_{ij}})\right)^{g_{ji}}
\end{align*}
for $i\in\Gamma_0$ and $(j,i)\in\Omega$.
In particular, $H(j,i)$ is a free left $H(i)$-module of rank $\nu_{ij}$ 
and a free right $H(j)$-module of rank $\nu_{ji}$.
The \emph{GLS order} $\hat{H}(\Gamma)$, 
\emph{GLS species} $\tilde{H}(\Gamma)$ and 
\emph{GLS algebra} $H(\Gamma)$ 
are defined as the respective tensor algebras.

\begin{theorem}\cite[Theorem 1.2, Proposition 6.4]{GLS17i} \cite[Proposition 4.1, 4.3 and 4.5]{GLS20} \label{thm:gls_algebras}
    Let $\Gamma$ be a valued quiver.
    \begin{enumerate}[label=(\roman*)]
        \item $\hat{H}(\Gamma)$ is an $R$-order in $\tilde{H}(\Gamma)$ with $\gldim(\hat{H}(\Gamma))\leq 2$.
        \item $\tilde{H}(\Gamma)\cong \hat{H}(\Gamma) \otimes_R L$ is an $L$-species, in particular a hereditary finite-dimensional $L$-algebra.
        \item $H(\Gamma)\cong \hat{H}(\Gamma) \otimes_R K$ is a finite-dimensional $K$-algebra and $1$-Iwanaga-Gorenstein. 
    \end{enumerate}
\end{theorem}

Form now on we write $\cH := \cH(\Gamma)$ for any of the GLS modulations $\cH\in\{H,\hat{H},\tilde{H}\}$
whenever the valued quiver $\Gamma$ is clear from the context.
The various module categories for different GLS modulations are related by push-forwards along the canonical morphisms of tensor algebras:
\begin{equation} \label{reduction_localization}
    \begin{tikzcd}
        \hat{H}
        \ar[d,"\delta"] 
        \ar[r,"\iota"]
        &
        \tilde{H}
        \\
        H
    \end{tikzcd}
    ~~~~ ~~~~ ~~~~ ~~~~ 
    \begin{tikzcd}
        \mod \hat{H}
        \ar[d,"\delta^*"] 
        \ar[r,"\iota^*"]
        &
        \mod \tilde{H}
        \\
        \mod H
    \end{tikzcd}
    ~~~~ ~~~~ ~~~~ ~~~~
    \begin{tikzcd}
        \iota^* := - \otimes_{R} L
        \\
        \delta^* := - \otimes_{R} K
    \end{tikzcd}
\end{equation}
The functor $\delta^*$ is called \emph{reduction} and 
$\iota^*$ is \emph{localization}.
Localization is an exact functor
while reduction has more delicate properties.
We refer to \cite[Section 5]{GLS20} for a detailed study of their properties.

\bigskip

GLS algebras have presentations in terms of quivers with relations
which is the initial definition given in \cite[Section 1.4]{GLS17i}:

\begin{proposition}\cite[Proposition 6.4]{GLS17i} \label{prop:quiver_with_relations}
    Let $\Gamma$ be a valued quiver.
    Then $H\cong KQ/I$ where $Q = Q(\Gamma)$ is the quiver with vertices and arrows
    \begin{align*}
        Q_0 := \Gamma_0
        &&
        Q_1 
        := \{ \alpha_{ji}^{(g)} \mid \text{$(j,i)\in\Omega$ and $1\leq g\leq g_{ji}$}\}
        \cup \{ \varepsilon_i \mid i\in\Gamma_0\}
    \end{align*}
    for which sources and targets are given as the indices suggest
    \begin{align*}
        s(\alpha_{ji}^{(g)}) = i && t(\alpha_{ji}^{(g)}) = j && s(\varepsilon_i) = i = t(\varepsilon_i).
    \end{align*}
    and the ideal $I = I(\Gamma)$ is generated by the relations 
        \begin{enumerate}[label=(H\arabic*)]
            \item 
            $\varepsilon_i^{c_i} = 0$ 
            for every $i\in\Gamma_0$;
            \item 
            $\varepsilon_j^{f_{ij}}\alpha_{ji}^{(g)} = \alpha_{ji}^{(g)}\varepsilon_i^{f_{ji}}$
            for every $(j,i)\in\Omega$ and $1\leq g\leq g_{ji}$.
        \end{enumerate}
\end{proposition}

\subsection{Unfolded path algebras} \label{sec:unfolded}

To gain a better understanding of the GLS species $\tilde{H}$
we introduce its \emph{unfolded path algebra}
\begin{align*}
    \bar{H} := \bar{H}(\Gamma) := \tilde{H}(\Gamma) \otimes_{L} \bar{L}.
\end{align*}
Note that the Galois group $\bar{G} := \Gal(\bar{L}/L)$ acts naturally on $\bar{H}$.
This is a basic finite-dimensional hereditary algebra
over the algebraically closed field $\bar{L}$
hence $\bar{H} \cong \bar{L}\bar{Q}$ is the path algebra of a quiver $\bar{Q}$.
The quiver $\bar{Q} = \bar{Q}(\Gamma)$ has an explicit description as the \emph{unfolded quiver} of $\Gamma$:
Its vertices and arrows are
\begin{align*}
    \bar{Q}_0 &= \{(i,k) \mid \text{$i\in\Gamma_0$, $k\in\bZ/c_i$}\}
    \\
    \bar{Q}_1 &= \{ \alpha_{(j,l),(i,k)}^{(g)} \mid 
        \text{
            $(i,k),(j,l) \in \bar{Q}_0$ 
            and $\alpha_{ji}^{(g)}\in Q_1$
            such that $k \equiv l$ in $\bZ/\gcd(c_i,c_j)$
        }
    \}
\end{align*}
where sources and targets are given as the indices suggest
\begin{align*}
    s(\alpha_{(j,l),(i,k)}^{(g)}) = (i,k) 
    && t(\alpha_{(j,l),(i,k)}^{(g)}) = (j,l).
\end{align*}

\begin{lemma}\label{lem:unfolded_path_algebra}
    For every $1\leq c < \infty$ with $c_i \mid c$ for all $i\in\Gamma_0$
    there is an isomorphism of algebras 
    \begin{align*}
        \tilde{H}_c(\Gamma):=\tilde{H}(\Gamma) \otimes_L L_c \cong L_c\bar{Q}.
    \end{align*}
    such that $\Gal(L_c/L) = \ideal{\sigma_c}$ acts on $L_c\bar{Q}$ as follows:
    \begin{align*}
        \sigma_c(\lambda \cdot e_{(i,k)}) 
        = \sigma_c(\lambda) \cdot e_{(i,{k+1})} 
        &&
        \sigma_c(\lambda \cdot \alpha_{(j,l),(i,k)}^{(g)}) 
        = \sigma_c(\lambda) \cdot \alpha_{(j,{l+1}),(i,{k+1})}^{(g)}
    \end{align*}
    for $\lambda\in L_c$, $(i,k),(j,l) \in \bar{Q}_0$ and $1\leq g \leq g_{ji}$.
\end{lemma}
\begin{proof}
    Note that $L_c\bar{Q} = L_c(\bar{Q})$
    is the tensor algebra of the simple quiver underlying $\bar{Q}$ with modulation
    \begin{align*}
        L_c((i,k)) := L_c \cdot e_{(i,k)}
        &&
        L_c((j,l),(i,k)) := \bigoplus_{g = 1}^{g_{ji}} L_c \cdot \alpha_{(j,l)(i,k)}^{(g)}
    \end{align*}
    for $(i,k),(j,l)\in\bar{Q}_0$.
    On the other hand $\tilde{H}_c(\Gamma)$ is the tensor algebra of the simple quiver underlying $\Gamma$ with modulation
    \begin{align*}
        \tilde{H}_c(i) = \tilde{H}(i) \otimes_L L_c
        &&
        \tilde{H}_c(j,i) = \tilde{H}(j,i) \otimes_L L_c
    \end{align*}
    for $i\in\Gamma_0$ and $(j,i)\in\Omega$.
    Since $c_i \mid c$, 
    we have $\varepsilon_i=\epsilon^{1/c_i}\in L_c$ for every $i\in\Gamma_0$.
    Thus, there are isomorphisms as in (\ref{tensor_product_of_fields})
    \begin{align*}
        \tilde{H}_c(i)
        \cong \prod_{k\in\bZ/c_i} L_c\cdot \bar{e}_{(i,k)}
        &&
        \tilde{H}_c(j,i) 
        \cong \bigoplus_{g=1}^{g_{ji}} \bigoplus_{m \in \bZ/\lcm(c_i,c_j)} 
        L_c\cdot \bar{\alpha}_{m}^{(g)}
    \end{align*}
    with $\Gal(L_c/L)$-action given by
    \begin{align*}
        \sigma_c(\lambda \cdot \bar{e}_{(i,k)}) 
        = \sigma_c(\lambda) \cdot \bar{e}_{(i,{k+1})} 
        &&
        \sigma_c(\lambda \cdot \bar{\alpha}_{m}^{(g)}) 
        = \sigma_c(\lambda) \cdot \bar{\alpha}_{m+1}^{(g)}.
    \end{align*}
    Moreover, the induced multiplication respectively bimodule structures are
    \begin{align*}
        \bar{e}_{(i,k)} \cdot \bar{e}_{(i,l)} 
        = 
        \begin{cases}
            1 & \text{if $k=l$}\\
            0 & \text{else}.
        \end{cases}
        &&
        \bar{e}_{(j,l)} \cdot \bar{\alpha}_m^{(g)} \cdot \bar{e}_{(i,k)} 
        =
        \begin{cases}
            \bar{\alpha}_m^{(g)} & \text{if $(m,m) \equiv (k,l)$ in $\bZ/c_i\times \bZ/c_j$}\\
            0 & \text{else}.
        \end{cases}
    \end{align*}
    Accordingly, we have well defined morphisms of algebras respectively bimodules
    \begin{align*}
        L_c(\bar{Q}_0) \xrightarrow{\iota_0} \tilde{H}_c(\Gamma), ~
        e_{(i,k)} \mapsto \bar{e}_{(i,k)}
        &&
        L_c(\bar{Q}_1) \xrightarrow{\iota_1} \tilde{H}_c(\Gamma), ~
        \alpha^{(g)}_{(j,l)(i,k)} \mapsto \bar{\alpha}_{m}^{(g)}
    \end{align*}
    where $m\in\bZ/\lcm(c_i,c_j)$ is unique 
    with $m \equiv k$ in $\bZ/c_i$ 
    and $m \equiv l$ in $\bZ/c_j$.
    The universal property of tensor algebras \cite[Lemma 1.3]{BSZ09} 
    yields an injective morphism of algebras 
    $\iota\colon L_c(\bar{Q}) \to \tilde{H}_c(\Gamma)$
    which is an isomorphism by comparing dimensions over $L_c$.
\end{proof}

\begin{proposition} \label{prop:unfolded_path_algebra}
    Let $\Gamma$ be a valued quiver.
    There is an isomorphism
    \begin{align*}
        \bar{H}(\Gamma) \cong \bar{L}\bar{Q}(\Gamma)
    \end{align*}
    such that any $\sigma \in \Gal(\bar{L}/L)$ 
    acts on $\bar{L}\bar{Q}(\Gamma)$ by
    \begin{align*}
        \sigma(\lambda\cdot e_{(i,k)}) = \sigma(\lambda) \cdot e_{\sigma(i,k)}, &&
        \sigma(\lambda \cdot \alpha_{(j,l),(i,k)}^{(g)}) = \sigma(\lambda) \cdot \alpha_{\sigma(j,l),\sigma(i,k)},
    \end{align*}
    for $\lambda\in\bar{L}$, $i,j\in Q_0$, $k\in\bZ/c_i$, $l\in\bZ/c_j$, $1\leq g \leq g_{ji}$.
\end{proposition}
\begin{proof}
    This follows readily from Lemma \ref{lem:unfolded_path_algebra}
    passing to the limit $c \to \infty$.
\end{proof}

The $\bar{G}$-action on $\bar{H}$ induces an action on $\mod(\bar{H})$ 
by exact autoequivalences $g\colon \mod(\bar{H}) \xrightarrow{\sim} \mod(\bar{H})$
which further induce linear maps $g\colon \K_0(\bar{H}) \to \K_0(\bar{H})$ for all $g\in \bar{G}$.
This action can be described explicitly:
Set $c := \lcm(c_i \mid i\in\Gamma_0)$ and consider the canonical projection $\rho\colon \bar{G} \twoheadrightarrow \bZ/c$.
The cyclic group $\bZ/c = \ideal{\sigma}$ acts on $\bZ \bar{Q}_0$ by permuting basis vectors
\begin{align}
    \sigma(\be_{i,k}) := \be_{i,k+1}, ~~~~ \forall~ (i,k)\in \bar{Q}_0.
\end{align}

\begin{corollary}
    For each $g\in\Gal(\bar{L}/L)$ there is a commutative square
    \[
        \begin{tikzcd}
            \K_0(\bar{H}) \ar[d,"\dimv"'] \ar[rr,"g"]&& 
            \K_0(\bar{H}) \ar[d,"\dimv"]\\
            \bZ\bar{Q}_0 \ar[rr,"\rho(g)"']&&
            \bZ\bar{Q}_0
        \end{tikzcd}
    \]
    Moreover, for all $\bv,\bw\in\bZ\bar{Q}_0$ and $g\in\bZ/c$ is 
    $\euler{g(\bv),g(\bw)}_{\bar{Q}} = \euler{\bv,\bw}_{\bar{Q}}$.
\end{corollary}

We may supplement the diagram (\ref{reduction_localization})
with the push-forward along $\pi\colon \tilde{H} \to \bar{H}$ 
called \emph{scalar extension}
\begin{align}
    \pi^*:= - \otimes_L\bar{L}\colon \mod(\tilde{H}) \to \mod(\bar{H}).
\end{align}
This is an exact functor, hence induces a linear map on Grothendieck groups
\begin{align}
    \pi^*_0\colon \K_0(\tilde{H}) \to \K_0(\bar{H}).
\end{align}

\begin{corollary} \label{cor:scalar_extension_isometry}
    There is a commutative square
    \[
        \begin{tikzcd}
            \K_0(\tilde{H}) \ar[d,"\dimv"'] \ar[rr,"\pi^*_0"]&& 
            \K_0(\bar{H}) \ar[d,"\dimv"]\\
            \bZ\Gamma_0 \ar[rr,"\Pi^*_0"']&&
            \bZ\bar{Q}_0
        \end{tikzcd}
    \]
    where $\Pi^*_0(\bmalpha_i) := \sum_{k\in\bZ/c_i} \be_{i,k}$ for $i\in\Gamma_0$.
    Moreover, for all $\bv,\bw\in\bZ\Gamma_0$ is
    $\euler{\Pi^*_0(\bv),\Pi^*_0(\bw)}_{\bar{Q}} = \euler{\bv,\bw}_\Gamma$.
\end{corollary}
\begin{proof}
    It is enough to note that the isomorphism (\ref{tensor_product_of_fields}) shows
    \begin{align*}
        \pi^*(S_i) = S_i \otimes_{L} \bar{L} \cong \bigoplus_{k\in\bZ/c_i} S_{i,k}
    \end{align*}
    for the simple $\tilde{H}$-module $S_i$ at $i\in\Gamma_0$
    and the simple $\bar{H}$-modules $S_{i,k}$ at $(i,k) \in \bar{Q}_0$.
    The final assertion follows e.g. from the homological interpretation of Ringel forms
    because $\bar{L}$ is free as an $L$-module. 
    Alternatively, use the explicit definition of the Ringel form (\ref{ringel_form})
    and the short exact sequence (\ref{chinese_remainder_sequence}).
\end{proof}

\subsection{Locally free modules} \label{sec:lf}

Let $\cH := \cH(\Gamma)$ for any of the GLS modulations $\cH \in \{H, \hat{H}, \tilde{H}\}$.
An $\cH$-module $V$ is said to be \emph{locally free}
if $V(i)$ is a free $\cH(i)$-module for every $i\in\Gamma_0$.
Consider the full subcategory
\begin{align*}
    \mod_\lf(\cH) := \{V\in\mod(\cH) \mid \text{$V$ is locally free} \} \subseteq \mod(\cH).
\end{align*}
For the GLS species is $\mod_{\lf}(\tilde{H}) = \mod(\tilde{H})$.
The standard resolution shows that $\pdim(V)\leq 1$ for every $V\in\mod_\lf(\cH)$
and any GLS modulation $\cH\in\{H,\hat{H},\tilde{H}\}$.
This characterizes locally free modules over GLS algebras:

\begin{proposition}\cite[Proposition 3.5, Corollary 11.2]{GLS17i} \label{prop:locally_free}
    Let $V\in\mod(H)$. The following are equivalent:
    \begin{enumerate}[label = (\roman*)]
        \item $V$ is locally free;
        \item $\pdim_H(V)\leq 1$;
        \item $\Hom_H(\tau_H^{-1}(V),H) = 0$
    \end{enumerate}
\end{proposition}

In particular, one has for every $k\in\Gamma_0$ 
the \emph{generalized simple} $E_k\in\mod_{\lf}(H)$ defined by
\begin{align*}
    E_k(i) :=
    \begin{cases}
        H(k) & \text{if $i=k$} \\
        0 & \text{else}
    \end{cases}
    &&
    E(j,i) = 0
\end{align*}
for all $i\in\Gamma_0$ and $(j,i)\in\Omega$.
Their classes $\bmalpha_k := [E_k]$ for $k\in\Gamma_0$ in the Grothendieck group $\K_0^{\lf}(H) := \K_0(\mod_\lf(H))$ 
form a basis.
Thus we may identify
\[
    \begin{tikzcd}
        \mod_\lf(H) \ar[d] \ar[dr,"\rkv"]\\
        \K_0^{\lf}(H) \ar[r,"\sim"]&
        \bZ \Gamma_0
    \end{tikzcd}
\]
where $\rkv(V) = (\rk_i(V))_{i\in Q_0}$ is the \emph{rank vector} of $V\in\mod_{\lf}(V)$
with $\rk_i(V)$ the rank of the free $H(i)$-module $V(i)$ for $i\in\Gamma_0$.
Note that we have a commutative square
\[
    \begin{tikzcd}
        \mod_{\lf}(H) \ar[r,hookrightarrow] \ar[d, "\rkv"']&
        \mod(H) \ar[d,"\dimv"]\\
        \bZ\Gamma_0 \ar[r,"D"'] & 
        \bZ Q_0
    \end{tikzcd}
\]
where $D := \Diag(c_i \mid i\in\Gamma_0)$ is the diagonal matrix with $D(\bmalpha_i) = c_i \cdot \be_i$ for $i\in\Gamma_0$.
More generally,
for any $V\in\mod(H)$ we set
\begin{align*}
    \rkv(V) := D^{-1}(\dimv(V)) \in \bR\Gamma_0.
\end{align*}

Similarly to the hereditary case, the Ringel form and Coxeter transformation are 
$\K$-theoretic shadows of the homological Euler form and the Auslander-Reiten translation.
This was already established by Geiß-Leclerc-Schröer,
but we need a slight generalization allowing one argument of the Euler form to be not locally free.

\begin{proposition}\label{prop:homological_ringel_coxeter}
    Let $U\in\mod(H)$ and $V,W\in\mod_{\lf}(H)$ 
    with $W$ indecomposable
    and $\tau_H(W)$ locally free and non-zero.
    Then
    \begin{align*}
        \euler{\rkv(V),\rkv(U)}_\Gamma = \euler{V,U}_H
        &&
        \rkv(\tau_H(W)) = \Phi_\Gamma(\rkv(W)).
    \end{align*}
\end{proposition}
\begin{proof}
    The second equality is entirely \cite[Proposition 11.5]{GLS17i}.
    The first is \cite[Proposition 4.1]{GLS17i} for $U$ locally free.
    But note that $\pdim_H(V)\leq 1$ hence
    \begin{align*}
        \euler{V,-}_H\colon \mod(H) \to \bZ
    \end{align*}
    factors linearly through $\K_0(H)$.
    In particular, for every $i\in Q_0$ is
    \begin{align*}
        \euler{V,S_i}_H = \frac{1}{c_i}\euler{V,E_i}_H = \frac{1}{c_i}\euler{\rkv(V),\rkv(E_i)}_\Gamma = \euler{\rkv(V),\rkv(S_i)}_\Gamma
    \end{align*}
    where the second equality holds by \cite[Proposition 4.1]{GLS17i}.
    This settles the first equality for locally free $V$ and general $U$.
\end{proof}

\subsection{$\tau$-rigid modules} \label{sec:tau_rigid}

In \cite{GLS20} $\tau$-rigid modules over GLS algebras were classified by translating the problem
to the classification of rigid modules over the GLS species $\tilde{H}$.
The latter was solved by Ringel \cite{Rin94} 
based on Crawley-Boevey's \cite{CB93} transitive braid group action on exceptional sequences.
We further observe that the classification of rigid modules over the GLS species $\tilde{H}$ 
can be recovered via Galois descent 
from the classification of rigid modules over the unfolded path algebra $\bar{H}$.
To state the various bijective correspondences we need to introduce some more sets:
\begin{align*}
    \rig(\cH) &:= \{V\in\mod(\cH) \mid \text{$\Ext^1_\cH(V,V) = 0$}\}/\cong, \\
    \rig_\lf(\cH) &:= \{V\in\mod_\lf(\cH) \mid \text{$\Ext^1_\cH(V,V) = 0$}\}/\cong, \\
    \trig(\cH) &:= \{V\in\mod(\cH) \mid \text{$\Hom_{\cH}(V,\tau_{\cH}(V)) = 0$}\}/\cong, \\
    \rig(\bar{H},\bar{G}) &:= \{ V \in \rig(\bar{H}) \mid \text{$g(V)\cong V$ for all $g\in \bar{G}$} \}/\cong.
\end{align*}
where the third is only defined for $\cH \neq \hat{H}$.
Note that $\rig(\cH) = \rig_\lf(\cH) = \trig(\cH)$ 
for $\cH\in \{\tilde{H}, \bar{H}\}$
by heredity.
Moreover, \emph{Demonet's Lemma} \cite[Lemma 6.2]{GLS20} says $\trig(H) = \rig_\lf(H)$
and it is conjectured in loc.cit. that $\rig(H) = \rig_\lf(H)$.
In the upcoming Section \ref{sec:tau_reduced},
we provide a geometric proof of Demonet's Lemma.
An important preliminary step is the fundamental and well known observation
that they are determined by their class in the corresponding Grothendieck group.

\begin{proposition} \label{prop:lf_rigid_class}
    Let $\Gamma$ be a valued quiver 
    and $\cH = \cH(\Gamma)$ 
    for $\cH\in\{H,\hat{H},\tilde{H},\bar{H}\}$.
    Suppose $V,W\in\rig_\lf(\cH)$.
    If $[V] = [W]$ in $\K_0(\cH)$, then $V \cong W$.
\end{proposition}
\begin{proof}
    For path algebras over algebraically closed fields, in particular $\bar{H}$, 
    this is classically known and follows from the irreducibility of varieties of representations.
    More generally, for species like $\tilde{H}$ this follows from e.g. \cite{Rin94}.
    For the GLS algebra $H$ this is \cite[Proposition 3.2]{GLS18ii}.
    Finally, for $\hat{H}$ this is \cite[Proposition 5.11]{GLS20}.
\end{proof}

The key to classify $\tau$-rigid modules over GLS algebras $H$
is to set them in bijective correspondence with 
rigid modules over GLS species $\tilde{H}$.
This is the lower part of the diagram of bijections in the subsequent Theorem \ref{thm:rigid_bijections}
and completely due to Geiß-Leclerc-Schröer.
Our modest contribution only concerns the unfolded algebra $\bar{H}$,
that is the vertical upper right bijection.

\begin{theorem} \label{thm:rigid_bijections}
    Let $\Gamma$ be a valued quiver 
    with associated GLS algebras $\cH = \cH(\Gamma)$ 
    for $\cH\in\{H,\hat{H},\tilde{H},\bar{H}\}$.
    The reduction, localization and scalar extension functors restrict to bijective correspondences
    \[
        \begin{tikzcd}
            &
            \rig(\bar{H},\bar{G})
            \\
            \rig_\lf(\hat{H}) \ar[r, "\iota^*"] \ar[d, "\delta^*"']
            & 
            \rig(\tilde{H}) \ar[u, "\pi^*"'] \\
            \trig(H)
        \end{tikzcd}
    \]
    Moreover, reduction $\delta^*$ and localization $\iota^*$ 
    preserve indecomposability of locally free rigid modules.
\end{theorem}
\begin{proof}
    For the modulations $H, \hat{H}$ and $\tilde{H}$ of $\Gamma$
    this is \cite[Theorem 1.1]{GLS20}.
    Since $\bar{L}$ is a free $L$-module
    and scalar extension is exact,
    it readily follows that the induced map
    $\rig(\tilde{H})\to \rig(\bar{H},\bar{G})$ is well defined and injective.
    For surjectivity let $V\in\rig(\bar{H},\bar{G})$.
    We may assume that $V = V' \otimes_{L'} L$ is defined over some finite field extension $L'/L$
    with $V'$ a module over $\tilde{H} \otimes_L L'$.
    Furthermore, by Proposition \ref{prop:unfolded_path_algebra} 
    we may assume that $\tilde{H}\otimes_L L' \cong L'\bar{Q}$.
    Let $G' := \Gal(L'/L)$ which is cyclic with generator denoted $\sigma$.
    We need to find for every $g\in G'$ an isomorphism $f_g:V'\xrightarrow{\sim} g(V')$ 
    such that $f_{gh} = g(f_h)\circ f_g$ for all $g,h\in G'$.
    In fact, it suffices to find an isomorphism $f\colon V' \xrightarrow{\sim} \sigma(V')$ 
    such that $\id_{V'} = \sigma^{m-1}(f)\cdots \sigma(f) f$.
    By assumption we have an isomorphism $f\colon V' \xrightarrow{\sim} \sigma(V')$.
    According to Crawley-Boevey \cite{CB96} the module $V'$ is defined over $\bZ\subseteq L'$.
    Therefore we may assume that the isomorphism $f\colon V' \xrightarrow{\sim} \sigma(V')$ 
    is defined over $K\subseteq L'$ as well.
    Thus we may apply a construction due to Gabriel \cite[Section 3.9]{Gab81} to modify $f$ if necessary:
    Say $\sigma^{m-1}(f)\cdots \sigma(f) f = w \neq \id_{V'}$. 
    One may choose $a\in K[w]$ 
    and define a new isomorphism $f'=fa$
    which then satisfies $\sigma^{m-1}(f')\cdots \sigma(f') f' = w a^m$
    because $\sigma(w) f = f w$.
    Therefore, we take $a$ to be an $m$-th root of $w^{-1}$ in $K[w]$
    which is possible because $K$ is algebraically closed of characteristic $0$.
\end{proof}

These bijective correspondences allow Gei{\ss}-Leclerc-Schröer 
to draw further consequences for the structure of $\tau$-rigid $H$-modules.

\begin{proposition} \cite[Theorem 1.2(b)]{GLS20} \label{prop:tau_rigid_structure}
    Let $V\in\mod(H)$ be indecomposable $\tau$-rigid.
    Then $V$ is free as a $\End_H(V)$-module and
    \begin{align*}
        \End_H(V) \cong K[\delta]/\ideal{\delta^{c_i}}
    \end{align*}
    where $c_i = q_\Gamma(\rkv(V))$ for some $i\in\Gamma_0$.
\end{proposition}

\subsection{Generically $\tau$-reduced components} \label{sec:tau_reduced}

Thanks to Geiß-Leclerc-Schröer's explicit presentation of $H$ 
in terms of the ordinary quiver $Q$ and relations $I$
from Proposition \ref{prop:quiver_with_relations},
we have a variety $\Rep(H,\bd)$ and a scheme $\bRep(H,\bd)$ of representations 
for each dimension vector $\bd\in\bN Q_0$.
In the geometric representation theory of finite-dimensional algebras $A \cong KQ/I$,
an important tool is \emph{Voigt's Isomorphism}
\begin{align} \label{voigt}
    \T_V(\bRep(A,\bd))/\T_V(\cO(V)) \xrightarrow{\sim} \Ext^1_A(V,V)
\end{align}
valid for all $V\in\Rep(H,\bd)$ and $\bd\in \K_0(A)^+$.
We write $T_x(\cX)$ for the tangent space of a scheme $\cX$ at a closed point $x\in\cX$.
Further, recall that orbits are always smooth hence
\begin{align} \label{orbit_dimension}
    \dim \T_V(\cO(V)) = \dim \cO(V) = \dim \GL(K,\bd) - \dim\End_A(V).
\end{align}

\begin{lemma} \label{lem:local_free_smooth}
    Let $n,d\geq 1$
    and consider $H := K[x]/\ideal{x^n}$ the truncated polynomial ring.
    The following are equivalent for $V\in\Rep(H,d)$:
    \begin{enumerate}[label = (\roman*)]
        \item \label{enum:local_free_module}
        $V$ is a free module in $\mod(H)$;
        \item \label{enum:local_smooth_point}
        $V$ is a smooth point of $\bRep(H,d)$.
    \end{enumerate}
\end{lemma}
\begin{proof}
    The implication $\ref{enum:local_free_module} \Rightarrow \ref{enum:local_smooth_point}$
    is true in general by Voigt's Isomorphism.
    Let's prove $\ref{enum:local_smooth_point} \Rightarrow \ref{enum:local_free_module}$:
    The indecomposable $H$-modules are up to isomorphism given by $N_k:=H/\ideal{x^k}$ for $k=1,\dots, n$. 
    Write any $d\in\mathbb{Z}_{\geq 0}$ as $d=r n + k$ 
    for unique $r\geq 0$ and $0\leq k < n$. 
    Define $N_d:=N_n^r\oplus N_k$ where $N_0:=0$ 
    and no ambiguity arises as this definition coincides with our previous notation for indecomposable $H$-modules for $1\leq d < n$. 
    Then the orbit $\mathcal{O}(N_d)$ is open and dense in $\Rep(H,d)$. 
    Therefore, the reduced variety $\Rep(H,d)$ is irreducible of dimension
    \begin{align*}
        \dim \Rep(H,d) = \dim \GL(K,d) - \dimend_H(N_d).
    \end{align*}
    In particular, $N_d$ is smooth if and only if it is reduced.
    By Voigt's Isomorphism, 
    $N_d$ is smooth if and only if
    \begin{align*}
    \dimext^1_H(N_d,N_d) - \dimend_H(N_d) + \dim \GL(K,d) = \dim \GL(K,d) - \dimend_H(N_d)
    \end{align*}
    equivalently $0 = \Ext^1_H(N_k,N_k)$ hence $k=0$.
    To conclude, recall that the smooth points of a finite type scheme form an open subscheme. 
    By our previous discussion,
    $\bRep(H,d)$ contains a smooth point if and only if $k=0$ i.e. $d=rn$ for some $r\geq 0$. 
    Now any $V\in\Rep(H,d)$ is isomorphic to a direct sum
    \begin{align*}
        V\cong N_n^{r_n}\oplus \cdots\oplus N_1^{r_1}
    \end{align*}
    for some $r_n,\dots,r_1\geq 0$ with $\sum_{k=1}^n kr_k=d$. 
    As before the module $V$ is smooth if and only if
    \begin{align*}
        0
        &= \dimend_H(N_d)-\dimend_H(V)+\dimext_H^1(V,V) \\
        &= n r^2-\sum_{1\leq k,l\leq n} \max\{k+l-n,0\} r_k r_l \\
        &= \sum_{1\leq k,l\leq n} \underbrace{\left(\frac{k l}{n}-\max\{k+l-n,0\}\right)}_{c_{k,l}} r_k r_l
    \end{align*}
    where the coefficients satisfy $c_{k,l}\geq 0$ 
    with equality if and only if $k$ or $l$ equals $n$. 
    This proves that $V$ can only be smooth if it is isomorphic to $N_n^r$.
\end{proof}

In \cite{LFZ16} a valued quiver $\Gamma$ is called \emph{strongly primitive} 
if $\gcd(c_i,c_j)=1$ for all $i,j\in\Gamma_0$.
Note that the symmetrizer of a strongly primitive valued quiver $\Gamma$ is necessarily minimal.
These are precisely those valued quivers for which the commutativity relations of $H(\Gamma)$ are redundant.
The affine non-simply laced strongly primitive valued quivers are those of type
$\tilde{B}_2, \tilde{C}_n, \tilde{\BC}_{1}, \tilde{\CD}_{n}$ and $\tilde{G}_{2,3}$.

\begin{proposition}\label{prop:strongly_primitive_lf_smooth_reduced}
    Assume that $\Gamma$ is strongly primitive. 
    Let $H = H(\Gamma)$ be its GLS algebra
    and $\bd\in\bN Q_0$.
    Then the following are equivalent for a representation $V\in\Rep(H,\bd)$:
    \begin{enumerate}[label = (\roman*)]
        \item $V$ is a locally free module in $\mod(H)$;
        \item $V$ is a smooth point of $\bRep(H,\bd)$.
    \end{enumerate}
\end{proposition}
\begin{proof}
    Since $\Gamma$ is assumed to be strongly primitive,
    we have a product decomposition:
    \begin{align*}
        \bRep(H,\bd) = \prod_{i \in Q_0} \bRep(H(i),d_i) \times \prod_{a\in Q^\circ_1} \bRep(K\Lambda,(d_{t(a)},d_{s(a)}))
    \end{align*}
    where $\Lambda\colon 1 \to 0$ and $Q^\circ$ is the quiver $Q$ without the loops.
    It is well-known 
    that a point $V$ of the product is smooth 
    if and only if all of its factors are smooth. 
    Note that the factors $\bRep(K\Lambda,(d_{t(a)},d_{s(a)})) = \bA^{d_{t(a)} d_{s(a)}}$ are just affine spaces
    hence smooth for all $a\in Q_1^\circ$.
    Now, our claim is a direct consequence of Lemma \ref{lem:local_free_smooth}.
\end{proof}

This proposition provides a geometric interpretation of local freeness for strongly primitive $\Gamma$.
We conjecture that one may drop the assumption on $\Gamma$.
For general $\Gamma$, we will now present an interpretation of local freeness in terms of so called generically $\tau$-reduced components. 
For a general finite-dimensional algebra $A$ and an irreducible component $\cZ\in\Irr(A)$,
Voigt's Isomorphism and Auslander-Reiten Duality yield estimates for the generic number of parameters of $\cZ$
\begin{align*}
    c_A(\cZ) \leq \dimext^1_A(\cZ) \leq \dimhom^\tau_A(\cZ)
\end{align*}
where $\dimext^1_A(\cZ) := \min\{\dimext^1_A(V,V) \mid V\in\cZ\}$ is the generic dimension of selfextensions of elements of $\cZ$.
In particular, $\cZ$ is generically reduced if $c_A(\cZ) = \dimext^1_A(\cZ)$.
Geiß-Leclerc-Schröer consider in \cite{GLS12} components $\cZ$ satisfying $c_A(\cZ) = \dimhom^\tau_A(\cZ)$
and these are now called \emph{generically $\tau$-reduced components}.
Let $\Irr^\tau(A)\subseteq\Irr(A)$ be the subset of generically $\tau$-reduced components.
Generically $\tau$-reduced components naturally generalize $\tau$-rigid modules.
Recall that $\tau$-rigid modules are determined by their $\g$-vectors defined in (\ref{g_vector}).
Plamondon proves that generically $\tau$-reduced components are determined by their \emph{generic $\g$-vectors};
namely, for any $\cZ\in\Irr(A)$, there is a well defined $\g(\cZ)\in\K_0^{\fin}(A)$
such that the generic element $V\in\cZ$ has $\g$-vector $\g(V) = \g(\cZ)$.

\begin{theorem}\cite[Theorem 1.2]{Pla13}\label{thm:plamondon}
    There is a left inverse to the assignment of generic $\g$-vectors
    \[
        \begin{tikzcd}
            \Irr^\tau(A)
            \ar[rr, shift left = 1mm, "\g"]
            &&
            \ar[ll, shift left = 1mm, "\Psi"]
            \K_0^{\fin}(A)
        \end{tikzcd}
    \]
    such that for all $\cZ\in\Irr^\tau(A)$ and 
    $S\in\proj(A)$ with $\dimhom_A(S,\cZ) = 0$ 
    is $\cZ = \Psi(\g(\cZ) - \g(S))$.
\end{theorem}

Over GLS algebras we only need to compare $\g$-vectors and rank vectors to obtain:

\begin{theorem} \label{thm:tau_reduced_locally_free}
    Let $\Gamma$ be a valued quiver
    with GLS algebra $H = H(\Gamma)$.
    For an irreducible component $\cZ\in\Irr(H)$ the following statements are equivalent:
    \begin{enumerate}[label = (\roman*)]
        \item \label{enum:gen_loc_fr} 
        $\cZ$ is generically locally free;
        \item \label{enum:gen_tau_red} 
        $\cZ$ is generically $\tau$-reduced.
    \end{enumerate}
    Moreover, if $\Gamma$ is strongly primitive, then there is are further equivalences with
    \begin{enumerate}[resume, label = (\roman*)]
        \item \label{enum:gen_reduced}
        $\cZ$ is generically reduced.
    \end{enumerate}
\end{theorem}
\begin{proof}
    $"\ref{enum:gen_loc_fr} \Rightarrow \ref{enum:gen_tau_red}":$ 
    Suppose $\cZ$ is generically locally free.
    Let then $\cZ_{\lf}\subseteq \cZ$ be the open subset of locally free representations.
    The functions $c_A$ and 
    $\dimhom_A(-,\tau_A(?))$ are well known to be upper semicontinuous 
    for general finite-dimensional algebras $A$,
    see e.g. \cite{GLFS23}.
    Thus, let $\cU\subseteq \cZ$ be the open subset of representations $V$ satisfying both:
    \begin{align*}
        c_H(\cZ) = \dim(\cZ) - \dim \cO(V),
        &&
        \dimhom^\tau_H(\cZ) = \dimhom_H(V,\tau_H(V)).
    \end{align*}
    Since $\cZ$ is irreducible, 
    the intersection $\cU_\lf:=\cU\cap \cZ_{\lf}$ is open and dense in $\cZ$. 
    Observe for any $V\in \cU_\lf$ the chain of equalities
    \begin{align*}
        c_H(\cZ)  = \dim(\cZ) - \dim(\cO(V))
                = \dimext_H^1(V,V)
                = \dimhom_H(V,\tau_H(V))
                = \dimhom^\tau_H(\cZ)
    \end{align*}
    the first and final equality hold because $V\in \cU_\lf$, 
    the two in between follow from $\pdim(V)\leq 1$
    together with Voigt's isomorphism respectively Auslander-Reiten's formula.
    This shows that $\cZ$ is generically $\tau$-reduced.

    $"\ref{enum:gen_tau_red} \Rightarrow \ref{enum:gen_loc_fr}":$
    Suppose $\cZ$ is generically $\tau$-reduced.
    Let $\g(\cZ) \in\bZ \Gamma_0$ be its generic $\g$-vector.
    Apply the following claim
    to $\bmgamma = \g(\cZ)$:
    \begin{claim}\label{claim:positive_g_vector_decomposition}
        For any $\bmgamma\in \bZ \Gamma_0$ 
        exist unique subsets $I,J\subseteq \Gamma_0$ with $I \cap J = \emptyset$
        and positive coefficients $r_i,s_j\in\bN_{>0}$ for $i\in I$ and $j\in J$
        such that
        \begin{align*}
            \bmgamma = \sum_{i\in I} r_i \cdot \g({E_i}) - \sum_{j\in J} s_j \cdot \g({P_j})
        \end{align*}
    \end{claim}
    \begin{proof}[Proof of Claim]
        We can order the vertices $\Gamma_0 = \{1,\dots,n\}$ such that
        $\Hom_H(P_i,P_j) = 0$ for all $i > j$.
        Let $i_0 := \min\{i\in Q_0 \mid \gamma_i \neq 0\}$.
        If $i_0 = 1$, we are done because $E_1 = P_1$.
        If $i_0 > 1$, consider
        \begin{align*}
            \bmgamma' :=
            \begin{cases}
                \bmgamma + \gamma_{i_0} \cdot \g(P_{i_0}) & \text{if $\gamma_{i_0} < 0$,} \\
                \bmgamma - \gamma_{i_0} \cdot \g(E_{i_0}) & \text{if $\gamma_{i_0} > 0$.}
            \end{cases}
        \end{align*}
        Then $\gamma'_j = 0$ for all $j \geq i_0$ 
        because $\g(E_{i_0})_j = 0$ for $j > i_0$ and $\g(E_{i_0})_{i_0} = 1$.
        By induction, there are unique subsets $I',J'\subseteq \{1,\dots,i_0-1\}$ 
        with $I'\cap J' = \emptyset$ and coefficients $r'_i,s'_j\in\bN_{> 0}$ for $i\in I'$ and $j\in J'$
        such that
        \begin{align*}
            \bmgamma' = \sum_{i\in I'} r'_i \cdot \g(E_i) - \sum_{j\in J'} s'_j \cdot \g(P_j).
        \end{align*}
        This proves the claim.
    \end{proof}
    Consider the generically locally free component $\cZ(\br)$ with generic rank vector $\br\in\bN\Gamma_0$
    whose entries are $r_i$ for $i\in I$ and $0$ else.
    Further, take $S\in\proj(A)$ to be the unique projective with $\g(S) = \bs$ where $\bs\in\bN\Gamma_0$
    has entries $s_j$ for $j\in J$ and $0$ else.
    Then $\cZ(\br)$ is generically $\tau$-reduced by the first part of the proof 
    and $\g(\cZ) = \g(\cZ(\br)) - \g(S)$ with $\dimhom_H(S,\cZ(\br)) = 0$ by construction.
    Therefore, Plamondon's Theorem \ref{thm:plamondon} allows to conclude 
    \begin{align*}
        \cZ = \Psi(\g(\cZ)) = \Psi(\g(\cZ(\br)) - \g(S)) = \cZ(\br)
    \end{align*}
    The further equivalence
    $\ref{enum:gen_loc_fr} \Leftrightarrow \ref{enum:gen_reduced}$
    for strongly primitive $\Gamma$ 
    is now a direct consequence of Proposition \ref{prop:strongly_primitive_lf_smooth_reduced}
    and the fact that generically reduced components are generically smooth.
\end{proof}

From our geometric considerations we obtain another proof of Demonet's Lemma \cite[Lemma 6.2]{GLS20}, 
that is the equivalence 
$\ref{enum:lf_rigid} \Leftrightarrow \ref{enum:tau_rigid}$ 
in the subsequent corollary.
The moreover part for strongly primitive $\Gamma$ 
also follows from Demonet's more constructive proof,
even though it was not stated explicitly in \cite{GLS20}.
We should also acknowledge that Demonet's proof works over arbitrary fields,
while we always work over an algebraically closed field $K$.

\begin{corollary} \label{cor:demonet_lemma}
    For a module $V\in\mod(H)$ the following are equivalent
    \begin{enumerate}[label = (\roman*)]
        \item \label{enum:lf_rigid}
        $V$ is locally free and rigid;
        \item \label{enum:tau_rigid}
        $V$ is $\tau$-rigid.
    \end{enumerate}
    Moreover, if $\Gamma$ is strongly primitive, 
    then every rigid $V$ is locally free.
\end{corollary}
\begin{proof}
    The implication $\ref{enum:lf_rigid} \Rightarrow \ref{enum:tau_rigid}$ 
    follows from Auslander-Reiten Duality. 
    Conversely, if $V$ is $\tau$-rigid, 
    then $\cZ:=\overline{\cO(V)}$ is a generically $\tau$-reduced component. 
    By Theorem \ref{thm:tau_reduced_locally_free}, 
    $\cZ$ is generically locally free, 
    thus $V$ is locally free. 
    This proves $\ref{enum:tau_rigid} \Rightarrow \ref{enum:lf_rigid}$. 

    For the moreover part,
    assume that $V$ is rigid.
    Then $\overline{\cO(V)}$ is a generically reduced component
    by Voigt's Isomorphism.
    As before $V$ is then already locally free by Theorem \ref{thm:tau_reduced_locally_free}.
\end{proof}

This shows that GLS algebras are still ``hereditary'' from a $\tau$-tilting perspective.
Note that GLS algebras are quite well designed for this to hold.
One easily finds examples of $1$-Iwanaga-Gorenstein algebras with invertible Cartan matrices
but ``non-hereditary'' $\tau$-tilting theory:

\begin{example}
    Consider the algebra $A = KQ/I$ given by the quiver
    \[
        \begin{tikzcd}
            Q \colon &
            1 \ar[loop, in = 150, out = -150, distance = 5ex, "a"] \ar[r, shift left = 1mm, "b"] &
            2 \ar[l, shift left = 1mm, "c"]
        \end{tikzcd}
    \]
    and the ideal $I$ is generated by paths of length $2$ in $Q$.
    This is a representation finite selfinjective string algebra.
    Its Cartan matrix is $C_A = \rsm{2&1\\1&1}$ with $\det(C_A) = 1$.
    Note that the simple module $S_2$ at vertex $2$ 
    is $\tau$-rigid with $\pdim(S_2) = \infty$.
\end{example}

It is asked in \cite{GLS20} whether one may drop the assumption 
that $\Gamma$ is strongly primitive in Corollary \ref{cor:demonet_lemma}.
More generally, we conjecture that one can drop this assumption 
in Proposition \ref{prop:strongly_primitive_lf_smooth_reduced}
hence also in Theorem \ref{thm:tau_reduced_locally_free}.

\section{Affine valued quivers} \label{sec:affine}

Finally, we arrive at the main section of the present work concering valued quivers $\Gamma$ of affine type.
A valued quiver $\Gamma$ is of \emph{finite type} if $q_\Gamma$ is positive definite.
A valued quiver $\Gamma$ is of \emph{affine type} if it is not of finite type
but $q_\Gamma$ is positive semidefinite.
The connected valued quivers $\Gamma$ of affine type 
are precisely those with underlying valued graph listed in Figure \ref{tab:affine},
in particular the type does not depend on the symmetrizer or orientation.
The names for the affine types vary in the literature (see e.g. \cite{Moo68}, \cite{Mac71}, \cite{DR76}, \cite{K90}).
From now on, 
we fix a connected valued quiver $\Gamma = (\Gamma_0,\Gamma_1,\Omega,\nu,\bc)$ of affine type
and set $\cH := \cH(\Gamma)$ for any of the GLS modulations $\cH \in \{H,\hat{H},\tilde{H},\bar{H}\}$.
We start with a brief summary of the structure of affine root systems following \cite{DR76} in Section \ref{sec:affine_roots}.
Then we present in Section \ref{sec:extended_type_bc1} a detailed study of the representation theory of GLS algebras of type $\tilde{\BC}_1$.
This is fundamental for the general construction of a $1$-parameter family of stable $H$-modules in Section \ref{sec:stable_family}.
In higher ranks, we need to understand how our family of stable modules interacts with regular $\tau$-rigid modules.
To this end, we study their structure in Section \ref{sec:regular_tau_rigid}.
Section \ref{sec:generic_classification} contains the proof of our Main Theorem \ref{thm:generic_classification}.
Some consequences for recent $\tau$-tilted notions of tameness are obtained in Section \ref{sec:tau_tame}.

\subsection{Affine root systems} \label{sec:affine_roots}

The structure of the root system $\Delta(\Gamma)$ 
for (connected) affine $\Gamma$ 
is well understood 
(see \cite[Section 1]{M69}, \cite[Section 1]{DR76} and \cite[Chapters 4.8, 6]{K90}).
There exists a unique non-zero $\bmeta_\Gamma\in\bN\Gamma_0$ up to scaling,
the minimal such called \emph{primitive null root},
such that one, and then any, of the following holds
\begin{enumerate}[label = (\roman*)]
    \item the vector $\bmeta_\Gamma$ is \emph{isotropic} i.e. $q_\Gamma(\bmeta_\Gamma) = 0$;
    \item the vector $\bmeta_\Gamma$ is \emph{radical} i.e. $(\bmeta_\Gamma,\bw)_\Gamma = 0$ for all $\bw\in\bZ\Gamma_0$;
    \item the vector $\bmeta_\Gamma$ is \emph{homogeneous} i.e. $\Phi_\Gamma(\bmeta_\Gamma) = \bmeta_\Gamma$.
\end{enumerate}
The \emph{defect} $\partial_\Gamma \in \bZ\Gamma_0^*$ is defined as
$\partial_\Gamma(\bv) := \euler{\bmeta_\Gamma, \bv}_\Gamma$ for $\bv\in\bZ\Gamma_0$.
Note that we follow a different scaling convention for the defect than in \cite[p. 11]{DR76}.
With its help it is easy to decide whether a given positive root is preprojective, regular or preinjective:
\begin{align*}
    \Delta_{\cP}(\Gamma) &= \{\bv\in\Delta^+(\Gamma) \mid \partial_\Gamma(\bv) < 0\}, \\
    \Delta_{\cR}(\Gamma) &= \{\bv\in\Delta^+(\Gamma) \mid \partial_\Gamma(\bv) = 0\}, \\
    \Delta_{\cI}(\Gamma) &= \{\bv\in\Delta^+(\Gamma) \mid \partial_\Gamma(\bv) > 0\}.
\end{align*}
Following \cite[p. 13]{DR76}, define the set $\Delta_{\cR}^{\simp}(\Gamma)$
of \emph{quasi-simple regular real roots} as those regular real roots 
which are minimal with respect to the partial order 
$\bv\leq_\Phi \bw$ if $\Phi_\Gamma^n(\bv)\leq \Phi_\Gamma^n(\bw)$ for all $n\in\bZ$.
The Coxeter transformation $\Phi_\Gamma$ permutes $\Delta_{\cR}^{\simp}(\Gamma)$
and decomposes it into at most three $\Phi$-orbits 
denoted $\Delta_{\cR}^{\lambda}(\Gamma)$ for $\lambda\in I_\Gamma$
where $I_\Gamma \in \{ \{\infty\}, \{0,\infty\}, \{0,1,\infty\}\}$
depending on the number of $\Phi$-orbits.
For $\lambda\in I_\Gamma$ 
let $r_\lambda := \norm{\Delta_{\cR}^\lambda(\Gamma)}$
and choose a $\bv_{\lambda,0} \in \Delta_{\cR}^\lambda(\Gamma)$
then
\begin{align*}
    \Delta_{\cR}^{\lambda}(\Gamma) = \{\bv_{\lambda,k} := \Phi_\Gamma^k (\bv_{\lambda,0}) \mid k\in\bZ/r_\lambda\}.
\end{align*}
Define inductively $\bv_{\lambda,k}^{(0)} := 0$ 
and $\bv_{\lambda,k}^{(l)} := \bv_{\lambda,k}^{(l-1)} + \bv_{\lambda,k-l+1}$ 
for $\lambda\in I_\Gamma$, $l\geq 1$ and $k\in\bZ/r_\lambda$.
Following the induction scheme, we may organize the just defined roots in tubes
(see Figure \ref{fig:tube_4} for an example with $r_\lambda = 4$).
We have now a very explicit description of all regular roots appearing in \cite[Proposition 1.9.(b)(3)]{DR76}
\begin{align*}
    \Delta_\cR(\Gamma)
    = \{\bv_{\lambda,k}^{(l)} \mid \text{$\lambda\in I_\Gamma, k\in\bZ/r_\lambda$ and $l\in\bN\setminus \bN r_\lambda$}\} \cup \bN \bmeta_\Gamma
\end{align*}
Note that 
\begin{align*}
    \bv_{\lambda,k}^{(r_\lambda)} = \sum_{k\in\bZ/r_\lambda} \bv_{\lambda,k} = t_\lambda \bmeta_\Gamma
    &&
    t_\lambda := \euler{\bv_{\lambda,k}^{(l)},\bv_{\lambda,k}^{(l)}}_\Gamma
\end{align*} 
for $\lambda\in I_\Gamma$
where $t_\lambda$ is independent of the choice of $k\in\bZ/r_\lambda$ and $l\in\bN\setminus \bN r_\lambda$.
The \emph{tier number} of $\Gamma$ is $t_\Gamma := \max\{t_\lambda \mid \lambda\in I_\Gamma\}$.
This number was introduced in \cite{M69} and can be found in Table \ref{tab:affine}.

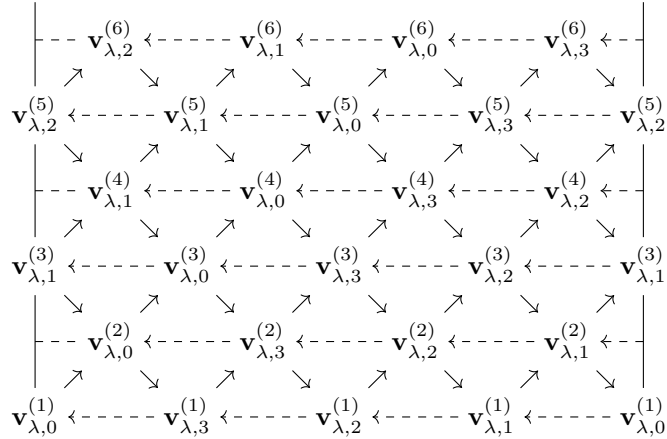
\begin{figure}[h]
    \centering
    \begin{tikzpicture}
        \foreach \x in {0,2,...,8}
        \foreach \y in {0,2,...,4}
            \pgfmathtruncatemacro{\l}{\y+1}
            \pgfmathtruncatemacro{\s}{mod(4-\x/2+\y/2,4)}
            \node (\x\y) at (\x,\y) {$\bv_{\lambda,\s}^{(\l)}$};
        \foreach \x in {1,3,...,7}
        \foreach \y in {1,3,...,5}
            \pgfmathtruncatemacro{\l}{\y+1}
            \pgfmathtruncatemacro{\s}{mod(3-(\x-1)/2+(\y+1)/2,4)}
            \node (\x\y) at (\x,\y) {$\bv_{\lambda,\s}^{(\l)}$};
        \foreach \x in {0,2,...,6}
        \foreach \y in {0,2,...,4}
            \pgfmathtruncatemacro{\a}{\x+1}
            \pgfmathtruncatemacro{\b}{\y+1}
            \draw[->] (\x\y) -- (\a\b);
        \foreach \x in {1,3,...,7}
        \foreach \y in {1,3}
            \pgfmathtruncatemacro{\a}{\x+1}
            \pgfmathtruncatemacro{\b}{\y+1}
            \draw[->] (\x\y) -- (\a\b);
        \foreach \x in {0,2,...,6}
        \foreach \y in {2,4}
            \pgfmathtruncatemacro{\a}{\x+1}
            \pgfmathtruncatemacro{\b}{\y-1}
            \draw[->] (\x\y) -- (\a\b);
        \foreach \x in {1,3,...,7}
        \foreach \y in {1,3,...,5}
            \pgfmathtruncatemacro{\a}{\x+1}
            \pgfmathtruncatemacro{\b}{\y-1}
            \draw[->] (\x\y) -- (\a\b);
        \foreach \x in {2,4,...,8}
        \foreach \y in {0,2,...,4}
            \pgfmathtruncatemacro{\a}{\x-2}
            \draw[->,dashed] (\x\y) --  (\a\y);
        \foreach \x in {3,5,...,7}
        \foreach \y in {1,3,...,5}
            \pgfmathtruncatemacro{\a}{\x-2}
            \draw[->,dashed] (\x\y) --  (\a\y);
        \foreach \y in {1,3,...,5}
            \draw[-,dashed] (0,\y) -- (1\y);
        \foreach \y in {1,3,...,5}
            \draw[->,dashed] (8,\y) -- (7\y);
        \draw[-] (00) -- (02) -- (04) -- (0,5.5);
        \draw[-] (80) -- (82) -- (84) -- (8,5.5);
    \end{tikzpicture}
    \caption{
        A tube of rank $4$. The left and right boundary should be identified.
        The dashed arrows represent the action of $\Phi_\Gamma$ on roots.
        The solid arrows represent irreducible morphisms
        between corresponding indecomposable modules over the species $\tilde{H}(\Gamma)$.
        For a fixed $l\geq 1$, 
        the roots $\bv_{\lambda,k}^{(4l)}$ coincide for all $k\in\bZ/4$,
        but we still list one for each $k$ 
        to maintain the shape of a tube;
        and there actually are corresponding pairwise non-isomorphic modules over species $\tilde{H}(\Gamma)$
        (see \cite[Main Theorem]{DR76}).
    }
    \label{fig:tube_4}
\end{figure}

\subsection{Affine type $\widetilde{\BC}_1$} \label{sec:extended_type_bc1}

Consider the following valued quiver $\Gamma$ of affine type $\tilde{\BC}_{1}$:
\[
    \begin{tikzcd}
        \Gamma \colon ~ 
        2 \ar[r,"1 \mid 4"] & 1
    \end{tikzcd}
\]
with minimal symmetrizer $D = \rsm{4 & 0 \\ 0 & 1}$.
The associated GLS algebra is $H = KQ/I$ with quiver and relations
\[
    \begin{tikzcd}
        Q\colon ~
        2 \ar[r, "\alpha"] &
        1 \ar[loop, out=30, in=-30, distance=5ex, "\varepsilon"] &&&&
        I := \ideal{\varepsilon^4}.
    \end{tikzcd}
\]
It is very well known that $H$ has wild representation type. 
Indeed it's Galois covering has a convex hypercritical subcategory
of twice extended type $\tilde{\tilde{D}}_5$ highlighted in red:
\[
    \begin{tikzcd}
        \cdots &
        2_0 \arrow[d, "\alpha_0"] &
        \red{2_1} \arrow[red, d, "\alpha_1"] &
        \red{2_2} \arrow[red, d, "\alpha_2"] &
        \red{2_3} \arrow[red, d, "\alpha_3"] &
        2_4 \arrow[d, "\alpha_4"] &
        2_5 \arrow[d, "\alpha_5"] &
        2_6 \arrow[d, "\alpha_6"] &
        \cdots \\
        \cdots &
        1_0 \arrow[r,"\varepsilon_0"] & 
        \red{1_1} \arrow[red, r, "\varepsilon_1"] 
        \arrow[-, dotted, rrrr, shift right =.2cm, bend right=.5cm] &
        \red{1_2} \arrow[red, r, "\varepsilon_2"] &
        \red{1_3} \arrow[red, r, "\varepsilon_3"] &
        \red{1_4} \arrow[r, "\varepsilon_4"] &
        1_5 \arrow[r, "\varepsilon_5"] &
        1_6 &
        \cdots
    \end{tikzcd}
\]
The Ringel form, Tits form and Coxeter transformation of $\Gamma$ are
\begin{align*}
    \euler{\bv,\bw}=4v_1w_1+v_2w_2-4v_2w_1 && 
    q(\bv) = (2v_1 - v_2)^2&& 
     \Phi = \left(\bsm -1 & 1 \\ -4 & 3\esm\right)
\end{align*}
Therefore, the primitive null root is $\bmeta = (1,2)$
and the preprojective and preinjective roots are
\begin{align*}
    \bp_1^{(n)}&:=(2n+1,4n) &&& \bq_1^{(n)}&:=(2n+1,4n+4) &&&
    \bp_2^{(n)}&:=(n+1,2n+1) &&& \bq_2^{(n)}&:=(n,2n+1)
\end{align*}
for $n\geq 0$ and these are all positive real roots. 
The corresponding indecomposable $\tau$-rigid $H$-modules are
\begin{align*}
    P_1^{(n)}& = \tau_H^{-n} (P_1) &&& I_1^{(n)}& = \tau_H^{n} (I_1) &&&
    P_2^{(n)}& = \tau_H^{-n} (P_2) &&& I_2^{(n)}& = \tau_H^{n} (I_2).
\end{align*}
for $n\geq 0$ with corresponding rank and $\g$-vectors
\begin{align*}
    \rkv(P_1^{(n)}) &= \bp_1^{(n)}, &&& \g(P_1^{(n)})& = (-4n, 8n+4), &&& 
    \rkv(I_1^{(n)}) &= \bq_1^{(n)}, &&& \g({I_1^{(n)}})& = (-4n-4, 8n+4), \\
    \rkv(P_2^{(n)}) &= \bp_2^{(n)}, &&& \g({P_2^{(n)}})& = (-2n-1, 4n+4), &&& 
    \rkv(I_2^{(n)}) &= \bq_2^{(n)}, &&& \g({I_2^{(n)}})& = (-2n-1, 4n).
\end{align*}
Part of the support $\tau$-tilting exchange graph 
and the $\g$-vector fan
are shown in Figure \ref{fig:exchange_quiver_and_fan}.
The red ray $\bR_{\geq 0}\partial$ spanned by the defect $\partial = (-1,2)\in\bR^2$ is the only wall which is not part of the $\g$-vector fan,
but it is the limit of walls in the $\g$-vector fan.
One immediately sees that the algebra $H$ is $\g$-tame.
For a complete classification of all stable $H$-modules
it remains to consider those stable with respect to the defect.
Explicitly, $V\in\mod(H)$ is $\partial$-semistable if and only if
$\dim(V(1)) = 2\cdot\dim(V(2))$ and
\begin{align} \label{bc1_defect_semistable}
    2\cdot \dim(U(2)) \leq \dim(U(1))
\end{align}
for all proper and non-zero submodules $U\subset V$.
A single $\partial$-stable representation is easy to find:
\[
    \begin{tikzcd}[ampersand replacement=\&]
        \bar{V}_{\infty}\colon 
        ~
        K \ar[r,"\rsm{1\\0}"] 
        \&
        K^2 
        \ar[loop, out=30, in=-30, distance=5ex, "\rsm{0 & 0 \\ 1 & 0}"]
    \end{tikzcd}
\]
Furthermore, given $\lambda\in K^2\setminus 0$ define the representation $V_\lambda$ as 
\[
    \begin{tikzcd}[ampersand replacement=\&]
        V_\lambda\colon 
        ~ 
        K^2 
        \ar[rr,"\rsm{1 & 0 \\ 0 & \lambda_2 \\ 0 & \lambda_1 \\ 0 & 0}"] 
        \&\&
        K^4 \ar[loop, out=30, in=-30, distance=5ex, "\rsm{0 & 0 & 0 & 0 \\
        1 & 0 & 0 & 0 \\ 0 & 1 & 0 & 0 \\ 0 & 0 & 1 & 0}"] 
    \end{tikzcd}
\]
Note that $V_\lambda\cong V_{t\lambda}$ for every $t\in K^\times$ and $\lambda\in K^2\setminus\{0\}$.
Hence the expression $V_\lambda$ is up to isomorphism well defined for $\lambda\in\bP^1$.
We set $\infty := (1:0)\in\bP^1$ and naturally identify $\bA^1 = \bP^1\setminus\{\infty\}$ 
by writing $\lambda = (\lambda : 1)\in\bP^1$ for $\lambda\in\bA^1$ with some abuse of notation.

\begin{proposition}\label{prop:bc1_family}
    Let $H$ and $V_\lambda, \bar{V}_\infty\in \mod(H)$ be as above.
    The following properties hold
    \begin{enumerate}[label = (\roman*)]
        \item \label{enum:stable} 
        $V_\lambda$ is $\partial$-stable for all $\lambda\in\bA^1$.
        \item \label{enum:semistable} 
        $\bar{V}_\infty$ is $\partial$-stable while $V_\infty$ is $\partial$-semistable and there is a non-split short exact sequence
        \[
        0\to \bar{V}_\infty \to V_\infty \to \bar{V}_\infty\to 0.
        \]
        \item \label{enum:iso}
        For $\lambda,\mu\in\bP^1$ is $V_\lambda \cong V_\mu$ if and only if $\lambda = \pm\mu$.
        \item \label{enum:hom_orthogonal}
        For $\lambda,\mu\in\bP^1$ with $\lambda \neq \pm\mu$ is $\Hom_H(V_\lambda,V_\mu)=0$.
        \item \label{enum:tau_translation}
        For every $\lambda\in\bP^1$ is $\tau_H(V_\lambda) \cong V_\lambda$.
    \end{enumerate}
\end{proposition}
\begin{proof}
    For \ref{enum:stable} let $U\subseteq V_\lambda$ be a proper non-zero submodule.
    We cannot have $\dim U(2) = 2$ otherwise $U = V$.
    If $\dim U(2) = 0$ we have strict inequality in (\ref{bc1_defect_semistable}).
    If $\dim U(2) = 1$, take any non-zero $u\in U(2)$.
    If $u_1\neq 0$ then $U(\varepsilon)^k U(\alpha) u$ are linearly independent for $0\leq k \leq 3$ 
    hence strict inequality in (\ref{bc1_defect_semistable}) holds.
    If $u_1 = 0$ then $u_2 \neq 0$ and by assumption $\lambda_2\neq 0$ 
    thus $U(\varepsilon)^k U(\alpha) u$ are linearly independent for $0\leq k \leq 2$
    and still strict inequality in (\ref{bc1_defect_semistable}) holds.

    For \ref{enum:semistable} it is enough to observe that $\bar{V}_\infty(2)$ generates $\bar{V}_\infty$.
    The obvious short exact sequence for $V_\infty$ also shows that $V_\infty$ is $\partial$-semistable.

    For \ref{enum:iso} and \ref{enum:hom_orthogonal} 
    assume first that $\lambda,\mu \neq \infty$ 
    then a general homomorphism $f\in\Hom(V_\lambda,V_\mu)$ is of the form
    \[
        \begin{tikzcd}[ampersand replacement=\&]
            V_\lambda
            \ar[ddd,"f"']
            \&\&
            K^2 
            \ar[rr,"\rsm{1 & 0 \\ 0 & 1 \\ 0 & \lambda\\ 0 & 0}"] 
            \ar[ddd, "\rsm{ a_1 & 0 \\ a_2 & a_1 }"']
            \&\&
            K^4
            \ar[loop, out=30, in=-30, distance=5ex, "\rsm{0 & 0 & 0 & 0 \\ 1 & 0 & 0 & 0 \\ 0 & 1 & 0 & 0 \\ 0 & 0 & 1 & 0}"]
            \ar[ddd,"\rsm{ a_1 & 0 & 0 & 0 \\ a_2 & a_1 & 0 & 0 \\ a_3 & a_2 & a_1 & 0 \\ a_4 & a_3 & a_2 & a_1}"]
            \\ \\ \\
            V_{\mu}
            \&\&
            K^2 
            \ar[rr,"\rsm{1 & 0 \\ 0 & 1 \\ 0 & \mu \\ 0 & 0}"'] 
            \&\&
            K^4
            \ar[loop, out=30, in=-30, distance=5ex, "\rsm{0 & 0 & 0 & 0 \\ 1 & 0 & 0 & 0 \\ 0 & 1 & 0 & 0 \\ 0 & 0 & 1 & 0}"] 
        \end{tikzcd}
    \]
    this diagram commutes if and only if
    $\rsm{\mu a_2 & \mu a_1 \\ 0 & 0} 
    = \rsm{a_3 & a_2 + \lambda a_1 \\ a_4 & a_3 + \lambda a_2}$
    which leaves two equations for $a_1$ and $a_2$, namely
    \begin{align*}
        a_2 = (\mu - \lambda)a_1 && 0 = (\mu+\lambda) a_2.
    \end{align*}
    Thus if $\lambda = -\mu$, then $a_2 = 2\mu a_1$ and one may choose $a_1\in K$ arbitrary.
    This proves $V_\lambda \cong V_{-\lambda}$.
    If on the other hand $\lambda \neq \pm\mu$ then $a_2=0$ and $a_1=0$.
    This proves $\Hom_H(V_\lambda,V_\mu) = 0$ in this case.
    Similarly, one checks $\Hom_H(\bar{V}_\infty,V_\lambda) = 0$ and $\Hom_H(V_\lambda,\bar{V}_\infty) = 0$ for any $\lambda\in\bA^1$.

    Finally, for \ref{enum:tau_translation} apply the dual of transpose construction
    to the minimal projective presentation
    \[
        \begin{tikzcd}
            P_1 \ar[r,"\ssm{\rho \\ -\alpha}"] 
            &
            P_2 \oplus P_2 \ar[r] 
            & V_\lambda \ar[r]
            & 0
        \end{tikzcd}
    \]
    where $\rho = \lambda_1\varepsilon^2\alpha + \lambda_2\varepsilon\alpha$.
    Applying $(-)^* := \Hom_H(-,H)$ to $\vec{P}(V_\lambda)$ 
    and taking the cokernel computes the Auslander-Bridger transpose $W_\lambda := \Tr_H(V_\lambda)$
    \[
        \begin{tikzcd}[ampersand replacement = \& ]
            P_2^* \oplus P_2^* \ar[r,"\text{$\ssm{\rho^* \mid -\alpha^*}$}"] 
            \&
            P_1^* \ar[r] 
            \& 
            W_\lambda \ar[r]
            \&
            0,
        \end{tikzcd}
        ~~~~ ~~~~ ~~~~ ~~~~
        \begin{tikzcd}[ampersand replacement=\&]
            W_\lambda\colon 
            ~ 
            K^2 
            \ar[<-,rrr,"\rsm{0 & 0 & 0 & 1 \\ 0 & \lambda_1 & \lambda_2 & 0}"] 
            \&\&\&
            K^4 \ar[loop, out=30, in=-30, distance=5ex, "\rsm{0 & 0 & 0 & 0 \\
            1 & 0 & 0 & 0 \\ 0 & 1 & 0 & 0 \\ 0 & 0 & 1 & 0}"] 
        \end{tikzcd}
    \]
    Dualizing with $\D = \Hom_K(-,K)$ gives the Auslander-Reiten translation $\tau_A(V_\lambda) \cong \D(W_\lambda) \cong V_\lambda$.
\end{proof}

We aim at proving that we found all $\partial$-stable modules. 
To this end we begin with a straight forward dimension count 
to check whether we have at least the correct number of parameters.

\begin{lemma}\label{lem:dimension}
    Let $\bd = (d_1,d_2)\in\bN Q_0$ and write $d_1 = 4r + s$ for some $r\in\bN$ and $0\leq s < 4$. 
    The variety $\Rep(H,\bd)$ is irreducible of dimension
    \begin{align*}
        \dim\Rep(H,\bd) = d_1d_2 + d_1^2 - 4r^2 - 2sr - s
    \end{align*}
\end{lemma}
\begin{proof}
    We have a product decomposition 
    \begin{align*}
        \Rep(H,\bd)\cong \Rep(K[x]/\ideal{x^4},d_1) \times \Rep(K\Lambda,\bd)
    \end{align*}
    where $\Lambda\colon 2 \to 1$.
    Clearly, $\Rep(K\Lambda,\bd)\cong \bA^{d_1d_2}$ is irreducible of dimension $d_1d_2$.
    On the other hand, consider the representation $N_{d_1}\in\Rep(K[x]/\ideal{x^4},d_1)$
    with dense orbit,
    defined as in the proof of Lemma \ref{lem:local_free_smooth}.
    In particular $\Rep(K[x]/\ideal{x^4},d_1)$ is irreducible of dimension $\dim\GL(K,d_1)-\dimend(N_{d_1})$ 
    where $\dim\GL(K,d_1)=d_1^2$ and $\dimend(N_{d_1}) = 4r^2 + 2sr + s$.
\end{proof}

Next,
we show that the generic representation with vanishing defect is semistable and
isomorphic to a direct sum of pairwise distinct $V_\lambda$ for $\lambda\in\bA^1$ 
and at most one $\bar{V}_\infty$.

\begin{proposition}\label{prop:generic_defect_semistable}
    Let $\bd = (d_1,d_2)\in\bN Q_0$ with $\partial(\bd)=0$. 
    Write $d_2 = 2r+s$ for some $r\in\bN$ and $0\leq s < 2$.
    Then
    \begin{align*}
        \Rep(H,\bd) = \overline{\bigcup_{\underline{\lambda} \in \bA^r} \cO(V_{\underline{\lambda}})},
        &&
        V_{\underline{\lambda}} := \bar{V}_{\infty}^s \oplus \bigoplus_{i = 1}^r V_{\lambda_i}.
    \end{align*}
\end{proposition}
\begin{proof}
    The orbit of $V_{\underline{\lambda}}$ 
    for $\underline{\lambda}\in \bA^r$ with pairwise different coordinates 
    has dimension 
    \[
        \dim \cO(V_{\underline{\lambda}}) = 5d_2^2 - r - s
    \] 
    while $\dim\Rep(H,\bd)=5d_2^2-s$, by Lemma \ref{lem:dimension},
    thus $\dim \Rep(H,\bd) - \dim \cO(V_{\underline{\lambda}}) = r$.
    Since $V_{\underline{\lambda}}$ has $r$-parameters, this proves the assertion.
\end{proof}

\begin{corollary}\label{cor:generic_locally_free}
    Let $\br\in\bN\Gamma_0$ be a rank vector.
    If $\br = m\cdot \bmeta$ for some $m\geq 0$, then
    \begin{align*}
        \cZ(\br) = \overline{\bigcup_{\underline{\lambda}\in\bA^m}\cO(V_{\underline{\lambda}})}.
    \end{align*}
    If $\br\notin\bN\bmeta$, then there exists a $\tau$-rigid $W\in\mod(H)$ with
    \begin{align*}
        \cZ(\br) = \overline{\cO(W)}.
    \end{align*}
\end{corollary}
\begin{proof}
    For $\br\in\bN\bmeta$, this is a special case of Proposition \ref{prop:generic_defect_semistable}.
    If instead $\br\notin \bN\bmeta$, then $\g(\br) \in\fan(H)$ as seen in Figure \ref{fig:exchange_quiver_and_fan}.
    Thus, there exists a $\tau$-rigid $H$-module $V$ with $\rkv(V) = \br$. 
    This proves the second assertion.
\end{proof}

Before we prove that we found all $\partial$-stable $H$-modules,
let us show a genericity property of bricks:

\begin{lemma} \label{lem:bc1_brick}
    Let $V \in\mod(H)$ be a brick.
    Then $V(\varepsilon)$ has maximal rank.
\end{lemma}
\begin{proof}
    Let $\dimv V = (v_1,v_2)$ 
    and write $v_1 = 4r + s$ with $0\leq s\leq 3$ and $r\in\bZ_{\geq 0}$.
    Note that
    \begin{align}\label{local_minimal_end}
    \dimend_{H(1)}(V(1)) \geq  e_{min} := 4r^2 + 2rs + s
    \end{align}
    with equality if and only if $V(\varepsilon)$ has maximal rank.
    On the other hand $v_1^2 = 4e_{min} - s(4-s)$.
    The key ingredient is the left exact sequence
    \begin{align*}
        0 \to \End_H(V) \to \End_{H(2)}(V(2)) \oplus \End_{H(1)}(V(1)) \xrightarrow{\Psi_V} \Hom_K(V(2),V(1))
    \end{align*}
    where $\Psi_V(f_2,f_1) := f_1 V(\alpha) - V(\alpha) f_2$. 
    Thus $\dimend_H(V) = 1$ implies
    \begin{align*}
        \dimend_{H(1)}(V(1)) 
        & \leq v_2v_1 - v_2^2 + 1 \\
        & = \frac{v_2}{v_1}\left(1 - \frac{v_2}{v_1}\right) v_1^2 + 1 \\
        & \leq \frac{1}{4}v_1^2 + 1 \\
        & = e_{min} + 1 - \frac{s}{4}(4-s)
    \end{align*}
    for the second inequality we used that $0\leq v_2\leq v_1$
    otherwise $V$ would have the simple module $S_2$ as a summand
    which we may rule out.
    But $0\leq s \leq 3$ 
    thus $\dimend_{H(1)}(V(1)) = e_{min}$ 
    and $V(\varepsilon)$ must have maximal rank.
\end{proof}

\begin{corollary}\label{cor:defect_stable}
    Let $V\in\mod(H)$ be $\partial$-stable.
    Then $V\cong \bar{V}_\infty$ or $V\cong V_\lambda$ for some $\lambda \in \bA^1$.
\end{corollary}
\begin{proof} 
    The key is that $\partial$-stable representations can only exist in dimensions $(2,1)$ and $(4,2)$.
    Indeed, in other dimensions $\bd\in\bN Q_0$ with $\partial(\bd) = 0$ is the generic element of $\Rep(H,\bd)$
    not a brick but bricks form open subsets thus there are no bricks in dimension $\bd$
    according to Theorem \ref{prop:generic_defect_semistable}.
    Also, $V(\varepsilon)$ has to have maximal rank by Lemma \ref{lem:bc1_brick}.
    
    \bigskip

    First assume $\bd = (2,1)$.
    Then $V$ is isomorphic to a representation of the form
    \[
        \begin{tikzcd}[ampersand replacement=\&]
            K 
            \ar[r,"\text{$\rsm{\alpha \\ \beta}$}"] 
            \&
            K^2
            \ar[loop, out=30, in=-30, distance=5ex, "\text{$\rsm{0 & 0 \\ 1 & 0}$}"] 
        \end{tikzcd}
    \]
    for some $\alpha,\beta\in K$. 
    But $\alpha \neq 0$ otherwise the subrepresentation generated by $V(2)$ is destabilizing.
    Thus one finds the isomorphism
    \[
        \begin{tikzcd}[ampersand replacement=\&]
            V
            \ar[dd,"\wr"']
            \&
            K 
            \ar[r,"\rsm{\alpha \\ \beta}"]
            \ar[dd,"\alpha"'] 
            \&
            K^2
            \ar[loop, out=30, in=-30, distance=5ex, "\text{$\rsm{0 & 0 \\ 1 & 0}$}"]
            \ar[dd,"\rsm{1 & 0 \\ -\beta/\alpha & 1}"] 
            \\ \\
            \bar{V}_\infty
            \&
            K 
            \ar[r,"\rsm{1 \\ 0}"]
            \&
            K^2
            \ar[loop, out=30, in=-30, distance=5ex, "\text{$\rsm{0 & 0 \\ 1 & 0}$}"]
        \end{tikzcd}
    \]

    Next consider the case $\bd=(4,2)$.
    Then $V$ is isomorphic to a representation of the form
    \[
        \begin{tikzcd}[ampersand replacement=\&]
            K^2 
            \ar[rr,"\ssm{A \\ B}"] 
            \&\&
            K^4
            \ar[loop, out=30, in=-30, distance=5ex, "\rsm{0 & 0 & 0 & 0 \\ 1 & 0 & 0 & 0 \\ 0 & 1 & 0 & 0 \\ 0 & 0 & 1 & 0}"] 
        \end{tikzcd}
    \]
    for some $A,B\in\Mat(2\times 2,K)$.
    Again $\det A \neq 0$ otherwise any non-zero $u\in V(2)$ with $Au = 0$ generates a destabilizing submodule.
    With column transformations and admissible row transformations, those which fix $V(\varepsilon)$, we find that $V$ is isomorphic to a representation of the form
    \[
        \begin{tikzcd}[ampersand replacement=\&]
            K^2 
            \ar[rr,"\ssm{1 & 0 \\ 0 & 1 \\ 0 & \alpha\\ 0 & \beta}"] 
            \&\&
            K^4
            \ar[loop, out=30, in=-30, distance=5ex, "\rsm{0 & 0 & 0 & 0 \\ 1 & 0 & 0 & 0 \\ 0 & 1 & 0 & 0 \\ 0 & 0 & 1 & 0}"] 
        \end{tikzcd}
    \]
    for some $\alpha,\beta\in K$. 
    Finally, there is an isomorphism $V\cong V_\lambda$ for $\lambda^2 = \alpha^2-\beta$.
\end{proof}

\begin{remark}
    Note that every proper factor algebra of $H$ is $\tau$-tilting finite.
    Therefore, our family $V_\lambda$ for $\lambda\in\bA^1$ of $\partial$-stables 
    cannot live over a proper factor algebra of $H$.
    To the best of our knowledge, 
    this is the first example of a $\tau$-tilting infinite representation wild algebra
    for which all stable modules can be classified
    and do not live over a tame factor algebra.
\end{remark}

\begin{remark} \label{rem:generic_basis}
    Recently, Mou realized all cluster monomials (without initial cluster variables as factors)
    of any cluster algebra of rank $2$
    as locally free Cladero-Chapoton functions of $\tau$-rigid modules \cite[Theorem 1.1]{LMou22}.
    In particular, this applies if $H$ is of type $\tilde{\BC}_1$.
    The corresponding cluster algebra $\cA(1,4)$ is studied by Sherman-Zelevinsky in \cite{SZ04}.
    The only remaining generically $\tau$-reduced components
    are those with generic rank vector a multiple of $\bmeta$.
    With our generic classification in Corollary \ref{cor:generic_locally_free}
    one explicitly computes that the generic locally free Caldero-Chapoton function of 
    $\cZ(m\bmeta)$ is
    \begin{align*}
        \CC_\lf(\cZ(m\bmeta)) = 
        \left(x_0^2x_3-(x_1+2)x_2^2\right)^m
    \end{align*}
    for any $m\geq 0$ (see \cite[Definition 1.1]{GLS18iv} and 
    \cite[Equation (1.1)]{LMou22} for the definition).
    These are powers of the ``primitive element'' $z_1 = x_0^2x_3-(x_1+2)x_2^2$ 
    of Sherman-Zelevinsky's canonical basis of $\cA(1,4)$ \cite[Theorem 2.8]{SZ04}.
    Form there it is straight forward to conclude 
    that the generic locally free Caldero-Chapoton functions 
    of (decorated) generically $\tau$-reduced components for the GLS algebra $H$
    provide a ``generic basis'' of the cluster algebra $\cA(1,4)$
    in the spirit of \cite[Theorem 5]{GLS12}.
\end{remark}

\subsection{A $1$-parameter family of $\partial$-stable modules} \label{sec:stable_family}

From now on,
we fix a connected affine valued quiver 
$\Gamma = (\Gamma_0,\Gamma_1,\Omega,\nu,\bc)$
with primitive null root $\bmeta := \bmeta_\Gamma$
and defect $\partial := \partial_\Gamma$.
Let $H:= H(\Gamma)$ be its associated GLS algebra.
The aim of this section is to construct a family 
of $\partial$-stable $H$-modules
with dimension vector $\bmeta$
analogous to the quasi-simple regular $\tilde{H}$-modules.
Recall, that every stable module is a brick 
hence we may assume in this section 
that the symmetrizer of $\Gamma$ is minimal.

\begin{theorem} \label{thm:null_family}
    Let $H = H(\Gamma)$
    for a valued quiver $\Gamma$
    of affine type
    and with minimal symmetrizer.
    There are $V_\lambda\in \mod(H)$ for $\lambda\in\bP^1$ 
    with the following properties:
    \begin{enumerate}[label = (\roman*)]
        \item \label{enum:locally_free_null_rank}
        Each $V_\lambda$ for $\lambda\in \bP^1$ is locally free with $\rkv(V_\lambda) = \bmeta_\Gamma$,
        \item \label{enum:hom_ext_ortho}
        for all $\lambda\neq \mu \in \bP^1$ is $\Hom_H(V_\lambda,V_\mu) = 0 = \Ext^1_H(V_\lambda,V_\mu)$,
        \item \label{enum:bricks}
        for all $\lambda\in\bA^1$ is $\End_H(V_\lambda) \cong K$,
        \item \label{enum:proj_stable}
        for all $\lambda\in\bP^1$ is $\Hom_H(V_\lambda,H) = 0$,
        \item \label{enum:defect_semistable}
        each $V_\lambda$ for $\lambda\in\bP^1$ is $\partial_\Gamma$-semistable,
        \item \label{enum:homogeneous}
        for almost all $\lambda\in\bP^1$ is $\tau_H(V_\lambda) \cong V_\lambda$,
        \item \label{enum:defect_stable}
        almost all $V_\lambda$ for $\lambda\in\bP^1$ are $\partial_\Gamma$-stable.
    \end{enumerate}
\end{theorem}

Note that this confirms, for the class of GLS algebras, a recent $\tau$-tilted version of the second Brauer-Thrall conjecture
proposed in \cite[Conjecture 1.3.(2)]{M22} and \cite[Conjecture 2]{STV21}.

\begin{conjecture}\label{conj:tau_btii}
    Let $A$ be a finite-dimensional algebra. 
    If $A$ is $\tau$-tilting infinite,
    then there exists a dimension $d\geq 0$
    and infinitely many pairwise non-isomorphic bricks $V\in\mod(A)$
    with $\dim(V) = d$.
\end{conjecture}

\begin{corollary}
    Let $H$ be a GLS algebra.
    Then $H$ satisfies Conjecture \ref{conj:tau_btii}.
\end{corollary}
\begin{proof}
    We may assume that $H$ is $\tau$-tilting infinite
    and that the symmetrizer is minimal.
    Then $\Gamma$ is not of finite type.
    Therefore there exists a full valued subquiver $\Gamma'\subseteq \Gamma$
    of affine type.
    Now $H':=H(\Gamma')$ is a factor algebra of $H(\Gamma)$.
    Thus it remains to prove Conjecture \ref{conj:tau_btii}
    for affine GLS algebras
    for which we have Theorem \ref{thm:null_family}.
\end{proof}

The proof of Theorem \ref{thm:null_family} will occupy the whole section.
Our construction is inspired by Ringel's 
of quasi-simple regular modules over ordinary path algebras
via one-point extensions.
The novelty is that we need to consider generalized one-point extensions
in most non-simply laced types.
To simplify our exposition,
we exclude the very well understood type $\tilde{\A}_{n\geq 2}$, 
despite its undisputed importance.

\bigskip

Let us start with the combinatorial construction of affine valued quivers as
extended finite type valued quivers:
Choose a vertex $0\in\Gamma_0$ with $\eta_0 = 1$ and set
\begin{align*}
    \bmeta' :=
    \begin{cases}
        \frac{1}{2}(\bmeta - \bmalpha_0) & \text{if $\Gamma$ is of type $\tilde{\BC}$,} \\
        \bmeta - \bmalpha_0 & \text{else.}
    \end{cases}
\end{align*}
It is possible to choose $0\in\Gamma_0$ 
such that $\bmeta'$ is a root of $\Gamma$
(our choice is highlighted in Table \ref{tab:affine}
and coincides with those in \cite[Table 2]{M69}).
The vertex $0\in\Gamma_0$ is called \emph{extending vertex}
and uniquely determined up to symmetries of the underlying graph.
Note that $\bmeta'$ is supported on the full valued subquiver $\Gamma'\subset \Gamma$ without vertex $0$,
and $\Gamma'$ is of finite type.
We will from now on assume that $0$ is a source.
The case when $0$ is a sink can be treated dually.
There is then a unique arrow $\alpha\in\Gamma_1$ 
with $s(\alpha) = 0$
and we set $1:= t(\alpha)$.
Here we needed to excluded the type $\tilde{\A}_{n\geq 2}$.
We define the \emph{extending type} $\Upsilon = \Upsilon(\Gamma)$ of $\Gamma$ 
to be the valued quiver:
\[
    \begin{tikzcd}
        \Upsilon\colon ~~~~ 
        c'_0 \ar[rr, "\nu'_{01} ~\mid~ \nu'_{10}"] && c'_1
    \end{tikzcd}
\]
with valuation $\nu'$ and symmetrizer $\bc'$ given by:
\begin{align*}
    c'_0 := c_0 = \euler{\bmalpha_0,\bmalpha_0}_\Gamma &&
    c'_1 := \euler{\bmeta',\bmeta'}_\Gamma &&
    c'_0\nu'_{01} = c'_1\nu'_{10} := \euler{\bmalpha_0,\bmeta'}_\Gamma
\end{align*}
For the sake of more compact notation,
we drop the labels $0$ and $1$ of vertices of $\Upsilon$
and replace them by the values $c'_0$ and $c'_1$ of the symmetrizer.
When $\Gamma$ is of type $\tilde{\BC}$, 
we recover the valued quiver
of type $\tilde{\BC}_{1}$
which is dual to the one studied in Section \ref{sec:extended_type_bc1}.
If $\Gamma$ is not of type $\tilde{\BC}$,
then $\Upsilon$ is symmetric with valuation $\nu'_{01} = \nu'_{10} = 2$,
thus uniquely determined by $c_0$,
and $c_0 = t_{\Gamma^\dagger}$ happens to be the tier number 
of the transposed valued graph $\Gamma^\dagger$.
The appearance of the transposed valued graph $\Gamma^\dagger$
in the construction of $\partial$-stable $H$-modules
is in striking accordance with \cite[Theorem 1.2.(c)]{GLS20}
where it is shown that the assignment of dimension vectors
sets up a bijective correspondence between 
left finite bricks over $H$ and real Schur roots for $\Gamma^\dagger$.

\bigskip

Next, let us narrow down where to find our family of modules.
For a general finite-dimensional algebra $A$
and modules $V_1,\dots,V_n \in\mod(A)$ write
$\filt(V_1,\dots,V_n)$ for the full subcategory of $\mod(A)$
consisting of modules $W\in\mod(A)$
possessing a filtration
\begin{align*}
    0 = W_0 \subseteq W_1 \subseteq \cdots \subseteq W_l = W
\end{align*}
such that for all $0\leq i < l$
exists $1\leq k\leq n$ with $W_{i+1}/W_{i} \cong V_k$.
Since $\bmeta'$ is by definition a real root 
for the valued quiver $\Gamma'$ of finite type
with $\euler{\bmeta',\bmeta'}_\Gamma = c'_1$,
there is a unique $\tau$-rigid module $V'$ of $H$ 
with $\rkv(V')=\bmeta'$ and its endomorphism algebra is 
$\End_H(V')\cong K[\delta_1]/\ideal{\delta_1^{c'_1}}$.
Let $V''$ be the top of $V'$ as a $\End_H(V')$-module.
Then $V''$ is a brick with $V'\in\filt(V'')$ by Proposition \ref{prop:tau_rigid_structure}.
Consider the Abelian subcategory $\cB := \filt(S_0,V'')$ 
filtered by the $\Hom$-orthogonal bricks $S_0$ and $V''$.
The main technical insight of this section 
is the determination of $\cB$
based on a construction due to Bongartz.

\begin{lemma} \label{lem:regular_bimodule_bongartz}
    Let $H$, $V'$ and $\cB$ be as above.
    The subcategory $\cB\subset \mod(H)$ is equivalent to the category of modules over the tensor algebra $B(\Upsilon)$ 
    with simple quiver $\Upsilon\colon 0 \xrightarrow{\beta} 1$ and modulation
    \begin{align*}
        B(0) := \End_H(E_0)^\op && B(1) := \End_H(V')^\op && B(\beta) := \Ext^1_H(E_0,V').
    \end{align*}
\end{lemma}
\begin{proof}
    Since $0\in\Gamma_0$ is by assumption a source,
    we have $\Ext^1(V',S_0) = 0$.
    Since $V'$ is $\tau$-rigid and $V''$ is a factor of $V'$,
    we have $\Ext^1(V',V'') = 0$.
    This shows that
    $V'$ is relative projective in $\cB$.
    It remains to find a relative projective cover of $S_0$ in $\cB$.
    This is a construction due to Bongartz \cite[Lemma 2.1]{Bon81}: 
    Choose generators $\gamma_1,\dots,\gamma_g$ of $\Ext^1_H(E_0,V')$ as an $\End_H(V')$-module.
    The universal sequence $\gamma \in \Ext^1_H(E_0,V'^g)$ is defined as the pull-back along the diagonal of their direct sum
    \[
        \begin{tikzcd}
            \gamma\colon & 
            0 \ar[r] & V'^g \ar[r] \ar[equal,d] & M \ar[r] \ar[d] & E_0 \ar[r] \ar[hookrightarrow, d] & 0 
            \\
            \bigoplus_{i=1}^g \gamma_i\colon & 
            0 \ar[r] & V'^g \ar[r] & \bigoplus_{i=1}^g M_i \ar[r] & E_0^g \ar[r] & 0
        \end{tikzcd}
    \]
    the rightmost vertical map denotes the diagonal embedding of $E_0$ in $E_0^g$.
    By construction, we have that the connecting homomorphism
    \begin{align*}
        \Hom_H(V'^g,V') \xrightarrow{\sim} \Ext^1_H(E_0,V')
    \end{align*}
    is an isomorphism.
    Moreover,
    $V'\oplus M$ is rigid and
    $\Ext^1_H(M,E_0) = 0$. 
    In particular $\Ext^1_H(M,V'') = 0$ and $\Ext^1_H(M,S_0) = 0$ because $\pdim(M)\leq 1$.
    This shows that $V'\oplus M$ is a relative projective generator in $\cB$.
    Therefore
    \begin{align*}
        \cB \xrightarrow{\sim} \mod(B^\op), &&
        B := \End_H(V'\oplus M) =
        \left[
            \begin{matrix}
                \End_H(M) & \Hom_H(V',M) \\ \Hom_H(M,V') & \End_H(V')
            \end{matrix}
        \right].
    \end{align*}
    We have $\Hom_H(M,V') = 0$, 
    $\End_H(M) \cong \End_H(E_0)$ 
    and $\Hom_H(V',M) \cong \D\Ext^1_H(E_0,V')$.
\end{proof}

We call ${_{B(0)}}B(\beta)_{B(1)}$ the \emph{extending bimodule} and 
$B(\Upsilon)$ the \emph{extending tensor algebra}.
It remains to explicitly calculate the extending tensor algebra in each type.
They depend only on the valued quiver $\Upsilon$, 
which is not obvious from the construction.

\begin{lemma} \label{lem:regular_bimodule_quiver_relations}
    Let $H$ be an affine GLS algebra with minimal symmetrizer
    whose extending vertex is a source.
    The extending tensor algebra $B$ of $H$ is 
    isomorphic to $KQ/I$ 
    for $Q$ and $I$ depending only on $\Upsilon$:
    \begin{enumerate}[label = (\roman*)]
        \item \label{enum:type_a_bimodule}
        If $\Upsilon\colon 1 \xrightarrow{2 \mid 2} 1$, 
        then $B$ is the Kronecker algebra i.e.
        \[
            \begin{tikzcd}
                Q \colon ~~~~ 
                0 \ar[r, shift left = 1.25mm] \ar[r, shift right = 1.25mm] & 1 
                &&&&
                I = 0.
            \end{tikzcd}
        \]
        \item \label{enum:type_c_bimodule}
        If $\Upsilon\colon 2 \xrightarrow{2 \mid 2} 2$,
        then $B$ is a gentle algebra, explicitly
        \[
            \begin{tikzcd}
                Q \colon &
                0 \ar[loop, in = 150, out = -150, distance = 5ex, "\delta_0"] \ar[r,"\beta"] &
                1 \ar[loop, in = 30, out = -30, distance = 5ex, "\delta_1"']
                &&&&
                I = \ideal{\delta_0^2,\delta_1^2}.
            \end{tikzcd}
        \]
        \item \label{enum:type_g_bimodule}
        If $\Upsilon\colon 3 \xrightarrow{2 \mid 2 } 3$,
        then $B$ is given by
        \[
            \begin{tikzcd}
                Q \colon &
                0 \ar[loop, in = 150, out = -150, distance = 5ex, "\delta_0"] \ar[r,"\beta"] &
                1 \ar[loop, in = 30, out = -30, distance = 5ex, "\delta_1"']
                &&&&
                I = \ideal{\delta_0^3,\delta_1^3, \delta_1^2\beta + \delta_1\beta\delta_0 + \beta\delta_0^2}.
            \end{tikzcd}
        \]
        \item \label{enum:type_bc_bimodule}
        If $\Upsilon\colon 4 \xrightarrow{4 \mid 1 } 1$,
        then $B$ is a GLS algebra of type $\tilde{\BC}_1$ with
        \[
            \begin{tikzcd}
                Q \colon &
                0 \ar[loop, in = 150, out = -150, distance = 5ex, "\delta_0"] \ar[r,"\beta"] &
                1
                &&&&
                I = \ideal{\delta_0^4}.
            \end{tikzcd}
        \]
    \end{enumerate}
\end{lemma}
\begin{proof}
    We already know that 
    \begin{align} \label{bimodule_dimension}
        B(0)\cong K[\delta_0]/\ideal{\delta_0^{c'_0}} &&
        B(1)\cong K[\delta_1]/\ideal{\delta_1^{c'_1}} &&
        \dim_K(B(\beta)) = \euler{\bmalpha_0,\bmeta'}_\Gamma
    \end{align}
    It remains to determine the structure of the bimodule
    \begin{align} \label{bimodule_structure}
        B(\beta) 
        \cong \{V(\alpha) \in \Hom_K(E_0(0),V'(1)) \mid V'(\varepsilon_1)^{f_{01}} \circ V(\alpha) = V(\alpha) \circ E_{0}(\varepsilon_0)^{f_{10}}\}
    \end{align} 

    \ref{enum:type_a_bimodule}: The isomorphism $B \cong KQ$ 
    follows immediately from (\ref{bimodule_dimension}) namely
    $B(0) \cong K$, $B(1) \cong K$ and $\dim_K B(\beta) = 2$.

    \bigskip

    \ref{enum:type_c_bimodule}:
    We claim that $B(\beta)$ is the free $B(0)$-$B(1)$-bimodule of rank $1$
    and prove this for the individual types case by case.

    \bigskip

    Suppose $\Gamma$ is of type $\tilde{\C}_{n\geq 2}$
    \[
        \begin{tikzcd}
            n \ar[-,r,"2 ~\mid~ 1"] & (n-1) \ar[-,r] & \cdots \ar[-,r] & 1 & 0 \ar[l, "1 ~\mid~ 2"']
        \end{tikzcd}
    \]
    where edges without an arrowhead are oriented arbitrarily.
    Then $H$ is itself a gentle algebra.
    The null root is $\bmeta = (1,2,\dots,2,1)$ 
    and the $\tau$-rigid module $V'$ is
    \[
        \begin{tikzcd}
            n_1 \ar[-,r] \ar[d] & (n-1)_1 \ar[-,r] & \cdots \ar[-,r] & 1_1 \\
            n_2 \ar[-,r] & (n-1)_2 \ar[-,r] & \cdots \ar[-,r] & 1_2
        \end{tikzcd}
    \]
    Note that $\End_H(V') \cong K[\delta_1]/\ideal{\delta_1^2}$ 
    where $\delta_1$ acts on $V'(1)$ 
    by sending the basis vector labeled $1_1$ to the basis vector labeled $1_2$.
    On the other hand, $E_0(0)$ has a standard basis labeled by $0_1$ and $0_2$
    and $\varepsilon_0$ acts on $E_0(0)$ 
    by sending the basis vector labeled $0_1$ to the basis vector labeled $0_2$.
    Therefore, the $B(0)$-$B(1)$-bimodule $B(\beta)$ 
    is generated by the linear map $V(\alpha)\colon E_0(0) \to V'(1)$
    which sends the basis vector of $E_0(0)$ labeled $0_2$
    to the basis vector of $V'(1)$ labeled $1_1$.
    This may be summarized as in the following diagram
    \[
        \begin{tikzcd}
            0_1 \ar[d,"\varepsilon_0"'] \\
            0_2 \ar[rightsquigarrow, rr,"V(\alpha)"]
            && 1_1 \ar[dashed, d, "\delta_1"] \\
            && 1_2
        \end{tikzcd}
    \]
    Here and for the remainder of the proof, 
    solid arrows represent the action of arrows in $Q$,
    dashed arrows represent the action of a generator $\delta_1\in\End_H(V')$
    and wiggly arrows represent a generator $V(\alpha)\in B(\beta)$ of the extending bimodule
    as given in (\ref{bimodule_structure}).
    
    \bigskip

    For type $\tilde{\BD}_{n\geq 2}$ 
    consider the valued quiver $\Gamma$
    \[
        \begin{tikzcd}
            & & & 0 \ar[d] & \\
            (n-1) \ar[r,"1 ~\mid~ 2"] & (n-2) \ar[-,r] & \cdots \ar[-,r] & 1 \ar[r] & n
        \end{tikzcd}
    \]
    where edges without an arrowhead are oriented arbitrarly.
    The null root is 
    $\bmeta = (1,2,\dots,2,1)$.
    The $\tau$-rigid module $V'$ can be found in \cite[Section 6.5.4]{GLS16iii} and is
    \[
        \begin{tikzcd}
            (n-1)_1 \ar[r] & (n-2)_1 \ar[-,r] \ar[d] & \cdots \ar[-,r] & 1_1 \ar[r] \ar[d]& n_1 \ar[d]\\
            (n-1)_2 \ar[r] \ar[dr] & (n-2)_2 \ar[-,r] & \cdots \ar[-,r] & 1_2 \ar[r] & n_2 \\
            & (n-2)_3 \ar[-,r] \ar[d] & \cdots \ar[-,r] & 1_3 \ar[d] & \\
            & (n-2)_4 \ar[-,r] & \cdots \ar[-,r] & 1_4 & 
        \end{tikzcd}
    \]
    where a vertex labeled $k_i$ represents 
    the $i^{th}$ standard basis vector 
    of $V'(k)$ for $k\in Q_0$
    and arrows represent the action of corresponding arrows in $Q_1$.
    For example $V'(\alpha_{n-1,n})$ sends the $2^{nd}$ standard basis vector of $V'(n)$ 
    to the sum of the $2^{nd}$ and $3^{rd}$ basis vector of $V'(n-1)$.
    Now, it is a straight forward calculation that 
    a generator $\delta_1$ of $\End_H(V') \cong K[\delta_1]/\ideal{\delta_1^2}$ 
    acts on $V'(1)$ as the dashed arrows in the following diagram:
    \[
        \begin{tikzcd}
            0_1 \ar[d] \ar[rightsquigarrow, rr] &&
            1_1 \ar[d] \ar[dashed, d, bend right = 2cm] \ar[dashed, r] & 
            1_3 \ar[d] \ar[dashed, d, bend left = 2cm, "-"]\\
            0_2 \ar[rightsquigarrow, rr]&&
            1_2 \ar[dashed, r] & 
            1_4
        \end{tikzcd}
    \]
    The minus sign indicates that $\delta_1$ sends the $3^{rd}$ basis vector of $V'(1)$
    to minus the $4^{th}$ basis vector.
    The map $V(\alpha)$, 
    depicted as wiggly arrows in the preceding diagram,
    generates $B(\beta)$ as a $B(0)$-$B(1)$-bimodule.
    But $\dim (B(\beta)) = 4$ and $\dim (B(1)\otimes_K B(0)) = 4$, hence $B(\beta)$ must be the free $B(1)$-$B(0)$-bimodule of rank 1.
    Other orientations of $\Gamma$ are treated similarly.

    \bigskip

    For type $\tilde{\F}_{4,1}$
    consider the valued quiver $\Gamma$
    \[
        \begin{tikzcd}
            4 \ar[r] & 3 \ar[r, "1 ~\mid~ 2"] & 2 \ar[r] & 1 & 0 \ar[l] 
        \end{tikzcd}
    \]
    The primitive null root is $\bmeta = (1,2,3,4,2)$
    and the $\tau$-rigid module $V'$ is
    \[
        \begin{tikzcd}
            & & & 
            4_1 \ar[r] & 3_1 \ar[r] & 2_1 \ar[r] \ar[d] & 1_1 \ar[d] \\
            1_3 \ar[d] & 2_5 \ar[d] \ar[l] & 3_3 \ar[l] & 
            4_2 \ar[l] \ar[r] & 3_2 \ar[r] \ar[dr] & 2_2 \ar[r] & 1_2 \\
            1_4 & 2_6 \ar[l] & 3_4 \ar[l] & & & 2_3 \ar[d] & \\
            &&& && 2_4 &
        \end{tikzcd}
    \]
    It is again an easy calculation that 
    a generator $\delta_1$ of $\End_H(V')\cong K[\delta_1]/\ideal{\delta_1^2}$
    acts on $V'(1)$ as the dashed arrows in the following diagram:
    \[
        \begin{tikzcd}
            0_1 \ar[d] \ar[rightsquigarrow, rr] &&
            1_1 \ar[d] \ar[dashed, d, bend right = 2cm] \ar[dashed, r, "-"] & 
            1_3 \ar[d] \ar[dashed, d, bend left = 2cm]\\
            0_2 \ar[rightsquigarrow, rr]&&
            1_2 \ar[dashed, r, "-"] & 
            1_4
        \end{tikzcd}
    \]
    The indicated map $V(\alpha)$
    generates $B(\beta)$ as a $B(0)$-$B(1)$-bimodule.
    Comparing dimensions shows that $B(\beta)$ is the free $B(0)$-$B(1)$-bimodule of rank $1$.

    \bigskip

    \ref{enum:type_g_bimodule}:
    Here $\Gamma$ is of type $\tilde{\G}_{2,1}$
    and we consider the following orientation:
    \[
        \begin{tikzcd}
            2 \ar[r, "1 ~\mid~ 3"] & 1 & 0 \ar[l]
        \end{tikzcd}
    \]
    The null root is $\bmeta = (1,2,3)$
    and the $\tau$-rigid module $V'$ is (see \cite[Section 6.8.4]{GLS16iii})
    \[
        \begin{tikzcd}
            & 2_1 \ar[r] & 1_1 \ar[d] \\
            1_4 \ar[d] & 2_2 \ar[r] \ar[l] & 1_2 \ar[d] \\
            1_5 \ar[d] & 2_3 \ar[l] & 1_3 \\
            1_6
        \end{tikzcd}
    \]
    Once more, it is easy to find a generator $\delta_1$ of $\End_H(V')\cong K[\delta_1]/\ideal{\delta_1^3}$
    which acts on $V'(1)$ as the dashed arrows
    in the following diagram:
    \[
        \begin{tikzcd}
            0_1 \ar[d] \ar[rightsquigarrow, rr] &&
            1_1 \ar[d] \ar[dashed, d, bend right = 2cm] \ar[dashed, r] & 
            1_4 \ar[d] \ar[dashed, ddl, "-"] \\
            0_2 \ar[rightsquigarrow, rr] \ar[d] &&
            1_2 \ar[dashed, d, bend right = 2cm] \ar[dashed, r] \ar[d]& 
            1_5 \ar[d] \\
            0_3 \ar[rightsquigarrow, rr] &&
            1_3 \ar[dashed, r] &
            1_6
        \end{tikzcd}
    \]
    The map $V(\alpha)$
    generates $B(\beta)$ as a $B(0)$-$B(1)$-bimodule
    and satisfies the relation 
    $\delta_1^2 V(\alpha) + \delta_1 V(\alpha) \delta_0 + V(\alpha) \delta_0^2 = 0$.
    Now, conclude $B\cong KQ/I$ by a simple dimension count.
    The other orientation is treated similarly.

    \bigskip

    \ref{enum:type_bc_bimodule}: The isomorphism $B \cong KQ/I$
    follows from a brief look at (\ref{bimodule_dimension}) and (\ref{bimodule_structure}):
    Indeed
    $B(0) \cong K[\delta_0]/\ideal{\delta_0^4}$, $B(1) \cong K$ and $\dim_K(B(\beta)) = 4$
    while $B(\beta)$ is 
    generated by a single element as a $B(1)$-$B(0)$-bimodule.
    By comparing dimensions, 
    $B$ must be the free $B(0)$-$B(1)$-bimodule of rank $1$.
\end{proof}

In analogy with our notation
we may think of $B(2 \xrightarrow{2 \mid 2}2)$
as an algebra of type $\tilde{C}_1$.
But note that $B(\Upsilon)$ is not the GLS algebra associated to the valued quiver $\Upsilon$
when $\Upsilon$ is $2 \xrightarrow{2 \mid 2} 2$ or $3 \xrightarrow{2 \mid 2} 3$.
On the other hand, $B(\Upsilon)$ is still $1$-Iwanaga-Gorenstein
and deforms to the path algebra 
of affine type $\tilde{\A}_{3}$ respectively $\tilde{\A}_{5}$
with bipartite orientation.
This explicit description of $B$ allows us to find 
reasonable candidates for our aimed family of $H$-modules:

\begin{lemma} \label{lem:generating_quasi_simple_family}
    Let $H$ be an affine GLS algebra with minimal symmetrizer.
    Let $B$ be the extending tensor algebra of $H$.
    There are $V_\lambda \in\mod(B)$ for $\lambda\in\bP^1$
    such that:
    \begin{enumerate}[label = (\roman*)]
        \item \label{enum:gqsf_lf}
        Each $V_\lambda$ for $\lambda\in\bP^1$ is locally free with
        \begin{align*}
            \rkv(V_\lambda) = 
            \begin{cases}
                (1,2) & \text{if $\Upsilon = 4 \xrightarrow{4 \mid 1}1$}, \\
                (1,1) & \text{else,}
            \end{cases}
        \end{align*}
        \item \label{enum:gqsf_hom_ortho}
        for all $\lambda,\mu \in \bP^1$ is $\Hom_B(V_\lambda,V_\mu) = 0$,
        \item \label{enum:gqsf_brick}
        for all $\lambda \in \bA^1$ is $\End_B(V_\lambda) \cong K$.
    \end{enumerate}
\end{lemma}
\begin{proof}
    We may assume that $0\in\Gamma_0$ is a source
    and consider the different cases from Lemma \ref{lem:regular_bimodule_quiver_relations}.

    \bigskip

    $\Upsilon \colon 1 \xrightarrow{2 \mid 2} 1$.
    Here $B$ is the Kronecker algebra
    thus the $V_\lambda$ for $\lambda\in\bP^1$ are the quasi-simple modules.

    \bigskip

    $\Upsilon \colon 2 \xrightarrow{2 \mid 2}2$.
    Over $B$ there is a unique brick in dimension $(1,1)$ 
    and a unique $\bA^1$-family of bricks in dimension $(2,2)$.
    The latter is a family of band modules: 
    There are only two types of bands up to equivalence in dimension $(2,2)$ namely:
    \[
        \begin{tikzcd}
            \omega \colon &
            0 \ar[dr,"\delta_0"] \\ 
            && 0 \ar[r,"\beta"] &  1 \ar[dr,"\delta_1"] \\
            &&&& 1 \ar[<-,r,"\beta"] & 0
        \end{tikzcd}
        ~~~~
        \begin{tikzcd}
            \nu \colon & 
            0 \ar[dr,"\delta_0"] &&& 1 \ar[<-,r,"\beta"] & 0 \\ 
            && 0 \ar[r,"\beta"] &  1 \ar[ur,<-,"\delta_1"] 
        \end{tikzcd}
    \]
    Only the band modules defined by $\omega$ are bricks
    and those (together with the string module at $\lambda = 0$)
    constitute an $\bA^1$-family of pairwise $\Hom$-orthogonal bricks $V_\lambda$.
    Moreover there is a unique brick $\bar{V}_\infty$
    with dimension vector $\dimv(\bar{V}_\infty) = (1,1)$.
    The module $\bar{V}_\infty$ 
    has a self-extension which we include in the family as $V_\infty$.

    \bigskip

    $\Upsilon \colon 3 \xrightarrow{2 \mid 2} 3$.
    The algebra $B$ deforms to a path algebra of type $\tilde{\A}_{5}$ 
    with bipartite orientation.
    This time there is a unique brick in dimension $(1,1)$ 
    and a unique $\bA^1$-family of pairwise $\Hom$-orthogonal bricks in dimension $(3,3)$.

    \bigskip

    $\Upsilon \colon 4 \xrightarrow{4 \mid 1} 1$.
    This is a special case of Proposition \ref{prop:bc1_family}
    with an appropriate reindexing:
    The group $\bZ/2$ acts on $\bP^1$ by switching signs, 
    and the quotient of $\bP^1$ by $\bZ/2$ is isomorphic to $\bP^1$.
\end{proof} 

\begin{proof}[Proof of Theorem \ref{thm:null_family}]
    The $B$-modules $V_\lambda$ for $\lambda\in\bP^1$ 
    found in Lemma \ref{lem:generating_quasi_simple_family}
    may be seen as $H$-modules along the equivalence $\mod(B)\simeq \filt(S_0,V'')$
    from Lemma \ref{lem:regular_bimodule_bongartz}.

    \bigskip

    \ref{enum:locally_free_null_rank}, \ref{enum:hom_ext_ortho} and \ref{enum:bricks}: 
    These follow immediately from the corresponding properties 
    established in Lemma \ref{lem:generating_quasi_simple_family}.
    Note that $\Ext^1$-orthogonality follows from $\Hom$-orthogonality, indeed
    \begin{align*}
        0 = \euler{\bmeta,\bmeta}_\Gamma = \euler{V_\lambda,V_\mu}_H = - \dimext^1_H(V_\lambda,V_\mu).
    \end{align*}

    \bigskip

    \ref{enum:proj_stable}:
    Note that $\Hom_H(S_0\oplus V'',P_i) = 0$ for all $i\in\Gamma'_0$.
    Moreover, $\Hom_H(S_0,P_0) = 0$
    and the kernel of $P_0 \twoheadrightarrow E_0$
    is a direct sum of projectives supported over $\Gamma'$
    hence $\Hom_H(V'',P_0) = 0$.
    This shows $\Hom_H(V_\lambda,H) = 0$ for all $\lambda\in\bP^1$.

    \bigskip

    \ref{enum:defect_semistable}: 
    Take $\lambda,\mu \in\bP^1$ with $\lambda \neq \mu$. 
    Let $U\subseteq V_\lambda$ be a submodule.
    Since $\idim_H(V_\mu)\leq 1$ by \ref{enum:locally_free_null_rank}
    and $\Ext^1_H(V_\lambda,V_\mu) = 0$ by \ref{enum:hom_ext_ortho},
    we find $\Ext^1_H(U,V_\mu) = 0$.
    Therefore
    \begin{align*}
        \partial(U) 
        = \euler{\bmeta, \rkv(U)}_\Gamma 
        = - \euler{\rkv(U), \bmeta}_\Gamma
        = \dimext^1_H(U,V_\mu) - \dimhom^1_H(U,V_\mu)
        = - \dimhom_H(U,V_\mu) \leq 0.
    \end{align*}

    \bigskip

    \ref{enum:homogeneous}: Recall that
    \ref{enum:proj_stable} implies that 
    $\tau_H(V_\lambda)$ is locally free for all $\lambda\in\bP^1$
    by Proposition \ref{prop:locally_free},
    and $\rkv(\tau_H(V_\lambda)) = \Phi_\Gamma(\rkv(V_\lambda)) = \bmeta$
    by Proposition \ref{prop:homological_ringel_coxeter}.
    Consider the constructible subsets
    \begin{align*}
        \cV := \bigcup_{\lambda\in\bP^1} \cO(V_\lambda) \subseteq \Rep_\lf(H,\bmeta) &&
        \cW := \bigcup_{\lambda\in\bP^1} \cO(\tau_H(V_\lambda)) \subseteq \Rep_\lf(H,\bmeta).
    \end{align*}
    The Auslander-Reiten duality formulas and \ref{enum:bricks} imply
    \begin{align*}
        \End_H(\tau_H(V_\lambda))\cong \D\Ext^1_H(V_\lambda,\tau_H(V_\lambda))\cong \End_H(V_\lambda) \cong K
    \end{align*}
    for all $\lambda\in\bP^1$ because $V_\lambda$ and $\tau_H(V_\lambda)$ are locally free. 
    This shows 
    \begin{align*}
        \dim \cO(\tau_H(V_\lambda)) = \dim \cO(V_\lambda)= \dim \GL(D\bmeta) - 1.
    \end{align*}
    On the other hand, $\dim \Rep_\lf(H,\bmeta) = \dim \GL(D\bmeta)$ by (\ref{lf_rep_varitey_dim}).
    Therefore, $\cV,\cW\subseteq \Rep_\lf(H,D\bmeta)$ are dense.
    By constructibility of $\cV$ and $\cW$ and irreducibility of $\Rep_\lf(H,\bmeta)$, 
    $\cV \cap \cW$ is still dense in $\Rep_\lf(H,\bmeta)$.
    If $\lambda,\mu \in \bP^1$ are such that $\tau_H (V_\mu) \cong V_\lambda$,
    then $\Ext^1_H(V_\mu,V_\lambda) \cong \D\Hom_H(V_\lambda,V_\lambda) \neq 0$ 
    which shows $\lambda = \mu$ by \ref{enum:hom_ext_ortho}.

    \bigskip

    \ref{enum:defect_stable}:
    For almost all $\lambda\in \bP^1$ is $V_\lambda$ a brick 
    with $\tau_H(V_\lambda) \cong V_\lambda$ 
    by \ref{enum:bricks} and \ref{enum:homogeneous}.
    Therefore Lemma \ref{lem:generalized_semibrick_semistability}
    tells us that all those $V_\lambda$ are $\partial$-stable.
\end{proof}

\subsection{Regular $\tau$-rigid modules} \label{sec:regular_tau_rigid}

The aim of this section is to transfer the well known structure 
of regular rigid modules over the species $\tilde{H}$
to the class of regular $\tau$-rigid $H$-modules
using Geiß-Leclerc-Schröer's bijections from Theorem \ref{thm:rigid_bijections}.
Recall that an indecomposable module $V\in\mod(A)$ 
over a finite-dimensional algebra $A$ is \emph{regular}
provided $\tau_A^n(V)\neq 0$ for all $n\in\bZ$.
The structure of regular rigid $\tilde{H}$-modules is very well known from Dlab-Ringel's work \cite{DR76}.
We will see, that the indecomposable regular $\tau$-rigid $H$ modules
organize in similar ``truncated pseudo tubes".
This further allows us to establish their $\partial$-semistability.
We begin with several equivalent characterizations of regular $\tau$-rigid $H$-modules.
The first four statements are easily seen to be equivalent 
while the fifth characterization is the main content of this subsection
and the key to proof $\Hom$- and $\Ext^1$-orthogonality with our constructed $1$-parameter family from the previous subsection.

\begin{proposition}\label{prop:regular_tau_rigid}
    The following are equivalent for an indecomposable $\tau$-rigid module $V\in\mod H$
    \begin{enumerate}[label = (\roman*)]
        \item \label{enum:regular}
        $V$ is regular i.e. $\tau_H^n(V) \neq 0$ for all $n\in\bZ$.
        \item \label{enum:regular_tau_periodic}
        $V$ is $\tau$-periodic i.e. there exists an $n\geq 1$ with $\tau_H^n(V) \cong V$.
        \item \label{enum:regular_real_Schur_root}
        There are 
        $\lambda\in I_\Gamma$,
        $k\in\bZ/r_{\lambda}$ and
        $1\leq l < r_\lambda$ with $\rkv(V) = \bv_{\lambda,k}^{(l)}$.
        \item \label{enum:regular_zero_defect}
        The defect of $V$ vanishes i.e. $\partial(V) = 0$.
        \item \label{enum:regular_defect_semistable}
        $V$ is $\partial$-semistable.
    \end{enumerate}
\end{proposition}
\begin{proof}
    From \cite[Theorem 1.2]{GLS20} we know that $\bv = \rkv(V)$ is a real Schur root of $\Gamma$ 
    and determines the isomorphism class of $V$ among all $\tau$-rigid $H$-modules uniquely.
    Moreover, it is already noted in \cite[Proposition 11.4]{GLS17i} 
    that $\tau_H^k(V)$ is $\tau$-locally free and rigid for every $k\in \bZ$
    with $\rkv (\tau_H^k (V)) = \Phi_\Gamma^k (\rkv (V))$.
    This and the structure of the affine root system $\Delta(\Gamma)$ 
    immediately yields the equivalences 
    $\ref{enum:regular} \Leftrightarrow \ref{enum:regular_tau_periodic} 
    \Leftrightarrow \ref{enum:regular_real_Schur_root} 
    \Leftrightarrow \ref{enum:regular_zero_defect}$. 
    The remaining equivalence with $\ref{enum:regular_defect_semistable}$ 
    is proved at the end of this section.
\end{proof}

\begin{corollary} \label{cor:regular_hom_ortho}
    Let $W\in\mod H$ be regular $\tau$-rigid. For almost all $\lambda\in\bP^1$ holds
    \begin{align*}
        \Hom_H(W,V_\lambda) = 0 = \Hom_H(V_\lambda,W) 
        && \text{and} &&
        \Ext^1_H(W,V_\lambda) = 0 = \Ext^1_H(V_\lambda,W).
    \end{align*}
\end{corollary}
\begin{proof}
    We may assume that $W$ is indecomposable.
    Let $W'$ be the top of $W$ as a $\End_H(W)$-module.
    Recall that $W'$ is a brick with $W\in\filt(W')$ by Proposition \ref{prop:tau_rigid_structure}.
    Then $\partial( W') = 0$ 
    and $W$ is $\partial$-semistable 
    by Theorem \ref{prop:regular_tau_rigid}
    thus $W'$ is $\partial$-semistable.
    Also, any $\partial$-stable subfactor $S$ of $W$ 
    is already a subfactor of $W'$ 
    hence 
    \begin{align*}
        \dimv(S) \leq \dimv(W') \leq \frac{1}{q_\Gamma(\rkv(W))} \dimv(W) < D\bmeta.
    \end{align*}
    Now $S$ and $V_\lambda$ for almost all $\lambda\in\bP^1$ are $\partial$-stable
    with $\dimv S < \dimv V_{\lambda}$
    hence $S\not\cong V_\lambda$ 
    and thereby $\Hom_H(S,V_\lambda) = 0 = \Hom_H(V_\lambda,S)$.
    This shows $\Hom$-orthogonality of $W$ and $V_\lambda$.
    For $\Ext^1$-orthogonality simply apply Proposition \ref{prop:homological_ringel_coxeter}
    and use $\rkv(V_\lambda) = \bmeta$:
    \begin{align*}
        0 = \partial(W) 
        = \euler{\rkv(V_\lambda),\rkv(W)}_\Gamma 
        = \euler{V_\lambda,W}_H
        = - \dimext^1_H(V_\lambda,W).
    \end{align*}
\end{proof}

We write $V_{\lambda,k}^{(l)}$ for the unique (indecomposable) $\tau$-rigid $H$-module
with $\rkv(V_{\lambda,k}^{(l)}) = \bv_{\lambda,k}^{(l)}$
as in the preceding Proposition \ref{prop:regular_tau_rigid}\ref{enum:regular_real_Schur_root}.
In analogy with the terminology for modules over the species $\tilde{H}$,
we call $l$ the \emph{quasi-length} of $V_{\lambda,k}^{(l)}$,
and we say that $V_{\lambda,k}^{(l)}$ is \emph{quasi-simple} if $l = 1$.
First, we prove that quasi-simple regular $\tau$-rigid $H$-modules are $\partial$-semistable. 
To this end, we need a substantial generalization of the well known observation
that a brick $S\in\mod(A)$ over a finite-dimensional algebra $A$ 
is $\g(S)^\vee$-stable provided it is homogeneous i.e. $\tau_A(S)\cong S$; 
see for example \cite[Lemma 5]{CKW15}.

\begin{lemma} \label{lem:generalized_semibrick_semistability}
    Let $A$ be a finite-dimensional algebra. 
    Assume that $V_1,\dots,V_n \in\mod(A)$ enjoy the following properties:
    \begin{enumerate}[label = (\roman*)]
        \item \label{enum:generalized_brick}
        For every $1\leq i\leq n$ exist an $m_i\geq 1$ such that $\End_A(V_i)\cong K[\delta_i]/\ideal{\delta_i^{m_i}}$,
        \item \label{enum:hom_orthogonality}
        for every $1\leq i,j\leq n$ with $i\neq j$ is $\Hom_A(V_i,V_j) = 0$,
        \item \label{enum:tau_invariance}
        their direct sum $V:= \bigoplus_{i = 1}^n V_i$ satisfies $\tau_A(V) \cong V$.
    \end{enumerate}
    Then $V_i$ is $\g(V)^\vee$-semistable for every $1\leq i\leq n$,
    and $V_i$ is $\g(V)^\vee$-stable if $m_i = 1$.
\end{lemma}
\begin{proof}
    We show that $V$ is $\g(V)^\vee$-semistable.
    Let $U\subseteq V$ be a submodule.
    By Auslander-Reiten's $\g$-vector formula (\ref{ar_g_vector_formula})
    and assumption \ref{enum:tau_invariance}
    we find
    \begin{align*}
        \euler{\g(V),\dimv U}_A = \dimhom_A(V,U) - \dimhom_A(U,V).
    \end{align*}
    In particular, $\euler{\g(V),\dimv V}_A = 0$.
    Due to assumption \ref{enum:generalized_brick} we can choose generators $\delta_i \in \End_A(V_i)$ for $1\leq i\leq n$
    and consider the linear map
    \begin{align*}
        \cf_i\colon \End_A(V_i) \rightarrow K, &&
        \cf_i\left(\sum_{k=1}^{m_i} c_k \delta^{k-1} \right) := c_{m_i}.
    \end{align*}
    These assemble into a linear map $\cf\colon \End_A(V) \rightarrow K$ given by the ``trace":
    \begin{align*}
        \cf(f) := \sum_{i=1}^n \cf_i(f_{ii})
    \end{align*}
    where $f_{ji}$ for $1\leq i,j\leq n$ denotes the morphism $f$ 
    precomposed with the canonical inclusion $V_i \hookrightarrow V$ 
    and postcomposed with the canonical projection $V \twoheadrightarrow V_j$.
    With this at hand one defines a bilinear pairing
    \begin{align*}
        \Hom_A(U,V) \times \Hom_A(V,U) \to K, &&
       (g,f) \mapsto \cf(gf).
    \end{align*}
    Let $f\in\Hom_A(V,U)$ be non-zero.
    Then $\iota f \in \End_A(V)$ is still non-zero,
    where $\iota: U\hookrightarrow V$ is the canonical embedding.
    Let $1\leq i,j\leq n$ be such that $(\iota f)_{ji} \neq 0$.
    By assumption \ref{enum:hom_orthogonality} we must already have $i=j$.
    Now, one finds $g\in\End_A(V)$ such that $\cf(g\iota f) \neq 0$.
    This yields a linear embedding $\Hom_A(V,U) \hookrightarrow D\Hom_A(U,V)$
    and thus proves $\euler{\g(V),\dimv(U)}_A \leq 0$.

    Now $V_i$ is $\g(V)^\vee$-semistable as a direct summand of $V$. 
    For the final assertion, 
    assume $m_i = 1$ i.e. $\End_A(V_i)\cong K$. 
    That $V_i$ is even $\g(V)^\vee$-stable, 
    follows along the same line of argument 
    using that $0 = \dimhom_A(V_i,U) < \dimhom_A(U,V_i)$ 
    for every proper submodule $U\subset V_i$. 
\end{proof}

We already know the structure of endomorphisms of indecomposable $\tau$-rigid $H$-modules
from Proposition \ref{prop:tau_rigid_structure}.
Thus, to apply our Lemma \ref{lem:generalized_semibrick_semistability} 
we need to establish the vanishing of $\Hom$-spaces among sufficiently many regular $\tau$-rigid modules.
First, we show that quasi-simple regular $\tau$-rigid $H$-modules form a generalized semibrick.

\begin{lemma}\label{lem:quasi_simples_generalized_semibrick}
    Let $V,V'\in\mod(H)$ be quasi-simple regular $\tau$-rigid
    and set $c:=\euler{\rkv(V),\rkv(V')}_\Gamma$.
    Then
    \begin{align*}
        \Hom_H(V,V') \cong 
        \begin{cases}
            K[\delta]/\ideal{\delta^{c}} & \text{if $V\cong V'$,} \\
            0 & \text{else.}
        \end{cases}
    \end{align*}
\end{lemma}
\begin{proof}
    If $V \cong V'$, this is Proposition \ref{prop:tau_rigid_structure}.
    Thus assume for the remainder of the proof that $V\not\cong V'$.
    Write $V = V_{\lambda,k}^{(1)}$ and $V' = V_{\lambda',k'}^{(1)}$
    for some indices $\lambda,\lambda'\in I_\Gamma$, $k\in\bZ/r_\lambda$ and $k'\in\bZ/r_{\lambda'}$.
    The properties of reduction and localization \cite[Lemma 5.5, Corollary 5.15]{GLS20}
    yield the implication
    \begin{align} \label{zero_hom_implication}
        \Hom_{\tilde{H}}(\tilde{V}_{\lambda,k}^{(1)},\tilde{V}_{\lambda',k'}^{(1)}) = 0
        \text{ and }
        \Ext_H^1({V}_{\lambda,k}^{(1)},{V}_{\lambda',k'}^{(1)}) = 0
        ~ \Rightarrow ~
        \Hom_{{H}}({V}_{\lambda,k}^{(1)},{V}_{\lambda',k'}^{(1)}) = 0.
    \end{align}
    Here $\tilde{V}_{\lambda,k}^{(l)}$
    denotes the unique rigid $\tilde{H}$-module
    with $\dimv(\tilde{V}_{\lambda,k}) = \bv_{\lambda,k}^{(l)}$.
    We already have 
    $\Hom_{\tilde{H}}(\tilde{V}_{\lambda,k}^{(1)},\tilde{V}_{\lambda',k'}^{(1)}) = 0$ 
    due to the structure of $\mod(\tilde{H})$ determined in the Main Theorem of \cite{DR76}.
    According to (\ref{zero_hom_implication}),
    it remains to show $\Ext^1_H(V_{\lambda,k}^{(1)},V_{\lambda',k'}^{(1)}) = 0$.

    If $\lambda \neq \lambda'$, 
    then $\tilde{V}_{\lambda,k}^{(1)}\oplus \tilde{V}_{\lambda',k'}^{(1)}$ is rigid
    hence ${V}_{\lambda,k}^{(1)}\oplus {V}_{\lambda',k'}^{(1)}$ must be rigid as well 
    by Theorem \ref{thm:rigid_bijections}.
    If $\lambda = \lambda'$,
    we need to distinguish three further cases:

    \bigskip

    \underline{$k' \neq k\pm 1$}:
    Then $\tilde{V}_{\lambda,k}^{(1)}\oplus \tilde{V}_{\lambda,k'}^{(1)}$ is rigid in $\mod(\tilde{H})$
    thus ${V}_{\lambda,k}^{(1)}\oplus {V}_{\lambda,k'}^{(1)}$ is rigid in $\mod(H)$ by Theorem \ref{thm:rigid_bijections}.

    \bigskip

    \underline{$k' = k+1$}: 
    We have $V_{\lambda,k'}^{(1)} \cong \tau_H(V_{\lambda,k}^{(1)})$ 
    hence $\Hom_H(V_{\lambda,k}^{(1)},V_{\lambda,k'}^{(1)}) = 0$ 
    because $V_{\lambda,k}^{(1)}$ is $\tau$-rigid.
    Note that this already covers all cases if $r_\lambda = 2$.

    \bigskip

    \underline{$k' = k-1$}:
    If $r_\lambda = 4$, then $k \neq k-3$ and we already know
    \[
        \Ext^1_H({V}_{\lambda,k}^{(1)},{V}_{\lambda,k'}^{(1)}) 
        \cong \D\Hom_H({V}_{\lambda,k'}^{(1)},\tau_H({V}_{\lambda,k}^{(1)})) 
        = \D\Hom_H({V}_{\lambda,k-1}^{(1)},{V}_{\lambda,k+1}^{(1)})
        = 0.
    \] 
    For general $r_\lambda\geq 3$,
    there is an almost split sequence
    \begin{align*}
        0 \to \tilde{V}_{\lambda,k}^{(1)} 
        \to \tilde{V}_{\lambda,k}^{(2)} 
        \to \tilde{V}_{\lambda,k-1}^{(1)} \to 0
    \end{align*}
    with $\tilde{V}_{\lambda,k}^{(1)} \oplus \tilde{V}_{\lambda,k}^{(2)}$ rigid.
    Consider the co-Bongartz complement $\tilde{X}$ 
    of $\tilde{V}_{\lambda,k}^{(1)} \oplus \tilde{V}_{\lambda,k}^{(2)}$
    in $\mod(\tilde{H})$
    i.e. $\tilde{T}:= \tilde{V}_{\lambda,k}^{(1)} \oplus \tilde{V}_{\lambda,k}^{(2)} \oplus \tilde{X}$
    is a basic $\Ext^1$-projective generator 
    of the torsion class 
    $\tilde{\cT} := \fac(\tilde{V}_{\lambda,k}^{(1)} \oplus \tilde{V}_{\lambda,k}^{(2)})$.
    In particular, $\tilde{T}$ is support tilting.
    Observe that $\tilde{U}:=\tilde{V}_{\lambda,k}^{(2)} \oplus \tilde{V}_{\lambda,k-1}^{(1)} \oplus \tilde{X}$
    is still support tilting:
    First, $\tilde{V}_{\lambda,k-1}^{(1)}$ is not $\Ext^1$-projective in $\tilde{\cT}$
    hence not a summand of $\tilde{X}$.
    It is clear that 
    $\tilde{V}_{\lambda,k}^{(2)} \oplus \tilde{V}_{\lambda,k-1}^{(1)}$ 
    and $\tilde{V}_{\lambda,k}^{(2)} \oplus \tilde{X}$
    are rigid.
    We have $\Ext^1_{\tilde{H}}(\tilde{X},\tilde{V}_{\lambda,k-1}^{(1)}) = 0$
    because $\tilde{V}_{\lambda,k-1}^{(1)}\in \tilde{\cT}$ 
    and by choice of $\tilde{X}$.
    Finally, 
    $\Ext^1(\tilde{V}_{\lambda,k-1}^{(1)},\tilde{X}) 
    \cong \D\Hom_{\tilde{H}}(\tilde{X},\tilde{V}_{\lambda,k}^{(1)}) = 0$
    because $\tilde{X}$ must be preinjective
    and $\tilde{V}_{\lambda,k}^{(1)}$ is regular.
    This shows that there is an arrow $\tilde{T}\to\tilde{U}$ 
    in $\sttilt(\tilde{H})$
    with corresponding exchange pair 
    $(\tilde{V}_{\lambda,k}^{(1)},\tilde{V}_{\lambda,k-1}^{(1)})$.
    According to the isomorphism of exchange graphs 
    $\sttilt({H})\cong\sttilt(\tilde{H})$ 
    from Theorem \ref{thm:rigid_bijections},
    there is a corresponding arrow $T \to U$ in $\sttilt({H})$ for which 
    $({V}_{\lambda,k}^{(1)},{V}_{\lambda,k-1}^{(1)})$
    is an exchange pair.
    Therefore, we can conclude
    \begin{align*}
        \Ext_{H}^1({V}_{\lambda,k}^{(1)},{V}_{\lambda,k-1}^{(1)})
        \cong \D\Hom_H({V}_{\lambda,k-1}^{(1)},\tau_H({V}_{\lambda,k}^{(1)})) 
        = 0
    \end{align*}
    by \cite{AIR14}.
\end{proof}

To establish $\partial$-semistability 
of all regular $\tau$-rigid $H$-modules, 
it suffices to make sure that these are filtered by quasi-simples.
The following lemma 
shows that regular $\tau$-rigid $H$-modules organize in ``truncated pseudo tubes"
similar to the regular rigid $\tilde{H}$-modules.

\begin{lemma} \label{lem:trunc_pseudo_tube}
    Let $\lambda\in I_\Gamma$, $k,k'\in\bZ/r_\lambda$ and $1\leq l < r_\lambda$.
    There are short exact sequences
    \begin{align*}
        0 \to V_{\lambda,k}^{(1)} 
        \rightarrow V_{\lambda,k}^{(l)} 
        \rightarrow V_{\lambda,k-1}^{(l-1)} \to 0
        &&
        0 \to V_{\lambda,k}^{(l-1)} 
        \rightarrow V_{\lambda,k}^{(l)} 
        \rightarrow V_{\lambda,k-l+1}^{(1)} \to 0
    \end{align*}
\end{lemma}
\begin{proof}
    We want to make use of a geometric argument outlined in \cite[Section 5.4]{GLS16iii}. 
    Let us repeat it here for completeness:
    Consider two $\tau$-rigid $W,U\in\mod H$ with $\bw := \rkv W$, $\bu := \rkv U$ and set $\bv:= \bu + \bw$.
    Assume that $\Ext^1_H(U,W) = 0$.
    In \cite[Theorem 1.3.(iii)]{CBS02} it is proved 
    that the subset $\cE(W,U)\subseteq \Rep_\lf(H,\bv)$,
    consisting of representations $V$ fitting in a short exact sequence $0\to U \to V \to W\to 0$,
    is an open subset.
    On the other hand, if there is a $\tau$-rigid $V\in\Rep_\lf(H,\bv)$, 
    its orbit $\cO(V)\subseteq \Rep_{\lf}(H,\bv)$ is another open subset.
    But the subset $\Rep_\lf(H,\bv)$ is irreducible.
    Thus, we must have $\cE(W,U)\cap \cO(V) \neq \emptyset$.
    In other words, there is a short exact sequence
    \begin{align*}
        0 \to U \to V \to W \to 0.
    \end{align*}
    Thereby, it remains to show that 
    $0 = \Ext^1_H(V_{\lambda,k}^{(1)}, V_{\lambda,k-1}^{(l-1)})$
    for all $k\in\bZ/r_\lambda$ and $1 < l < r_\lambda$.
    By Auslander-Reiten duality we have
    \begin{align*}
    \Ext^1_H(V_{\lambda,k}^{(l-1)}, V_{\lambda,k-l+1}^{(1)})
    \cong \D\Hom_H(V_{\lambda,k-l}^{(1)},V_{\lambda,k}^{(l-1)})
    \end{align*}
    We prove more generally
    $\Hom_H(V_{\lambda,k'}^{(1)},V_{\lambda,k}^{(l-1)}) = 0$
    for all $k'\neq k$
    by induction on $2\leq l < r_\lambda$.
    Note that $k-l \neq k$ because $l < r_\lambda$.

    If $l = 2$, this is
    Lemma \ref{lem:quasi_simples_generalized_semibrick}.
    If $l\geq 3$ we may assume by induction, that
    $\Hom_H(V_{\lambda,k'}^{(1)},V_{\lambda,k}^{(l-2)}) = 0$
    for all $k'\neq k$.
    In particular,
    for $k' = k-l+1$,
    the initial discussion shows that
    there exists a short exact sequence
    \begin{align*}
        0 \to V_{\lambda,k}^{(l-2)} 
        \rightarrow V_{\lambda,k}^{(l-1)} 
        \rightarrow V_{\lambda,k-l+2}^{(1)} \to 0.
    \end{align*}
    Lemma \ref{lem:quasi_simples_generalized_semibrick}
    shows $\Hom_H(V_{k'}^{(1)},V_{k-l+2}^{(1)}) = 0$
    for $k' \neq k-l+2$,
    while $\Hom_H(V_{k'}^{(1)},V_{\lambda,k}^{(l-2)}) = 0$
    for $k'\neq k$
    by our induction hypothesis. 
    On the other hand,
    $\Hom_H(V_{\lambda,k-l+2}^{(1)}, V_{\lambda,k}^{(l-1)}) = 0$
    by a similar application of the implication \ref{zero_hom_implication}
    in the proof of Lemma \ref{lem:quasi_simples_generalized_semibrick}.
    Therefore, 
    $\Hom_H(V_{\lambda,k'}^{(1)},V_{\lambda,k}^{(l-1)}) = 0$
    for all $k'\neq k$.
    The existence of the other short exact sequence is proved dually.
\end{proof}

\begin{proof}[Proof of Proposition \ref{prop:regular_tau_rigid}]
    We are left to show the implication ``$\ref{enum:regular},\ref{enum:regular_tau_periodic},\ref{enum:regular_zero_defect}\Rightarrow\ref{enum:defect_semistable}$''.
    Fix an exceptional index $\lambda\in I_\Gamma$.
    The quasi-simple regular $\tau$-rigid $H$-modules 
    $V_{\lambda,k}^{(1)}$ for $k \in \bZ/r_\lambda$,
    and their direct sum $V:=\bigoplus_{k\in\bZ/r_\lambda}V_{\lambda,k}^{(1)}$
    satisfy the conditions of Lemma \ref{lem:generalized_semibrick_semistability}
    by Lemma \ref{lem:quasi_simples_generalized_semibrick} 
    and part \ref{enum:regular_tau_periodic}.
    Moreover $\rkv(V) = t_\lambda \bmeta$ 
    hence $\g(V) = t_\lambda\partial$.
    Now Lemma \ref{lem:generalized_semibrick_semistability} allows to conclude 
    that $V_{\lambda,k}^{(1)}$ is $\partial$-semistable for every $k\in\bZ/r_\lambda$.
    
    We know from Lemma \ref{lem:trunc_pseudo_tube} 
    that every other regular $\tau$-rigid $W\in\mod H$
    is filtered by quasi-simple regular $\tau$-rigid $H$-modules
    thus $W$ is $\partial$-semistable as well.
\end{proof}

\subsection{Generic decomposition of locally free components} \label{sec:generic_classification}

The missing ingredient for our proof of the Main Theorem, 
is a generalization of Kac's canonical decomposition 
for rank vectors.
The Galois descent of rigid modules 
established in Theorem \ref{thm:rigid_bijections}
allows to ``fold" Kac's canonical decomposition
of dimension vectors over the affine unfolded quiver.

\begin{lemma} \label{lem:folded_kac_decomposition}
    Let $\Gamma$ be an affine valued quiver.
    For every $\bv\in \bN\Gamma_0$ there exist unique $m\geq 0$ and $\bw\in\bN\Gamma_0$ with 
    $\bv = m \bmeta + \bw$
    such that
    \begin{enumerate}[label = (\roman*)]
        \item there exists a $\tau$-rigid $W\in\mod {H}$ with $\bw = \rkv W$,
        \item if $m \neq 0$, then $W$ is regular 
        i.e. every indecomposable summand of $W$ is regular.
    \end{enumerate}
\end{lemma}
\begin{proof}
    By Theorem \ref{thm:rigid_bijections} and \ref{prop:regular_tau_rigid}
    it suffices to find a rigid $W\in\mod(\tilde{H})$ with $\bw = \rkv(W)$
    which is moreover regular provided $m \neq 0$.
    Consider scalar extension $\pi^*\colon \mod(\tilde{H}) \to \mod(\bar{H})$
    where $\bar{H} \cong \bar{L}\bar{Q}$ for the unfolded quiver $\bar{Q}$.
    Since $\pi^*$ induces an isometry $\Pi_0^*\colon \bZ\Gamma_0 \to \bZ\bar{Q}_0$
    of Grothendieck groups by Corollary \ref{cor:scalar_extension_isometry},
    we get that $\bar{Q}$ is of affine type as well
    with primitive null root $\bar{\bmeta} = \Pi^*_0(\bmeta)$.
    Consider Kac's canonical decomposition of $\Pi^*_0(\bv) \in \bN\bar{Q}_0$:
    \begin{align*}
        \Pi^*_0(\bv) = m \bar{\bmeta} + \bar{\bw}
    \end{align*} 
    for some $m\geq 0$ 
    and some rigid $\bar{W}\in\mod(\bar{H})$ 
    with $\bar{\bw} = \dimv \bar{W}$. 
    Since $\Pi^*_0(\bv)$ and $\bar{\bmeta}$ are $G$-invariant
    we find that $\bar{\bw}$ is $G$-invariant as well.
    Thus, for every $g\in \Gal(\bar{L}/L)$ is $\dimv(g(\bar{W})) = \dimv(\bar{W})$
    and $g(\bar{W})$ is still rigid 
    hence $g(\bar{W})\cong\bar{W}$
    because rigid $K\bar{Q}$-modules are determined by their dimension vector.
    According to Theorem \ref{thm:rigid_bijections} 
    there exists a rigid $W\in\mod(\tilde{H})$
    with $\bar{W} = \pi^*(W)$.
    In particular, $\bar{\bw} = \Pi^*_0(\bw)$
    for $\bw = \dimv(W)$
    thus $\Pi^*_0(\bv) = \Pi^*_0(m\bmeta + \bw)$.
    But $\Pi^*_0$ is injective hence $\bv = m \bmeta + \bw$.

    Suppose $X$ is a preprojective summand of $\bar{W}$
    then $\euler{\bar{\bmeta}, \dimv(X)}_{\bar{Q}} < 0$
    hence $\Ext^1_{\bar{H}}(V,X) \neq 0$ for every $V\in\mod \bar{H}$ with $\dimv(V) = \bar{\bmeta}$.
    This contradicts the properties of Kac's canonical decomposition.
    Similarly, $\bar{W}$ cannot have a preinjective summand.
    But then $W$ must be regular as well:
    Indeed, let $X$ be an indecomposable direct summand of $W$.
    Assume that $X$ is preprojective.
    Then $\bar{X}:=\pi^*(X)$ is a direct summand of $\bar{W}$ 
    with $\euler{\bar{\bmeta},\dimv(\bar{X})}_{\bar{Q}} = \euler{\bmeta,\dimv(X)}_\Gamma > 0$
    by Corollary \ref{cor:scalar_extension_isometry}.
    Hence $\bar{X}$ has a preprojective summand
    which contradicts regularity of $\bar{W}$.
    Similarly, $X$ cannot be preinjective.
\end{proof}

\begin{theorem}\label{thm:generic_classification}
    Let $\Gamma$ be a valued quiver with minimal symmetrizer.
    Set $H:= H(\Gamma)$ and 
    let $\bv\in\bN\Gamma_0$ be a rank vector.
    Then there exist $m\geq 0$ and $\bw\in \bZ\Gamma_0$
    such that 
    \begin{align} \label{canonical_decomposition}
        \Rep_\lf(H,\bv) = \overline{\Rep_\lf(H,\bmeta_\Gamma)^{m} \oplus \Rep_\lf(H,\bw)}
    \end{align}
    and
    \begin{align} \label{dense_family_and_orbit}
        \Rep_\lf(H,\bmeta_\Gamma) = \overline{\bigcup_{\lambda\in\bP^1}\cO(V_\lambda)}
        &&
        \Rep_\lf(H,\bw) = \overline{\cO(W)}
    \end{align}
    for some $\tau$-rigid $W\in\mod(H)$
    and the $\bP^1$-family of $\partial_\Gamma$-semistable $V_\lambda\in\mod(H)$ from Theorem \ref{thm:null_family}.
\end{theorem}
\begin{proof}
    Write $\bv = m\bmeta + \bw$ as in Lemma \ref{lem:folded_kac_decomposition}.
    Recall that $\dim \Rep_\lf(H,\bmeta_\Gamma) = \dim\GL(D\bmeta_\Gamma)$
    and $V_\lambda \in \Rep_\lf(H,\bmeta_\Gamma)$ is a $1$-parameter family of
    pairwise non-isomorphic bricks, 
    hence the union of their orbits is dense in $\Rep_\lf(H,\bmeta_\Gamma)$.
    On the other hand, 
    we already know that 
    every $\tau$-rigid $H$-module $W$ is locally free
    and its orbit is open hence dense in the irreducible $\Rep_\lf(H,\bw)$
    for $\bw=\rkv(W)$.
    This proves (\ref{dense_family_and_orbit})
    and if $m=0$, we are thereby done.
    If $m\neq 0$, then $W$ is regular by Lemma \ref{lem:folded_kac_decomposition}.
    By Corollary \ref{cor:regular_hom_ortho} we have
    and $\Ext^1_H(W,V_\lambda) = 0$ and $\Ext^1_H(V_\lambda,W) = 0$.
    Therefore we have settled (\ref{canonical_decomposition}) by \cite[Theorem 1.2]{CBS02}.
\end{proof}

\subsection{$\tau$-tilting tameness} \label{sec:tau_tame}

Having Kac's canonical decomposition from Lemma \ref{lem:folded_kac_decomposition}
and the defect $\partial\in \K_0(H)^*$,
we can follow almost verbatim the proof of $\g$-tameness of affine path algebras
given in \cite[Theorem 5.1.4]{H06}.

\begin{corollary} \label{cor:g_tame}
    Let $H$ be an affine GLS algebra.
    Then $H$ is $\g$-tame
    i.e. the $\g$-vector fan $\fan({H})$ is dense in $\K_0^\fin(H)_\bR$.
\end{corollary}
\begin{proof}
    Recall that the canonical inclusion $\mod_\lf(H)\hookrightarrow\mod(H)$
    induces an embedding $\K_0^\fin(H) = \K_0^\lf(H) \hookrightarrow \K_0(H)$
    which is an isomorphism over $\bR$
    because the Cartan matrix of $H$ is invertible.
    And any $\tau$-rigid module $V\in\mod(H)$
    is locally free i.e. $\pdim_H(V)\leq 1$
    hence its $\g$-vector $\g(V)$ coincides along the above mentioned embedding 
    with its class in $\K_0(H)$.
    Let $v\in\K_0(H)^+$, 
    we may write $\bv = m\bmeta + \bw$ 
    as in Lemma \ref{lem:folded_kac_decomposition}.
    Suppose $\bv\not\in\fan(H)$, then $m\neq 0$ thus $\partial(\bv) = \partial(\bw) = 0$
    because $\bw = \rkv(W)$ for a regular $W$.
    Therefore $\K_0^\fin(H)_\bR \setminus \fan({H})$ lies in the kernel of the defect $\partial$,
    hence $\fan(H)$ is dense in $\K_0^\fin(H)$.
\end{proof}

\begin{corollary} \label{cor:tau_reduced_tame}
    Let $H$ be an affine GLS algebra.
    If the symmetrizer is minimal,
    then $H$ is generically $\tau$-reduced tame.
    For general symmetrizers is $H$ $E$-tame.
\end{corollary}
\begin{proof}
    We may assume that the symmetrizer is minimal by \cite[Equivalence (4.2)]{EJR18}.
    It suffices to show $c_H(\cZ(\br)) \leq 1$
    for every rank vector $\br\in\bN\Gamma_0$ such that
    the generically $\tau$-reduced component $\cZ(\br)$
    with generic rank vector $r$
    is generically indecomposable.
    According to Theorem \ref{thm:generic_classification},
    we only need to consider $\br = \bmeta$ or $\br = \rkv(W)$ 
    for $W\in\mod(H)$ indecomposable $\tau$-rigid.
    In the latter case,
    $\cZ(\br) = \overline{\cO(W)}$
    hence $c_H(\cZ(\br)) = 0$.
    If $\br = \bmeta$
    we have seen in Theorem \ref{thm:generic_classification} that 
    the generic element of $\cZ(\br)$ is isomorphic to 
    $V_\lambda$ for some $\lambda \in \bA^1$ as in Theorem \ref{thm:null_family}.
    Therefore,
    $c_H(\cZ(\br)) = \dimhom_H(V_{\lambda}, \tau_H(V_\lambda)) 
    = \dimext^1_H(V_\lambda,V_\lambda) = \dimend_H(V_\lambda) = 1$
    because $\cZ(\br)$ is $\tau$-reduced, 
    $\pdim(V_\lambda)\leq 1$ and 
    $\euler{V_\lambda,V_\lambda}_H = q_\Gamma(\bmeta) = 0$.
\end{proof}

\subsection*{Acknowledgements}
This work is part of my Ph.D. thesis.
I am deeply indebted to Syddansk Universitet (SDU)
as well as to my co-supervisors Prof. Christof Geiß and Dr. Fabian Haiden
for making this project possible.
Special thanks go to Prof. Christof Geiß 
for his invaluable continuous support
and for sharing his recent work with me.
I am also grateful to Dr. Fabian Haiden for his helpful advice 
and many stimulating ideas.
Last but not least, 
I would like to thank Dr. Hipolito Treffinger
for interesting discussions
and his constant encouragement
during my research visits at the Université Paris Cité.

This paper is partly a result of the 
ERC-SyG project, Recursive and Exact New Quantum Theory (ReNewQuantum) 
which received funding from the European Research Council (ERC) 
under the European Union's Horizon 2020 research and innovation programme 
under grant agreement No 810573,
held by Prof. Jørgen Ellegaard Andersen
to whom I would like to express my sincere gratitude 
for giving me the opportunity to pursue my Ph.D. studies at SDU.

\bibliographystyle{amsalpha}
\bibliography{bibliography.bib}

\end{document}